\documentclass[11pt]{amsart}
\usepackage{amsmath,amstext,amsbsy,amssymb,amscd}
\usepackage{amsmath}
\usepackage{amsxtra}
\usepackage{amscd}
\usepackage{amsthm}
\usepackage{amsfonts}
\usepackage{amssymb}
\usepackage{eucal}
\usepackage{color}
\usepackage[enableskew,vcentermath]{youngtab}

\newtheorem{thm}{Theorem}[section]
\theoremstyle{plain}

\newtheorem{lem}[thm]{Lemma}
\newtheorem{prop}[thm]{Proposition}
\newtheorem{cor}[thm]{Corollary}

\theoremstyle{definition}
\newtheorem{defn}[thm]{Definition}
\newtheorem{example}[thm]{Example}

\theoremstyle{remark}
\newtheorem{rem}[thm]{Remark}
\newtheorem{conjecture}[thm]{Conjecture}

\definecolor{A}{rgb}{.75,1,.75}

\numberwithin{equation}{section}

\newcommand{\ds}{\displaystyle}

\newcommand{\I}{\mathbb{I}}

\newcommand{\var}{\varepsilon}

\newcommand{\F}{\mathbb F}
\newcommand{\mhcn}{\mathfrak{H}_n^{\mathfrak{c}}} 
\newcommand{\mpcn}{\mathcal{P}_n^{\mathfrak{c}}} 
\newcommand{\RImhcn}{\operatorname{Rep}_{\mathbb{I}}\mhcn} 

{\vskip-\lastskip\medskip
  \noindent
  {\em #1.}\enspace
  }%
{\qed\par\medskip
  }

\begin{document}

\title[Affine Hecke-Clifford algebras]{Completely splittable representations
of affine Hecke-Clifford algebras}
\author{Jinkui Wan}
\address{Department of Mathematics, University of Virginia,
Charlottesville,VA 22904, USA.}\email{jw5ar@virginia.edu}

\begin{abstract} We classify and construct irreducible completely splittable
representations of affine and finite Hecke-Clifford algebras over an
algebraically closed field of characteristic  not equal to $2$.
\end{abstract}
\maketitle

 \setcounter{tocdepth}{1}
\tableofcontents
\section{Introduction}
Let $\F$ be an algebraically closed field of characteristic $p$ and
denote by $S_n$ the symmetric group on $n$ letters. In~\cite{M},
Mathieu gave the dimension of the irreducible $\F S_n$-modules
associated to the partitions $\lambda=(\lambda_1,\ldots,\lambda_l)$
of $n$ with length $l$ and $\lambda_1-\lambda_l\leq (l-p)$ by using
the well-known Schur-Weyl duality. Subsequently, Kleshchev~\cite{K1}
showed that these representations are exactly these whose
restrictions to the subgroup $S_k$ are semi-simple for any $k\leq n$
or equivalently on which the Jucys-Murphy elements in $\F S_n$ act
semisimply. These $\F S_n$-modules are called {\em completely
splittable} in~\cite{K1}. By using the modular branching rules for
$\F S_n$ (cf. \cite{K2}), a formula for the dimensions of completely
splittable modules was obtained in terms of the paths in Young
modular graphs, which recovers Mathieu's result~\cite{M}.
Generalizing the work in~\cite{K1, M}, Ruff~\cite{Ru} formulated and
classified the irreducible completely splittable representations of
degenerate affine Hecke algebras $\mathcal{H}_n$ (introduced by
Drinfeld \cite{Dr} and Lusztig \cite{Lus}). Over the complex field
$\mathbb{C}$, these $\mathcal{H}_n$-modules were constructed and
classified originally by Cherednik~\cite{C1}. Generalizations were
established to affine Hecke algebras of type A in \cite{C2} and
Ram~\cite{Ra} and to Khovanov-Lauda-Rouquier algebras in~\cite{KR}.

From now on let us assume $p\neq 2$. This paper aims to classify and
construct completely splittable representations of affine
Hecke-Clifford algebras $\mathfrak{H}^{\mathfrak{c}}_n$ over
$\mathbb{F}$. The algebra $\mathfrak{H}^{\mathfrak{c}}_n$ was
introduced by Nazarov~\cite{Na2}(called affine Sergeev algebra) to
study the spin (or projective) representations of the symmetric
group $S_n$ or equivalently to study the representations of the spin
symmetric group algebra $\F S_n^-$. Our construction is a
generalization of Young's seminormal construction of the irreducible
representations of symmetric groups and affine Hecke algebras of
type A (cf. \cite{C2, Ra}). The approach is similar in spirit to the
technique introduced by Okounkov and Vershik~\cite{OV} on symmetric
groups over $\mathbb{C}$.

Let us denote by $x_1,\ldots, x_n$ the polynomial generators of the
algebra $\mhcn$ (cf. subsection~\ref{basicofAHCa} for the
definition). According to Brundan and Kleshchev~\cite{BK} (cf.
\cite[Part II]{K2}), one can reduce the study of the
 finite dimensional $\mhcn$-modules to these so-called integral modules on
which the eigenvalues of $x_1^2,\ldots,x_n^2$ are of the form $q(i)$
for $i\in\I$ (cf. (\ref{defn:I}) and~(\ref{qi}) for notations). Then
each finite dimensional $\mhcn$-module $M$ admits a decomposition as
$ M=\oplus_{\underline{i}\in\I^n}M_{\underline{i}},$ where
$M_{\underline{i}}$ is the simultaneous generalized eigenspace for
the commuting operators $x_1^2,\ldots,x_n^2$ corresponding to the
eigenvalues $q(i_1),\ldots,q(i_n)$. We call $\underline{i}$ a weight
of $M$ if $M_{\underline{i}}\neq 0$. By definition, a finite
dimensional $\mhcn$-module is completely splittable if the
polynomial generators $x_1,\ldots,x_n$ act semisimply.

Our work is based on several equivalent characterizations (cf.
Proposition~\ref{prop:equiv.cond.} for precise statements) of
irreducible completely splittable $\mhcn$-modules. In particular, an
irreducible $\mhcn$-module is completely splittable if and only if
its restriction to the subalgebra
$\mathfrak{H}^{\mathfrak{c}}_{(r,1^{n-r})}$ (cf.
subsection~\ref{basicofAHCa} for notations) is semisimple for any
$1\leq r\leq n$. It follows that any irreducible completely
splittable $\mhcn$-module is semisimple on restriction to the
subalgebra of $\mhcn$ generated by $s_k, c_k,c_{k+1},x_k,x_{k+1}$
(cf. subsection~\ref{basicofAHCa} for notations) which is isomorphic
to $\mathfrak{H}_2^{\mathfrak{c}}$ for fixed $1\leq k\leq n-1$. By
exploring irreducible $\mathfrak{H}_2^{\mathfrak{c}}$-modules, we
obtain an explicit description of the action of the simple
transpositions $s_k$ on irreducible completely splittable
$\mhcn$-modules and identify all possible weights of irreducible
completely splittable $\mhcn$-modules.  This leads to the
construction of a family of irreducible completely splittable
$\mhcn$-modules. It turns out that these modules exhaust the
non-isomorphic irreducible completely splittable $\mhcn$-modules. We
further show that these representations are parameterized by skew
shifted Young diagrams with precise constraints depending on $p$ and
give a dimension formula in terms of the associated standard Young
tableaux. We remark that in the special case when $p=0$, our result
confirms a conjecture of Wang and it has been independently obtained
by Hill, Kujawa, and Sussan~\cite{HKS}.

Denote by $\mathcal{Y}_n$ the finite Hecke-Clifford algebra
$\mathcal{Y}_n=\mathcal{C}_n\rtimes\F S_n$, where $\mathcal{C}_n$ is
the Clifford algebra over $\F$ generated by $c_1,\ldots, c_n$
subject to the relations $c_k^2=1, c_kc_l=-c_lc_k$ for $1\leq k\neq
l\leq n$. A $\mathcal{Y}_n$-module is called completely splittable
if the Jucys-Murphy elements $L_1,\ldots, L_n$ (cf. (\ref{JM}) for
notations) act semisimply. There exists a surjective homomorphism
(cf. \cite{Na2}) from $\mhcn$ to $\mathcal{Y}_n$ which maps $x_k$ to
the Jucys-Murphy elements $L_k$ for $1\leq k\leq n$. By applying the
results established for $\mhcn$ to $\mathcal{Y}_n$, we classify
irreducible completely splittable $\mathcal{Y}_n$-modules and obtain
a dimension formula for these modules. We understand that an
unpublished work of Kleshchev and Ruff independently gave the
classification of irreducible completely splittable
$\mathcal{Y}_n$-modules. In~\cite{BK}, irreducible representations
of $\mathcal{Y}_n$ over $\F$ are shown to be parameterized by
$p$-restricted $p$-strict partitions of $n$. In this paper, we
identify the subset $\Gamma$ of $p$-restricted $p$-strict partitions
of $n$ which parameterizes irreducible completely splittable
$\mathcal{Y}_n$-modules. This together with a well-known Morita
super-equivalence between the spin symmetric group algebra $\F
S_n^-$ and $\mathcal{Y}_n$ leads to an interesting family of
irreducible $\F S_n^-$-modules parameterized by $\Gamma$ for which
dimensions and characters can be explicitly described. In the
special case when $p=0$, we recover the main result of~\cite{Na1} on
the seminormal construction of all simple representations of $\F
S_n^{-}$.

We observe that the $L_k^2, 1\leq k\leq n, $ act semsimply on the
basic spin $\mathcal{Y}_n$-module $I(n)$ (cf. \cite[(9.11)]{BK})
which is not completely splittable.  On the other hand,
Wang~\cite{W} introduced the degenerate spin affine Hecke-Clifford
algebras $\mathfrak{H}^-$ and established an isomorphism between
$\mhcn$ and $\mathcal{C}_n\otimes \mathfrak{H}^{-}$ which sends
$x_k^2$ to $2b_k^2$ (cf. Section~\ref{aLc} for notations). As the
generators $b_1,\ldots, b_n$ are anti-commutative, it is reasonable
to study the $\mathfrak{H}^{-}$-modules on which the commuting
operators $b_1^2,\ldots,b_n^2$ act semisimply. This is equivalent to
studying $\mhcn$-modules on which $x_k^2, 1\leq k\leq n, $ act
semisimply by using the isomorphism between $\mhcn$ and
$\mathcal{C}_n\otimes \mathfrak{H}^{-}$. Motivated by these
observations, we study and obtain a necessary condition in terms of
weights for the classification of irreducible $\mhcn$-modules on
which $x_k^2$, $1\leq k\leq n$, act semisimply; moreover, this
condition is conjectured to be sufficient, and the conjecture is
verified when $n=2, 3$.

The paper is organized as follows. In Section~\ref{AHCa}, we recall
some basics about superalgebra and also the affine Hecke-Clifford
algebras $\mhcn$. In Section~\ref{wofCS}, we analyze the structure
of completely splittable $\mhcn$-modules by studying their weights
and a classification of irreducible completely splittable
$\mhcn$-modules is obtained in Section~\ref{classification}.  In
Section~\ref{combinatorics}, we give a reinterpretation for weights
of irreducible completely splittable $\mhcn$-modules in terms of
shifted Young diagrams. In Section~\ref{fHCa}, we classify the
irreducible completely splittable representations of finite
Hecke-Clifford algebras. Finally, in Section~\ref{aLc} we introduce
a larger category consisting of $\mhcn$-modules on which $x_k^2$ act
semisimply and state a conjecture for classification of modules in
this larger category.

{\bf Acknowledgments.}  I thank A. Kleshchev and especially my
advisor W. Wang for many helpful suggestions and discussions. I
would also like to thank the referees for their useful comments.
This research is partly supported by Wang's NSF grant.

\section{Affine Hecke-Clifford algebras $\mhcn$}\label{AHCa}

Recall that $\F$ is an algebraically closed field of characteristic
$p$ with $p\neq 2$. Denote by $\mathbb{Z}_+$ the set of nonnegative
integers and let
\begin{align}
\mathbb{I}=\left\{
\begin{array}{ll}
\mathbb{Z}_+, & \text{ if } p=0, \\
\{0,1,\ldots,\frac{p-1}{2}\}, &\text{ if } p\geq 3.
\end{array}
\right.\label{defn:I}
\end{align}
\subsection{Some basics about superalgebras}
We shall recall some basic notions of superalgebras, referring the
reader to~\cite[\S 2-b]{BK}. Let us denote by
$\bar{v}\in\mathbb{Z}_2$ the parity of a homogeneous vector $v$ of a
vector superspace. By a superalgebra, we mean a
$\mathbb{Z}_2$-graded associative algebra. Let $\mathcal{A}$ be a
superalgebra. A $\mathcal{A}$-module means a $\mathbb{Z}_2$-graded
left $\mathcal{A}$-module. A homomorphism $f:V\rightarrow W$ of
$\mathcal{A}$-modules $V$ and $W$ means a linear map such that $
f(av)=(-1)^{\bar{f}\bar{a}}af(v).$  Note that this and other such
expressions only make sense for homogeneous $a, f$ and the meaning
for arbitrary elements is to be obtained by extending linearly from
the homogeneous case.  Let $V$ be a finite dimensional
$\mathcal{A}$-module. Let $\Pi
 V$ be the same underlying vector space but with the opposite
 $\mathbb{Z}_2$-grading. The new action of $a\in\mathcal{A}$ on $v\in\Pi
 V$ is defined in terms of the old action by $a\cdot
 v:=(-1)^{\bar{a}}av$. Note that the identity map on $V$ defines
 an isomorphism from $V$ to $\Pi V$.

A superalgebra analog of Schur's Lemma states that the endomorphism
algebra of a finite dimensional irreducible module over a
superalgebra is either one dimensional or two dimensional. In the
former case, we call the module of {\em type }\texttt{M} while in
the latter case the module is called of {\em type }\texttt{Q}.


Given two superalgebras $\mathcal{A}$ and $\mathcal{B}$, we view
the tensor product of superspaces $\mathcal{A}\otimes\mathcal{B}$
as a superalgebra with multiplication defined by
$$
(a\otimes b)(a'\otimes b')=(-1)^{\bar{b}\bar{a'}}(aa')\otimes (bb')
\qquad (a,a'\in\mathcal{A}, b,b'\in\mathcal{B}).
$$
Suppose $V$ is an $\mathcal{A}$-module and $W$ is a
$\mathcal{B}$-module. Then $V\otimes W$ affords $A\otimes B$-module
denoted by $V\boxtimes W$ via
$$
(a\otimes b)(v\otimes w)=(-1)^{\bar{b}\bar{v}}av\otimes bw,~a\in A,
b\in B, v\in V, w\in W.
$$
If $V$ is an irreducible $\mathcal{A}$-module and $W$ is an
irreducible $\mathcal{B}$-module, $V\boxtimes W$ may not be
irreducible. Indeed, we have the following standard lemma (cf.
\cite[Lemma 12.2.13]{K1}).
\begin{lem}\label{tensorsmod}
Let $V$ be an irreducible $\mathcal{A}$-module and $W$ be an
irreducible $\mathcal{B}$-module.
\begin{enumerate}
\item If both $V$ and $W$ are of type $\texttt{M}$, then
$V\boxtimes W$ is an irreducible
$\mathcal{A}\otimes\mathcal{B}$-module of type $\texttt{M}$.

\item If one of $V$ or $W$ is of type $\texttt{M}$ and the other
is of type $\texttt{Q}$, then $V\boxtimes W$ is an irreducible
$\mathcal{A}\otimes\mathcal{B}$-module of type $\texttt{Q}$.

\item If both $V$ and $W$ are of type $\texttt{Q}$, then
$V\boxtimes W\cong X\oplus \Pi X$ for a type $\texttt{M}$
irreducible $\mathcal{A}\otimes\mathcal{B}$-module $X$.
\end{enumerate}
Moreover, all irreducible $\mathcal{A}\otimes\mathcal{B}$-modules
arise as constituents of $V\boxtimes W$ for some choice of
irreducibles $V,W$.
\end{lem}
If $V$ is an irreducible $\mathcal{A}$-module and $W$ is an
irreducible $\mathcal{B}$-module, denote by $V\circledast W$ an
irreducible component of $V\boxtimes W$. Thus,
$$
V\boxtimes W=\left\{
\begin{array}{ll}
V\circledast W\oplus \Pi (V\circledast W), & \text{ if both } V \text{ and } W
 \text{ are of type }\texttt{Q}, \\
V\circledast W, &\text{ otherwise }.
\end{array}
\right.
$$
\subsection{Affine Hecke-Clifford algebras $\mhcn$}\label{basicofAHCa}
Now we proceed to define the superalgebra we will be interested in.
For $n\in\mathbb{Z}_+$, the affine Hecke-Clifford algebra $\mhcn$ is
the superalgebra generated by even generators
$s_1,\ldots,s_{n-1},x_1,\ldots,x_n$ and odd generators
$c_1,\ldots,c_n$ subject to the following relations
\begin{align}
s_i^2=1,\quad s_is_j =s_js_i, \quad
s_is_{i+1}s_i&=s_{i+1}s_is_{i+1}, \quad|i-j|>1,\label{braid}\\
x_ix_j&=x_jx_i, \quad 1\leq i,j\leq n, \label{poly}\\
c_i^2=1,c_ic_j&=-c_jc_i, \quad 1\leq i\neq j\leq n, \label{clifford}\\
s_ix_i&=x_{i+1}s_i-(1+c_ic_{i+1}),\label{px1}\\
s_ix_j&=x_js_i, \quad j\neq i, i+1, \label{px2}\\
s_ic_i=c_{i+1}s_i, s_ic_{i+1}&=c_is_i,s_ic_j=c_js_i,\quad j\neq i, i+1, \label{pc}\\
x_ic_i=-c_ix_i, x_ic_j&=c_jx_i,\quad 1\leq i\neq j\leq n.
\label{xc}
\end{align}
\begin{rem} The affine Hecke-Clifford algebra $\mhcn$ was introduced by
Nazarov~\cite{Na2}(called affine Sergeev algebra) to study the
representations of $\mathbb{C}S_n^-$.  The quantized version of the
$\mhcn$ introduced later by Jones-Nazarov~\cite{JN} to study the
$q$-analogues of Young symmetrizers for projective representations
of the symmetric group $S_n$ is often also called affine
Hecke-Clifford algebras.
\end{rem}

For $\alpha=(\alpha_1,\ldots,\alpha_n)\in\mathbb{Z}_+^n$ and
$\beta=(\beta_1,\ldots,\beta_n)\in\mathbb{Z}_2^n$, set
$x^{\alpha}=x_1^{\alpha_1}\cdots x_n^{\alpha}$ and
$c^{\beta}=c_1^{\beta_1}\cdots c_n^{\beta_n}$. Then we have the
following.
\begin{lem}\cite[Theorem 2.2]{BK}\label{lem:PBW}
The set $\{x^{\alpha}c^{\beta}w~|~ \alpha\in\mathbb{Z}_+^n,
\beta\in\mathbb{Z}_2^n, w\in S_n\}$ forms a basis of $\mhcn$.
\end{lem}
Denote by $\mpcn$ the superalgebra generated by even generators
$x_1,\ldots,x_n$ and odd generators $c_1,\ldots,c_n$ subject to the
relations~(\ref{poly}),~(\ref{clifford}) and~(\ref{xc}). By
Lemma~\ref{lem:PBW}, $\mpcn$ can be identified with the subalgebra
of $\mhcn$ generated by $x_1,\ldots,x_n$ and $c_1,\ldots,c_n$. For a
composition $\mu=(\mu_1,\mu_2,\ldots,\mu_r)$ of $n$, we define
$\mathfrak{H}_{\mu}^{\mathfrak{c}}$ to be the subalgebra of $\mhcn$
generated by $\mpcn$ and $s_j\in S_{\mu}=S_{\mu_1}\times\cdots
\times S_{\mu_r}$. Note that
$\mpcn=\mathfrak{H}_{(1^n)}^{\mathfrak{c}}$. For each
$i\in\mathbb{I}$, set
\begin{align}
q(i)=i(i+1).\label{qi}
\end{align}
Let us denote by
$\operatorname{Rep}_{\I}\mathfrak{H}_{\mu}^{\mathfrak{c}}$ the
category of so-called \emph{integral} finite dimensional
$\mathfrak{H}_{\mu}^{\mathfrak{c}}$-modules on which the
$x^2_1,\ldots, x^2_n$ have eigenvalues of the form $q(i)$ for
$i\in\I$. For each $i\in\I$, denote by $L(i)$ the $2$-dimensional
$\mathcal{P}_1^{\mathfrak{c}}$-module with
$L(i)_{\bar{0}}=\mathbb{F}v_0$ and $L(i)_{\bar{1}}=\mathbb{F}v_1$
and
$$
x_1v_0=\sqrt{q(i)}v_0,\quad x_1v_1=-\sqrt{q(i)}v_1, \quad
c_1v_0=v_1,\quad c_1v_1=v_0.
$$
Note that $L(i)$ is irreducible of type $\texttt{M}$ if $i\neq 0$,
and irreducible of type $\texttt{Q}$ if $i=0$. Moreover $L(i),
i\in\I$ form a complete set of pairwise non-isomorphic irreducible
$\mathcal{P}_1^{\mathfrak{c}}$-module in the category
$\operatorname{Rep}_{\I}\mathcal{P}_1^{\mathfrak{c}}$. Observe that
$\mpcn\cong \mathcal{P}_1^{\mathfrak{c}}\otimes\cdots\otimes
\mathcal{P}_1^{\mathfrak{c}}$, and hence we have the following
result by Lemma~\ref{tensorsmod}.

\begin{lem} \cite[Lemma 4.8]{BK}\label{lem:irrepPn}
The $\mpcn$-modules
$$
\{L(\underline{i})=L(i_1)\circledast
L(i_2)\circledast\cdots\circledast
L(i_n)|~\underline{i}=(i_1,\ldots,i_n)\in\I^n\}
$$
forms a complete set of pairwise non-isomorphic irreducible
$\mpcn$-module in the category $\operatorname{Rep}_{\I}\mpcn$.
Moreover, denote by $\gamma_0$ the number of $1\leq j\leq n$ with
$i_j=0$. Then $L(\underline{i})$ is of type $\texttt{M}$ if
$\gamma_0$ is even and type $\texttt{Q}$ if $\gamma_0$ is odd.
Furthermore,
$\text{dim}~L(\underline{i})=2^{n-\lfloor\frac{\gamma_0}{2}\rfloor}$,
where $\lfloor\frac{\gamma_0}{2}\rfloor$ denotes the greatest
integer less than or equal to $\frac{\gamma_0}{2}$ .
\end{lem}

\begin{rem}\label{rem:Ltau}
Note that each permutation $\tau\in S_n$ defines a superalgebra
isomorphism $\tau:\mpcn\rightarrow \mpcn$ by mapping $x_k$ to
$x_{\tau(k)}$ and $c_k$ to  $c_{\tau(k)}$, for $1\leq k\leq n$. For
$\underline{i}\in\I^n$, the twist of the action of
$\mathcal{P}_n^{\mathfrak{c}}$ on $L(\underline{i})$ with
$\tau^{-1}$ leads to a new $\mpcn$-module denoted by
$L(\underline{i})^{\tau}$ with
$$
L(\underline{i})^{\tau}=\{z^{\tau}~|~z\in L(\underline{i})\} ,\quad
fz^{\tau}=(\tau^{-1}(f)z)^{\tau}, \text{ for any }f\in
\mathcal{P}_n^{\mathfrak{c}}, z\in L(\underline{i}).
$$
So in particular we have $(x_kz)^{\tau}=x_{\tau(k)}z^{\tau}$ and
$(c_kz)^{\tau}=c_{\tau(k)}z^{\tau}$. It is easy to see that $
L(\underline{i})^{\tau}\cong L(\tau\cdot \underline{i})$, where
$\tau\cdot \underline{i}=(i_{\tau^{-1}(1)},\ldots,i_{\tau^{-1}(n)})$
for $\underline{i}=(i_1,\ldots,i_n)\in\I^n$ and $\tau\in S_n$.
\end{rem}
\subsection{Intertwining elements for $\mhcn$.}
Following~\cite{Na2}, we define the intertwining elements as

\begin{align}
\Phi_k:=s_k(x_k^2-x^2_{k+1})+(x_k+x_{k+1})+c_kc_{k+1}(x_k-x_{k+1}),\quad
1\leq k<n.\label{intertw}
\end{align}
It is known that
\begin{align}
\Phi_k^2=2(x_k^2+x^2_{k+1})-(x_k^2-x^2_{k+1})^2\label{sqinter},\\
\Phi_kx_k=x_{k+1}\Phi_k, \Phi_kx_{k+1}=x_k\Phi_k,
\Phi_kx_l=x_l\Phi_k\label{xinter},\\
\Phi_kc_k=c_{k+1}\Phi_k, \Phi_kc_{k+1}=c_k\Phi_k,
\Phi_kc_l=c_l\Phi_k\label{cinter},\\
\Phi_j\Phi_k=\Phi_k\Phi_j,
\Phi_k\Phi_{k+1}\Phi_k=\Phi_{k+1}\Phi_k\Phi_{k+1}\label{braidinter}
\end{align}
for all admissible $j,k,l$ with $l\neq k, k+1$ and $|j-k|>1$.
\section{Weights of completely splittable
$\mhcn$-modules}\label{wofCS} In this section, we shall describe the
weights of completely splittable $\mhcn$-modules.
\subsection{Structure of completely splittable $\mhcn$-modules.}
For $M\in\operatorname{Rep}_{\mathbb{I}}\mhcn$ and
$\underline{i}=(i_1,\ldots,i_n)\in\I^n$, set
$$
M_{\underline{i}}=\{z\in M~|~(x_k^2-q(i_k))^Nz=0 \text{ for }N\gg 0,
1\leq k\leq n\}.
$$
If $M_{\underline{i}}\neq 0$, then $\underline{i}$ is called a {\em
weight} of $M$ and $M_{\underline{i}}$ is called a weight space.
Since the polynomial generators $x_1,\ldots,x_n$ commute, we have
\begin{align}
M=\bigoplus_{\underline{i}\in\I^n}M_{\underline{i}}.\label{decomp.1}
\end{align}
For $i\in\I$ and $1\leq m\leq n$, set
$$
\Theta_{i^m}M=\{z\in M~|~ (x_j^2-q(i))^Nz=0, \text{ for }N\gg 0,
n-m+1\leq j\leq n\}.
$$
One can show using~(\ref{px1}) that
\begin{align}
x_k^2s_k&=s_kx_{k+1}^2-\big(x_k(1-c_kc_{k+1})+(1-c_kc_{k+1})x_{k+1}\big)\label{reln:x2s1}\\
x_{k+1}^2s_k&=s_kx_k^2+\big(x_{k+1}(1+c_kc_{k+1})+(1+c_kc_{k+1})x_k\big)\label{reln:x2s2}.
\end{align}
Hence $\Theta_{i^m}$ defines an exact functor
$$
\Theta_{i^m}: \RImhcn\longrightarrow
\operatorname{Rep}_{\I}\mathfrak{H}_{n-m,m}^{\mathfrak{c}}.
$$
Moreover as $\mathfrak{H}^{\mathfrak{c}}_{n-1,1}$-modules, we have
\begin{align}
{\rm
res}^{\mhcn}_{\mathfrak{H}^{\mathfrak{c}}_{n-1,1}}M=\oplus_{i\in\I}\Theta_iM.\label{decomp.2}
\end{align}
For $i\in\I$ and $M\in\RImhcn$, define
$$
\var_i(M)=\max\{m\geq 0~|~\Theta_{i^m}M\neq 0\}.
$$
\begin{lem}\cite[Lemma 5.4]{BK}\label{lem:K1-1}
Suppose that $M\in\RImhcn$ is irreducible. Let $i\in\I$ and
$m=\varepsilon_i(M)$. Then $\Theta_{i^m}M$ is isomorphic to
$L\circledast {\rm
ind}^{\mathfrak{H}_m^{\mathfrak{c}}}_{\mathcal{P}_m^{\mathfrak{c}}}L(i^m)$
for some irreducible $L
\in\operatorname{Rep}_{\I}\mathfrak{H}_{n-m}^{\mathfrak{c}}$ with
$\var_i(L)=0$.
\end{lem}
\begin{defn}\label{defn:CS}
A representation of $\mhcn$ is called {\em completely splittable} if
$x_1,\ldots,x_n$ act semisimply.
\end{defn}
\begin{rem}\label{rem:genera.}
Observe that if $M\in\RImhcn$ is completely splittable, then for
 $\underline{i}\in\I^n$,
$$
M_{\underline{i}}=\{z\in M~|~x_k^2z=q(i_k)z, 1\leq k\leq n\}.
$$
\end{rem}
\begin{lem}\label{lem:separate weig.} Suppose that $M\in\RImhcn$ is
 completely splittable and that
$M_{\underline{i}}\neq 0$ for some $\underline{i}\in\I^n$. Then
$i_k\neq i_{k+1}$ for all $1\leq k\leq n-1$.
\end{lem}
\begin{proof} Suppose $i_k=i_{k+1}$ for some $1\leq k\leq n-1$.
Let $0\neq z\in M_{\underline{i}}$. Since $M$ is completely
splittable, $(x_k^2-q(i_k))z=0=(x^2_{k+1}-q(i_{k+1}))z$. This
together with~(\ref{reln:x2s1}) shows that
$$
(x_k^2-q(i_k))s_kz=(x_k^2-q(i_{k+1}))s_kz=-\big(x_k(1-c_kc_{k+1})+(1-c_kc_{k+1})x_{k+1}\big)z.
$$
and hence
$$
(x_k^2-q(i_k))^2s_kz=
-\big(x_k(1-c_kc_{k+1})+(1-c_kc_{k+1})x_{k+1}\big)(x_k^2-q(i_k))z=0.
$$
Similarly, we see that
$$
(x_{k+1}^2-q(i_{k+1}))^2s_kz=0.
$$
Hence $s_kz\in M_{\underline{i}}$. By Remark~\ref{rem:genera.}, we
deduce that $(x_k^2-q(i_k))s_kz=0$ and therefore
$$
\big(x_k(1-c_kc_{k+1})+(1-c_kc_{k+1})x_{k+1}\big)z=0.
$$
This implies
\begin{align}
2(x_k^2+x^2_{k+1})z=\big(x_k(1-c_kc_{k+1})+(1-c_kc_{k+1})x_{k+1}\big)^2z=0.\notag
\end{align}
This means
 $q(i_{k+1})=-q(i_k)$ and hence $q(i_k)=q(i_{k+1})=0$ since
$i_k=i_{k+1}$. Therefore $x_k^2=0=x^2_{k+1}$ on $
M_{\underline{i}}$. Since $x_k,x_{k+1}$ act semisimply on
$M_{\underline{i}}$, $x_k=0=x_{k+1}$ on $ M_{\underline{i}}$. This
implies $x_{k+1}s_kz=0$ since $s_kz\in M_{\underline{i}}$ as shown
above. Then
$$(1+c_kc_{k+1})z=x_{k+1}s_kz-s_kx_kz=0.
$$ This means
$2z=(1-c_kc_{k+1})(1+c_kc_{k+1})z=0$. Hence $z=0$ since $p\neq 2$.
This contradicts the assumption that $z\neq 0$.
\end{proof}
\begin{cor}\label{cor:variM} Suppose that $M\in\RImhcn$ is completely splittable. Then
$\var_i(M)\leq 1$ for any $i\in\I$.
\end{cor}
\begin{prop}\label{prop:equiv.cond.} Let  $M\in\RImhcn$ be irreducible.
The following are equivalent.
\begin{enumerate}
\item $M$ is completely splittable.

\item For any $\underline{i}\in\I^n$ with $M_{\underline{i}}\neq
0$, we have $i_k\neq i_{k+1}$ for all $1\leq k\leq n-1$.

\item The restriction
${\rm res}^{\mhcn}_{\mathfrak{H}^{\mathfrak{c}}_{(r,1^{n-r})}}M$ is
semisimple for any $1\leq r\leq n$.

\item For any $\underline{i}\in\I^n$ with $M_{\underline{i}}\neq
0$, we have $M_{\underline{i}}\cong L(\underline{i})$ as
$\mpcn$-modules.
\end{enumerate}
\end{prop}
\begin{proof} By Lemma~\ref{lem:separate weig.}, $(1)$
implies $(2)$. Suppose $(2)$ holds, then by Lemma~\ref{lem:K1-1} and
Corollary~\ref{cor:variM} we have $\Theta_iM$ is either zero or
irreducible for each $i\in\I$ and hence by~(\ref{decomp.2}) ${\rm
res}^{\mhcn}_{\mathfrak{H}^{\mathfrak{c}}_{(n-1,1)}}M$ is
semisimple. Observe that if $\Theta_iM\cong N\circledast L(i)$ for
some irreducible
$N\in\operatorname{Rep}_{\I}\mathfrak{H}^{\mathfrak{c}}_{n-1}$, then
$(2)$ also holds for $N$. This implies ${\rm
res}^{\mathfrak{H}^{\mathfrak{c}}_{n-1}}_{\mathfrak{H}^{\mathfrak{c}}_{(n-2,1)}}N$
is semisimple. Therefore ${\rm
res}^{\mathfrak{H}^{\mathfrak{c}}_{n}}_{\mathfrak{H}^{\mathfrak{c}}_{(n-2,1,1)}}M$
is semisimple by~(\ref{decomp.2}). Continuing in this way we see
that the restriction ${\rm
res}^{\mhcn}_{\mathfrak{H}^{\mathfrak{c}}_{(r,1^{n-r})}}M$ is
semisimple for any $1\leq r\leq n$, whence $(3)$.

 Now assume $(3)$ holds.  In particular
${\rm res}^{\mhcn}_{\mathfrak{H}^{\mathfrak{c}}_{(1^{n})}}M$ is
semisimple, that is, $M$ is  isomorphic to a direct sum of
$L(\underline{i})$ as $\mpcn$-modules. It is clear that
$x_1,\ldots,x_n$ act semisimply on $L(\underline{i})$ for each
$\underline{i}\in\I^n$, whence $(1)$.

Clearly $(1)$ holds if $(4)$ is true. Now suppose $(1)$ holds and we
shall prove $(4)$ by induction on $n$. Suppose
$M_{\underline{i}}\neq 0$ for some $\underline{i}\in\I^n$. Observe
that $M_{\underline{i}}\subseteq\Theta_{i_n}M\neq 0$. By
Lemma~\ref{lem:K1-1} and Corollary~\ref{cor:variM},
$\Theta_{i_n}M\cong N\circledast L(i_n)$ for some irreducible
$N\in\operatorname{Rep}_{\I}\mathfrak{H}^{\mathfrak{c}}_{n-1}$. This
means $M_{\underline{i}}\cong N_{\underline{i}^{\prime}}\circledast
L(i)$, where $\underline{i}^{\prime}=(i_1,\ldots,i_{n-1})$. Note
that $N$ is completely splittable and hence by induction
$N_{\underline{i}^{'}}\cong L(i_1)\circledast\cdots\circledast
L(i_{n-1})$. Therefore $M_{\underline{i}}\cong
L(i_1)\circledast\cdots\circledast L(i_n)$.
\end{proof}
\begin{rem} \label{rem:twist}Note that $\mhcn$ possesses an automorphism $\sigma_n$
which sends $s_k$ to $-s_{n-k}$, $x_l$ to $x_{n+1-l} $ and $c_l$ to
$c_{n+1-l}$ for $1\leq k\leq n-1$ and $1\leq l\leq n$. Moreover
$\sigma_n$ induces an algebra isomorphism for each composition
$\mu=(\mu_1,\ldots,\mu_m)$ of $n$
$$\sigma_{\mu}:
\mathfrak{H}^{\mathfrak{c}}_{\mu}\longrightarrow
\mathfrak{H}^{\mathfrak{c}}_{\mu^t},
$$
where $\mu^t=(\mu_m,\ldots,\mu_1)$. Given
$M\in\mathfrak{H}^{\mathfrak{c}}_{\mu^t}$, we can twist with
$\sigma_{\mu}$ to get a $\mathfrak{H}^{\mathfrak{c}}_{\mu}$-module
$M^{\sigma_{\mu}}$. Observe that for $\mhcn$-module $M$, we have
$$
\big({\rm
res}^{\mhcn}_{\mathfrak{H}^{\mathfrak{c}}_{(r,1^{n-r})}}M^{\sigma_n}\big)^{\sigma_{(1^{n-r},r)}}
\cong{\rm res}^{\mhcn}_{\mathfrak{H}^{\mathfrak{c}}_{(1^{n-r},r)}}M.
$$
Hence $M\in\RImhcn$ is irreducible completely splittable  if and
only if ${\rm
res}^{\mhcn}_{\mathfrak{H}^{\mathfrak{c}}_{(1^{n-r},r)}}M$ is
semisimple for any $1\leq r\leq n$ by
Proposition~\ref{prop:equiv.cond.}.
\end{rem}
\begin{cor}\label{cor:restr.to 2}
Let  $M\in\RImhcn$ be irreducible completely splittable. Then the
restriction ${\rm
res}^{\mhcn}_{\mathfrak{H}^{\mathfrak{c}}_{(1^{k-1},2,1^{n-k-1})}}M$
is semisimple for any $1\leq k\leq n-1$. Hence $M$ is semisimple on
restriction to the subalgebra generated by
$s_k,x_k,x_{k+1},c_k,c_{k+1}$ which is isomorphic to
$\mathfrak{H}_2^{\mathfrak{c}}$ for fixed $1\leq k\leq n-1$.
\end{cor}
\begin{proof}By Proposition~\ref{prop:equiv.cond.},
${\rm res}^{\mhcn}_{\mathfrak{H}^{\mathfrak{c}}_{(k+1,1^{n-k-1})}}M$
is semisimple. Hence
$${\rm res}^{\mhcn}_{\mathfrak{H}^{\mathfrak{c}}_{(1^{k-1},2,1^{n-k-1})}}M={\rm res}^{\mathfrak{H}^{\mathfrak{c}}_{(k+1,1^{n-k-1})}}_{
\mathfrak{H}^{\mathfrak{c}}_{(1^{k-1},2,1^{n-k-1})}} \big({\rm
res}^{\mhcn}_{\mathfrak{H}^{\mathfrak{c}}_{(k+1,1^{n-k-1})}}M\big)$$
is semisimple by Remark~\ref{rem:twist}.
\end{proof}
\subsection{The weight constraints.}
Suppose that $M\in\RImhcn$ is completely splittable and that
$M_{\underline{i}}\neq 0$ for some $\underline{i}\in\I^n$. By
Lemma~\ref{lem:separate weig.}, $ i_k\neq i_{k+1}$ for $1\leq k\leq
n-1$. It follows from Remark~\ref{rem:genera.} that
$x_k^2-x_{k+1}^2$ acts as the nonzero scalar $q(i_k)-q(i_{k+1})$ on
$M_{\underline{i}}$ for each $1\leq k\leq n-1$. So we define linear
operators $\Xi_k$ and $\Omega_k$ on $M_{\underline{i}}$ such that
for any $z\in M_{\underline{i}}$,
\begin{align}
\Xi_kz&:=-\Big(\frac{x_k+x_{k+1}}{x_k^2-x_{k+1}^2}+c_kc_{k+1}\frac{x_k-x_{k+1}}{x_k^2-x_{k+1}^2}\Big)z,\label{Deltak}\\
\Omega_kz&:=\Bigg(\sqrt{1-\frac{2(x_k^2+x_{k+1}^2)}{(x_k^2-x_{k+1}^2)^2}}\Bigg)z
=\Bigg(\sqrt{1-\frac{2(q(i_k)+q(i_{k+1}))}{(q(i_k)-q(i_{k+1}))^2}}\Bigg)z.\label{omegak}
\end{align}
Both $\Xi_k$ and $\Omega_k$ make sense as linear operators on
$L(\underline{i})$ for $\underline{i}\in\I^n$ whenever $i_k\neq
i_{k+1}$ for $1\leq k\leq n$.
\begin{prop}\label{prop:rank 2}
The following holds for $i,j\in\I$.
\begin{enumerate}

\item If $i=j\pm 1$, then the irreducible
$\mathcal{P}^{\mathfrak{c}}_2$-module $L(i)\circledast L(j)$ affords
an irreducible $\mathfrak{H}_2^{\mathfrak{c}}$-module denoted by
$V(i,j)$ with the action $\ds s_1z=\Xi_1z$ for any $z\in
L(i)\circledast L(j)$. The $\mathfrak{H}_2^{\mathfrak{c}}$-module
$V(i,j)$ has the same type as the
$\mathcal{P}^{\mathfrak{c}}_2$-module $L(i)\circledast L(j)$.
Moreover, it is always completely splittable.

\item If $i\neq j\pm 1$, the
$\mathfrak{H}_2^{\mathfrak{c}}$-module $V(i,j):={\rm
ind}^{\mathfrak{H}_2^{\mathfrak{c}}}_{\mathcal{P}^{\mathfrak{c}}_2}L(i)\circledast
L(j)$ is irreducible and has the same type as the
$\mathcal{P}^{\mathfrak{c}}_2$-module $L(i)\circledast L(j)$. It is
completely splittable if and only if $i\neq j$ (and recall $i\neq
j\pm1$).

\item Every irreducible module in the category
$\operatorname{Rep}_{\I}\mathfrak{H}_2^{\mathfrak{c}}$ is isomorphic
to some $V(i,j)$.
\end{enumerate}
\end{prop}
\begin{proof}
$(1)$. It is routine to check $s_1x_1=x_2s_1-(1+c_1c_2)$ and
$s_1c_1=c_2s_1$, hence it remains to prove $s_1^2=1$ on $V(i,j)$.
Indeed, for $z\in L(i)\circledast L(j)$, we have
$$
s_1^2z=\Xi_1^2z=\frac{2(x_1^2+x_2^2)}{(x_1^2-x_2^2)^2}v=\frac{2(q(i)+q(j))}{(q(i)-q(j))^2}z=z,
$$
where the last identity follows from the definition of $q(i)$ and
the assumption $i=j\pm 1$. It is clear that ${\rm
End}_{\mathcal{P}^{\mathfrak{c}}_2}(L(i)\circledast L(j))\cong{\rm
End}_{\mathfrak{H}_2^{\mathfrak{c}}}(V(i,j))$. Hence $V(i,j)$ has
the same type as the $\mathcal{P}^{\mathfrak{c}}_2$-module
$L(i)\circledast L(j)$. Since $x_1,x_2$ act semisimply on
$L(i)\circledast L(j)$, $V(i,j)$ is completely splittable.

 $(2)$. Assume that $i\neq j\pm1$ and that $M$
is a nonzero proper submodule of $V(i,j)={\rm
ind}^{\mathfrak{H}_2^{\mathfrak{c}}}_{\mathcal{P}^{\mathfrak{c}}_2}L(i)\circledast
L(j)$. Observe that $V(i,j)=1\otimes(L(i)\circledast L(j))\oplus
s_1\otimes(L(i)\circledast L(j))$ as vector spaces. Without loss of
generality, we can assume $M$ contains a nonzero vector $v$ of the
form $v=1\otimes u+s_1\otimes u$ or $v=1\otimes u-s_1\otimes u$ for
some $0\neq u\in L(i)\circledast L(j)$. Otherwise, we can replace
$v$ by $v+s_1v$ or $v-s_1v$ since either of them is nonzero.
By~(\ref{reln:x2s1}),
\begin{align}
x_1^2v&=1\otimes x_1^2u\pm s_1\otimes
x_2^2u\mp1\otimes\big(x_1(1-c_1c_2)+(1-c_1c_2)x_2\big)u\notag\\
&=1\otimes q(i)u\pm q(j)s_1\otimes
u\mp1\otimes\big(x_1(1-c_1c_2)+(1-c_1c_2)x_2\big)u\notag.
\end{align}
This together with $(x_1^2-q(j))v\in M$ shows that
$$
1\otimes\Big((q(i)-q(j))u\pm\big(x_1(1-c_1c_2)+(1-c_1c_2)x_2\big)u\Big)
\in M.
$$
Since
$1\otimes\Big((q(i)-q(j))u\pm\big(x_1(1-c_1c_2)+(1-c_1c_2)x_2\big)u\Big)\in
L(i)\circledast L(j)$ and $M$ is a proper
$\mathfrak{H}_2^{\mathfrak{c}}$-submodule of $V(i,j)$, we have
$$
(q(i)-q(j))u\pm[x_1(1-c_1c_2)+(1-c_1c_2)x_2]u=0
$$
and therefore
$$
(q(i)-q(j))^2u=(x_1(1-c_1c_2)+(1-c_1c_2)x_2)^2u.
$$
This together with
 $(x_1(1-c_1c_2)+(1-c_1c_2)x_2)^2u=2(x_1^2+x_2^2)u$ shows that
$$
2(q(i)+q(j))=(q(i)-q(j))^2.
$$ This contradicts the assumption
$i\neq j\pm 1$ and hence $V(i,j)$ is irreducible.

Note that if $i\neq j$, then $V(i,j)$ has two weights, that is,
$(i,j)$ and $(j,i)$. By Proposition~\ref{prop:equiv.cond.}, we see
that ${\rm
res}_{\mathcal{P}^{\mathfrak{c}}_2}^{\mathfrak{H}_2^{\mathfrak{c}}}V(i,j)$
is semisimple and is isomorphic to the direct sum of
$L(i)\circledast L(j)$ and $L(j)\circledast L(i)$. This means
$${\rm Hom}_{\mathcal{P}^{\mathfrak{c}}_2}(L(i)\circledast
L(j),{\rm
res}^{\mathfrak{H}^{\mathfrak{c}}_2}_{\mathcal{P}_2^{\mathfrak{c}}}V(i,j))
\cong {\rm End}_{\mathcal{P}^{\mathfrak{c}}_2}(L(i)\circledast
L(j)).$$ By Frobenius reciprocity we obtain
$$
{\rm End}_{\mathfrak{H}_2^{\mathfrak{c}}}(V(i,j))\cong{\rm
Hom}_{\mathcal{P}_2^{\mathfrak{c}}}(L(i)\circledast L(j),{\rm
res}^{\mathfrak{H}_2^{\mathfrak{c}}}_{\mathcal{P}^{\mathfrak{c}}_2}V(i,j))
\cong {\rm End}_{\mathcal{P}^{\mathfrak{c}}_2}(L(i)\circledast
L(j)).$$ Hence $V(i,j)$ has the same type as the
$\mathcal{P}^{\mathfrak{c}}_2$-module $L(i)\circledast L(j)$.


Now suppose $i=j$. This implies that $(i,i)$ is a weight of $V(i,i)$
and hence $V(i,i)$ is not completely splittable by
Lemma~\ref{lem:separate weig.}. By
Proposition~\ref{prop:equiv.cond.}, ${\rm
res}_{\mathcal{P}_2^{\mathfrak{c}}}^{\mathfrak{H}_2^{\mathfrak{c}}}V(i,i)$
is not semisimple. Note that ${\rm
res}_{\mathcal{P}_2^{\mathfrak{c}}}^{\mathfrak{H}_2^{\mathfrak{c}}}V(i,i)$
has two composition factors and both of them are isomorphic to
$L(i)\circledast L(i)$. Therefore the socle of ${\rm
res}_{\mathcal{P}_2^{\mathfrak{c}}}^{\mathfrak{H}_2^{\mathfrak{c}}}V(i,i)$
is simple and isomorphic to $L(i)\circledast L(i)$. Hence ${\rm
Hom}_{\mathcal{P}_2^{\mathfrak{c}}}(L(i)\circledast L(i),{\rm
res}^{\mathfrak{H}_2^{\mathfrak{c}}}_{\mathcal{P}_2^{\mathfrak{c}}}V(i,i))
\cong {\rm End}_{\mathcal{P}_2^{\mathfrak{c}}}(L(i)\circledast
L(i))$. By Frobenius reciprocity we obtain
$$
{\rm
End}_{\mathfrak{H}_2^{\mathfrak{c}}}(V(i,i))\cong\text{Hom}_{\mathcal{P}_2^{\mathfrak{c}}}(L(i)\circledast
L(i),{\rm
res}^{\mathfrak{H}_2^{\mathfrak{c}}}_{\mathcal{P}_2^{\mathfrak{c}}}V(i,i))
\cong {\rm End}_{\mathcal{P}_2^{\mathfrak{c}}}(L(i)\circledast
L(i)).$$ Hence $V(i,i)$ has the same type as the
$\mathcal{P}^{\mathfrak{c}}_2$-module $L(i)\circledast L(i)$.

$(3)$. Suppose
$M\in\operatorname{Rep}_{\I}\mathfrak{H}_2^{\mathfrak{c}}$ is
irreducible, then there exist $i,j\in\I$ such that $L(i)\circledast
L(j)\subseteq {\rm
res}^{\mathfrak{H}_2^{\mathfrak{c}}}_{\mathcal{P}_2^{\mathfrak{c}}}M$.
By Frobenius reciprocity $M$ is an irreducible quotient of the
induced module ${\rm
ind}^{\mathfrak{H}_2^{\mathfrak{c}}}_{\mathcal{P}^{\mathfrak{c}}_2}L(i)\circledast
L(j)$. If $i\neq j\pm 1$, then $M\cong {\rm
ind}^{\mathfrak{H}_2^{\mathfrak{c}}}_{\mathcal{P}^{\mathfrak{c}}_2}L(i)\circledast
L(j)$ since ${\rm
ind}^{\mathfrak{H}_2^{\mathfrak{c}}}_{\mathcal{P}^{\mathfrak{c}}_2}L(i)\circledast
L(j)$ is irreducible by $(2)$; otherwise using the fact that
$\Xi_1^2=1$ on $L(i)\circledast L(j)$ one can show that the vector
space
$$
L:={\rm span}\big\{s_1\otimes u- 1\otimes\Xi_1u~|~u\in
L(i)\circledast L(j)\big\}
$$
is a $\mathfrak{H}_2^{\mathfrak{c}}$-submodule of ${\rm
ind}^{\mathfrak{H}_2^{\mathfrak{c}}}_{\mathcal{P}^{\mathfrak{c}}_2}L(i)\circledast
L(j)$ and it is isomorphic to $V(j,i)$. It is easy to check the
quotient ${\rm
ind}^{\mathfrak{H}_2^{\mathfrak{c}}}_{\mathcal{P}^{\mathfrak{c}}_2}L(i)\circledast
L(j)/L$ is isomorphic to $V(i,j)$. Hence $M\cong V(i,j)$.
\end{proof}
Observe from the proof above that if $i\neq j, j\pm1$ then the
completely splittable $\mathfrak{H}_2^{\mathfrak{c}}$-module
$V(i,j)$ has two weights $(i,j)$ and $(j,i)$ and moreover
$s_1-\Xi_1$ gives a bijection between the associated weight spaces.
This together with Corollary~\ref{cor:restr.to 2} and
Proposition~\ref{prop:rank 2} leads to the following.
\begin{cor}\label{cor:per. action} Let  $M\in\RImhcn$ be irreducible completely
splittable. Suppose $0\neq v\in M_{\underline{i}}$ for some
$\underline{i}=(i_1,\ldots,i_n)\in\I^n$. The following holds for
$1\leq k\leq n-1$.
\begin{enumerate}
\item If $i_k=i_{k+1}\pm 1$, then $\ds s_kv=\Xi_kv$.

\item If $i_k\neq i_{k+1}\pm 1$, then $\ds 0\neq
(s_k-\Xi_k)v\in M_{s_k\cdot\underline{i}}$ and hence
$s_k\cdot\underline{i}$ is a weight of $M$.
\end{enumerate}
\end{cor}
\begin{defn} Let $\underline{i}\in\I^n$. For $1\leq k\leq n-1$,
the simple transposition $s_k$ is called {\em admissible} with
respect to $\underline{i}$ if $i_k\neq i_{k+1}\pm1$.
\end{defn}

Let $W(\mhcn)$ be the set of weights $\underline{i}\in\I^n$ of
irreducible completely splittable $\mhcn$-modules. By
Corollary~\ref{cor:per. action}, if $\underline{i}\in W(\mhcn)$ and
$s_k$ is admissible with respect to $\underline{i}$, then
$s_k\cdot\underline{i}\in W(\mhcn)$; moreover $\underline{i}$ and
$s_k\cdot\underline{i}$ must occur as weights in an irreducible
completely splittable $\mhcn$-module simultaneously.

\begin{lem}\label{lem:restr.1} Let $\underline{i}\in W(\mhcn)$.
 Suppose that
$i_k=i_{k+2}$ for some $1\leq k\leq n-2$.
\begin{enumerate}
\item If $p=0$, then $i_k=i_{k+2}=0, i_{k+1}=1$.

\item If $p\geq 3$, then either $i_k=i_{k+2}=0, i_{k+1}=1$ or
$i_k=i_{k+2}=\frac{p-3}{2}, i_{k+1}=\frac{p-1}{2}$.
\end{enumerate}
\end{lem}
\begin{proof}
Suppose $\underline{i}$ occurs in the irreducible completely
splittable $\mhcn$-module $M$ and $i_k=i_{k+2}$ for some $1\leq
k\leq n-2$. If $i_k\neq i_{k+1}\pm1 $, then $s_k\cdot\underline{i}$
is a weight of $M$ with the form $(\cdots, u,u,\cdots)$  by
Corollary~\ref{cor:per. action}. This contradicts
Lemma~\ref{lem:separate weig.}. Hence $i_k=i_{k+1}\pm 1$. This
together with Corollary~\ref{cor:per. action} shows that $s_k=\Xi_k$
and $s_{k+1}=\Xi_{k+1}$ on $M_{\underline{i}}$ and by~(\ref{xc}) we
have
\begin{align}s_ks_{k+1}s_k-s_{k+1}s_ks_{k+1}
=&\frac{1}{(a-b)(b-a)(a-b)}(x_k+x_{k+2})(6x_{k+1}^2+2x_kx_{k+2})\notag\\
&+\frac{1}{(a-b)(b-a)(a-b)}c_kc_{k+2}(x_k-x_{k+2})(6x^2_{k+1}-2x_kx_{k+2})
\end{align}
on $M_{\underline{i}}$, where $a=q(i_k)=q(i_{k+2})$ and $
b=q(i_{k+1})$. This implies that for $z\in M_{\underline{i}}$,
\begin{align}(x_k+x_{k+2})(6x_{k+1}^2+2x_kx_{k+2})z
+c_kc_{k+2}(x_k-x_{k+2})(6x^2_{k+1}-2x_kx_{k+2})z=0.\label{diffbraid}
\end{align}
On $M_{\underline{i}}$,  $x_k,x_{k+2}$ act semisimply and $x_k^2,
x_{k+2}^2$ act as scalars $q(i_k), q(i_{k+2})$. Hence
 $M_{\underline{i}}$ admits a decomposition
$M_{\underline{i}}=N_1\oplus N_2$, where $N_1=\{z\in
M_{\underline{i}}~|~x_kz=x_{k+2}z=\pm\sqrt{q(i_k)}z\}$ and
$N_2=\{z\in M_{\underline{i}}~|~x_kz=-x_{k+2}z=\pm\sqrt{q(i_k)}z\}$.
Applying the identity~(\ref{diffbraid}) to $N_1$ and $N_2$, we
obtain
\begin{align}
2\sqrt{q(i_k)}\big(6q(i_{k+1})+2q(i_k)\big)=0.\label{eq.ik}
\end{align}
By the fact that $i_{k+1}=i_{k}\pm 1$, and the definition of
$q(i_k)$ and $q(i_{k+1})$, one can check that (\ref{eq.ik}) is
equivalent to the following
\begin{align}
i_{k+1}=i_k-1,\quad &\sqrt{i_k(i_k+1)}(4i_k-2)i_k=0\label{eq.ik1}\\
&or\notag\\
i_{k+1}=i_k+1, \quad
&\sqrt{i_k(i_k+1)}(4i_k+6)(i_k+1)=0.\label{eq.ik2}
\end{align}
(1). If $p=0$, since $i_k, i_{k+1}$ are nonnegative there is no
solution for the equation~(\ref{eq.ik1}) and the solution of
(\ref{eq.ik2}) is $i_k=0, i_{k+1}=1$.

\noindent (2). If $p\geq 3$, since $1\leq i_k,
i_{k+1}\leq\frac{p-3}{2}$ there is no solution for the
equation~(\ref{eq.ik1}) and the solutions of (\ref{eq.ik2}) are
$i_k=0, i_{k+1}=1$ or $i_k=\frac{p-3}{2}, i_{k+1}=\frac{p-1}{2}$.
\end{proof}
\begin{lem}\label{lem:restr.2}Let $\underline{i}\in W(\mhcn)$.
Suppose $i_k=i_l$ for some $1\leq k<l\leq n$.  Then
$i_k+1\in\{i_{k+1},\ldots,i_{l-1}\}$.
\end{lem}
\begin{proof} Suppose $i_k=i_l=u$ for some $1\leq k<l\leq n$.
Without loss of generality, we can assume
$u\notin\{i_{k+1},\ldots,i_{l-1}\}$. If $u=0$, then
$1\in\{i_{k+1},\ldots,i_{l-1}\}$; otherwise we can apply admissible
transpositions to $\underline{i}$ to obtain an element in $W(\mhcn)$
of the form $(\cdots,0,0,\cdots)$, which contradicts
Lemma~\ref{lem:separate weig.}.

Now assume $u\geq 1$ and $u+1\notin\{i_{k+1},\ldots,i_{l-1}\}$. If
$u-1$ does not appear between $i_{k+1}$ and $i_{l-1}$ in
$\underline{i}$, then we can apply admissible transpositions to
$\underline{i}$ to obtain an element in $W(\mhcn)$ of the form
$(\cdots,u,u,\cdots)$, which contradicts Lemma~\ref{lem:separate
weig.}. If $u-1$ appears only once between $i_{k+1}$ and $i_{l-1}$
in $\underline{i}$, then we can apply admissible transpositions to
$\underline{i}$ to obtain an element in $W(\mhcn)$ of the form
$(\cdots,u,u-1,u,\cdots)$, which contradicts
Lemma~\ref{lem:restr.1}. Hence $u-1$ appears at least twice between
$i_{k+1}$ and $i_{l-1}$ in $\underline{i}$. This implies that there
exist $k<k_1<l_1<l$ such that
$$
i_{k_1}=i_{l_1}=u-1,
\{u,u-1\}\cap\{i_{k_1+1},\ldots,i_{l_1-1}\}=\emptyset.
$$
An identical argument shows that there exist $k_1<k_2<l_2<l_1$ such
that
$$
i_{k_2}=i_{l_2}=u-2,
\{u,u-1,u-2\}\cap\{i_{k_2+1},\ldots,i_{l_2-1}\}=\emptyset.
$$
Continuing in this way, we obtain $k<s<t<l$ such that
$$
i_{s}=i_{t}=0,
\{u,u-1,\ldots,1,0\}\cap\{i_{s+1},\ldots,i_{t-1}\}=\emptyset,
$$
which is impossible as shown at the beginning.
\end{proof}
\begin{prop}\label{prop:restr.2} Let
$\underline{i}\in W(\mhcn)$. Then
\begin{enumerate}
\item $i_k\neq i_{k+1}$ for all $1\leq k\leq n-1$.

\item If $p\geq 3$, then $\frac{p-1}{2}$ appears at most once in
$\underline{i}$.

\item If $i_k=i_l=0$ for some $1\leq k<l\leq n$, then
$1\in\{i_{k+1},\ldots, i_{l-1}\}$.

 \item If $p=0$ and $
i_k=i_l\geq 1$ for some $1\leq k<l\leq n$, then
$\{i_k-1,i_k+1\}\subseteq\{i_{k+1},\ldots, i_{l-1}\}$.

\item If $p\geq 3$ and
 $i_k=i_l\geq 1$ for
some $1\leq k<l\leq n$, then either of the following holds:
\begin{enumerate}
\item $\{i_k-1,i_k+1\}\subseteq\{i_{k+1},\ldots, i_{l-1}\}$,

\item there exists a sequence of integers
$k\leq r_0< r_1<\cdots< r_{\frac{p-3}{2}-i_k}< q<
t_{\frac{p-3}{2}-i_k}<\cdots< t_1<t_0\leq l$ such that
$i_q=\frac{p-1}{2}, i_{r_j}=i_{t_j}=i_k+j$ and $i_k+j$ does not
appear between $i_{r_j}$ and $i_{t_j}$ in $\underline{i}$ for each
$0\leq j\leq \frac{p-3}{2}-i_k$.
\end{enumerate}
\end{enumerate}
\end{prop}
\begin{proof}
(1). It follows from Lemma~\ref{lem:separate weig.}.

(2). If $\frac{p-1}{2}$ appears more than once in $\underline{i}$,
then it follows from Lemma~\ref{lem:restr.2} that $\frac{p+1}{2}$
appears in $\underline{i}$ which is impossible since
$\frac{p+1}{2}\notin\I$.

(3). It follows from Lemma~\ref{lem:restr.2}.

(4). Now suppose $p=0$ and $i_k=i_l=u\geq 1$ for some $1\leq k<l\leq
n$. Without loss of generality, we can assume
$u\notin\{i_{k+1},\ldots,i_{l-1}\}$. By Lemma~\ref{lem:restr.2} we
have $u+1\in\{i_{k+1},\ldots,i_{l-1}\}$ and hence it suffices to
show $u-1\in\{i_{k+1},\ldots,i_{l-1}\}$. Now assume
$u-1\notin\{i_{k+1},\ldots,i_{l-1}\}$. Then $u+1$ must appear in the
subsequence $(i_{k+1},\ldots,i_{l-1})$ at least twice, otherwise we
can apply admissible transpositions to $\underline{i}$ to obtain an
element in $W(\mhcn)$ of the form $(\cdots,u,u+1,u\cdots)$ which
contradicts Lemma~\ref{lem:restr.1}. Hence there exist $k<k_1<l_1<l$
such that
$$
i_{k_1}=i_{l_1}=u+1,\quad  u+1 \text{ does not appear between
}i_{k_1}\text{ and }i_{l_1} \text{ in }\underline{i}.
$$
Since
$u\notin\{i_{k+1},\ldots,i_{l-1}\}\supseteq\{i_{k_1+1},\ldots,i_{l_1-1}\}$,
a similar argument gives $k_2,l_2$ with $k_1<k_2<l_2<l_1$ such
that
$$
i_{k_2}=i_{l_2}=u+2, \quad u+2 \text{ does not appear between
}i_{k_2}\text{ and }i_{l_2} \text{ in }\underline{i}.
$$
Continuing in this way we see that any integer greater than $u$ will
appear in the subsequence $(i_{k+1},\ldots,i_{l-1})$ which is
impossible. Hence $u-1\in\{i_{k+1},\ldots, i_{l-1}\}$.

(5). Suppose $p\geq 3$ and $1\leq i_k=i_l=u\leq\frac{p-3}{2}$ for
some $1\leq k<l\leq n$ and $u-1\notin\{i_{k+1},\ldots,i_{l-1}\}$.
Clearly there exist $k\leq r_0<t_0\leq l$ such that
$$
i_{r_0}=i_{t_0}=u, u\notin\{i_{r_0+1},\ldots,i_{t_0-1}\}.
$$
An identical argument used for proving $(2)$ shows that there exists
a sequence of integers
$$
k\leq r_0< r_1<\cdots< r_{\frac{p-3}{2}-u}<
t_{\frac{p-3}{2}-u}<\cdots< t_1<t_0\leq l
$$
such that
$$
r_j=t_j=u+j,\quad
\{u,u+1,\ldots,u+j\}\cap\{i_{r_j+1},\quad\ldots,i_{t_j-1}\}=\emptyset
$$
for each $0\leq j\leq \frac{p-3}{2}-u$. Since
$i_{r_{\frac{p-3}{2}-u}}=i_{t_{\frac{p-3}{2}-u}}=\frac{p-3}{2}$, by
Lemma~\ref{lem:restr.2} there exists $r_{\frac{p-3}{2}-u}< q<
t_{\frac{p-3}{2}-u}$ such that $i_q=\frac{p-1}{2}$.
\end{proof}
\section{Classification of irreducible completely splittable $\mhcn$-modules} \label{classification}
In this section, we shall give an explicit construction and a
classification of irreducible completely splittable $\mhcn$-modules.

Recall that for $\underline{i}\in\I^n$ and $1\leq k\leq n-1$,  the
simple transposition $s_k$ is said to be admissible with respect
to $\underline{i}$ if $i_k\neq i_{k+1}\pm1$. Define an equivalence
relation $\sim$ on $\I^n$ by declaring that
$\underline{i}\sim\underline{j}$ if there exist $s_{k_1},\ldots,
s_{k_t}$ for some $t\in\mathbb{Z}_+$ such that
$\underline{j}=(s_{k_t}\cdots s_{k_1})\cdot\underline{i}$ and
$s_{k_{l}}$ is admissible with respect to $(s_{k_{l-1}}\cdots
s_{k_1})\cdot\underline{i}$ for $1\leq l\leq t$.

Denote by $W'(\mhcn)$ the set of $\underline{i}\in\I^n$ satisfying
the properties (3), (4) and (5) in Proposition~\ref{prop:restr.2}.
Observe that if $\underline{i}\in W^{\prime}(\mhcn)$ and $s_k$ is
admissible with respect to $\underline{i}$,  then the properties in
Proposition~\ref{prop:restr.2} hold for $s_k\cdot\underline{i}$ and
hence $s_k\cdot\underline{i}\in W'(\mhcn)$. This means there is an
equivalence relation denoted by $\sim$ on $W'(\mhcn)$ inherited from
the equivalence relation $\sim$ on $\I^n$. For each
$\underline{i}\in W^{\prime}(\mhcn)$, set
\begin{align}
P_{\underline{i}}=\{\tau=s_{k_t}\cdots s_{k_1}|~s_{k_l}\text{ is
admissible with respect to } s_{k_{l-1}}\cdots
s_{k_1}\cdot\underline{i}, 1\leq l\leq t,
t\in\mathbb{Z}_+\}.\label{Punderi}
\end{align}
\begin{lem}\label{lem:lambdaP}
 Let $\Lambda\in W^{\prime}(\mhcn)/\sim$ and
$\underline{i}\in\Lambda$. Then the map
$$
\varphi: P_{\underline{i}}\rightarrow \Lambda, \tau\mapsto
\tau\cdot\underline{i}
$$
is bijective.
\end{lem}

\begin{proof}
By the definitions of $P_{\underline{i}}$ and the equivalence
relation $\sim$ on $W'(\mhcn)$, one can check that $\varphi$ is
surjective. Note that if $\tau, \sigma\in P_{\underline{i}}$ then
$\sigma^{-1}\tau\in P_{\underline{i}}$. Therefore, to check the
injectivity of $\varphi$, it suffices to show that for $\tau\in
P_{\underline{i}}$ if $\tau\cdot\underline{i}=\underline{i}$ then
$\tau=1$. Associated to each $\underline{j}\in W'(\mhcn)$, there
exists a unique table $\Gamma(\underline{j})$ whose $a$th column
consists of all numbers $k$ with $j_k=a$ and is increasing for each
$a\in\I$. Since $\underline{j}\in W'(\mhcn)$, $j_k\neq j_{k+1}$ and
hence $k$ and $k+1$ are in different columns in
$\Gamma(\underline{j})$ for each $1\leq k\leq n-1$. This means each
simple transposition $s_k$ can naturally act on the table
$\Gamma(\underline{j})$ by switching $k$ and $k+1$ to obtain a new
table denoted by $s_k\cdot\Gamma(\underline{j})$. It is clear that
\begin{align}
s_k\cdot\Gamma(\underline{j})=\Gamma(s_k\cdot\underline{j}), \quad
1\leq k\leq n-1.\label{table}
\end{align}
Since $\tau\in P_{\underline{i}}$, we can write
$\tau=s_{k_t}s_{k_{t-1}}\cdots s_{k_1}$ so that $s_{k_l}$ is
admissible with respect to $s_{k_{l-1}}\cdots
s_{k_1}\cdot\underline{i}$ for each $1\leq l\leq t$. Observe that
$s_{k_{l-1}}\cdots s_{k_1}\cdot\underline{i}\in W'(\mhcn)$ and
hence there exists a table $\Gamma(s_{k_{l-1}}\cdots
s_{k_1}\cdot\underline{i})$ as defined above for $1\leq l\leq t$.
By (\ref{table}) we have
$$
s_{k_l}\cdot\Gamma(s_{k_{l-1}}\cdots s_{k_1}\cdot\underline{i})
=\Gamma(s_{k_l}s_{k_{l-1}}\cdots s_{k_1}\cdot\underline{i})
$$
for $1\leq l\leq t$. This implies
$$
\tau\cdot\Gamma(\underline{i})=s_{k_t}\cdots
s_{k_1}\cdot\Gamma(\underline{i})=\Gamma(s_{k_t}\cdots
s_{k_1}\cdot\underline{i})=\Gamma(\underline{i}).
$$
Therefore $\tau=1$.

\end{proof}

Before stating the main theorem of this section, we need the
following two lemmas. Let  $M\in\RImhcn$ be irreducible completely
splittable and suppose $M_{\underline{i}}\neq0$ for some
$\underline{i}=(i_1,\ldots,i_n)\in\I^n$. Recall the linear operators
$\Xi_k$ and $\Omega_k$ on $M_{\underline{i}}$
from~(\ref{Deltak})~and~(\ref{omegak}). If $s_k$ is admissible with
respect to $\underline{i}$, then $i_k\neq i_{k+1}\pm1$ and hence
$2(q(i_k)+q(i_{k+1}))\neq (q(i_k)-q(i_{k+1}))^2$. This implies that
on $M_{\underline{i}}$ the linear operator $\Omega_k$ acts as a
nonzero scalar and hence is invertible. Therefore we can define the
linear map $\widehat{\Phi}_k$ as follows:
\begin{align}
\widehat{\Phi}_k: M_{\underline{i}}&\longrightarrow M,\notag\\
z\mapsto&(s_k-\Xi_k)\frac{1}{\Omega_k} z.\notag
\end{align}

\begin{lem}\label{lem:newoperator} Let $M\in\RImhcn$ be irreducible
completely splittable. Assume that $M_{\underline{i}}\neq0$ and that
$s_k$ is admissible with respect to $\underline{i}$  for some
$\underline{i}=(i_1,\ldots,i_n)\in\I^n$ and $1\leq k\leq n-1$. Then,
\begin{enumerate}
\item $\widehat{\Phi}_k$ satisfies
\begin{align}
\widehat{\Phi}_kx_k=x_{k+1}\widehat{\Phi}_k,~\widehat{\Phi}_kx_{k+1}&=x_k\widehat{\Phi}_k,~
\widehat{\Phi}_kx_l=x_l\widehat{\Phi}_k,\label{xinterprime}\\
\widehat{\Phi}_kc_k=c_{k+1}\widehat{\Phi}_k,~\widehat{\Phi}_kc_{k+1}&=c_k\widehat{\Phi}_k,~
\widehat{\Phi}_kc_l=c_l\widehat{\Phi}_k,\label{cinterprime}
\end{align}
for $1\leq l\leq n$ with $|k-l|>1$. Hence for each $z\in
M_{\underline{i}}$, $\widehat{\Phi}_k(z)\in
M_{s_k\cdot\underline{i}}$.

\item $\widehat{\Phi}_k^2=1$, and hence $\widehat{\Phi}_k:
M_{\underline{i}}\rightarrow M_{s_k\cdot\underline{i}}$ is a
bijection.



\item
\begin{align}
\widehat{\Phi}_j\widehat{\Phi}_l&=\widehat{\Phi}_l\widehat{\Phi}_j~
\text{ if } |j-l|>1\label{braidprime1},\\
\widehat{\Phi}_j\widehat{\Phi}_{j+1}\widehat{\Phi}_{j}&=
\widehat{\Phi}_{j+1}\widehat{\Phi}_j\widehat{\Phi}_{j+1}\label{braidprime2}.
\end{align}
whenever both sides are well-defined.
\end{enumerate}
\end{lem}
\begin{proof} (1)  Recalling the intertwining element $\Phi_k$
from~(\ref{intertw}), we see that
\begin{align}
\widehat{\Phi}_k=\Phi_k\frac{1}{x_k^2-x^2_{k+1}}\frac{1}{\Omega_k}.\label{sDelta}
\end{align}
This together with~(\ref{xinter})~and~(\ref{cinter})
implies~(\ref{xinterprime})~and~(\ref{cinterprime}).
By~(\ref{xinterprime}), we have for any $z\in M_{\underline{i}}$,
\begin{align}
(x^2_k-q(i_{k+1}))\widehat{\Phi}_kz=0,~
(x^2_{k+1}-q(i_k))\widehat{\Phi}_kz=0,~
(x^2_l-q(i_{l}))\widehat{\Phi}_kz=0,\text{ for all }l\neq
k,k+1.\notag
\end{align}
This means $\widehat{\Phi}_kz\in M_{s_k\cdot\underline{i}}$.

(2) By (\ref{sqinter})~and~(\ref{sDelta}), one can check that for
$z\in M_{\underline{i}}$,
$$
\widehat{\Phi}_k^2z=\Phi_k^2\frac{1}{(x_k^2-x^2_{k+1})(x^2_{k+1}-x_k^2)}\frac{1}{\Omega_k^2}z
=\Big(1-\frac{2(x_k^2+x^2_{k+1})}{(x_k^2-x^2_{k+1})^2}\Big)\frac{1}{\Omega_k^2}z=z.
$$
Hence $\widehat{\Phi}_k^2=1$ and so $\widehat{\Phi}_k$ is bijective.

(3). If $|j-l|>1$ and both $\widehat{\Phi}_j\widehat{\Phi}_l$ and
$\widehat{\Phi}_l\widehat{\Phi}_j$ are well-defined on
$M_{\underline{i}}$, then by~(\ref{xinter})~and~(\ref{sDelta}) we
see that
\begin{align}
\widehat{\Phi}_j\widehat{\Phi}_l&=\Phi_j\Phi_l
\frac{1}{\Omega_j\Omega_l(x_j^2-x^2_{j+1})(x_l^2-x^2_{l+1})}, \notag\\
\widehat{\Phi}_l\widehat{\Phi}_j&=\Phi_l\Phi_j
\frac{1}{\Omega_l\Omega_j(x_l^2-x^2_{l+1})(x_j^2-x^2_{j+1})}\notag.
\end{align}
This together with~(\ref{braidinter}) implies~(\ref{braidprime1}).
By~(\ref{sDelta}), one can check that if both
$\widehat{\Phi}_k\widehat{\Phi}_{k+1}\widehat{\Phi}_{k}$ and
$\widehat{\Phi}_{k+1}\widehat{\Phi}_k\widehat{\Phi}_{k+1}$ are
well-defined on $M_{\underline{i}}$ then
\begin{align}
\widehat{\Phi}_k\widehat{\Phi}_{k+1}\widehat{\Phi}_{k}&=C\Phi_k\Phi_{k+1}\Phi_k\notag,\\
\widehat{\Phi}_{k+1}\widehat{\Phi}_k\widehat{\Phi}_{k+1}&=C\Phi_{k+1}\Phi_k\Phi_{k+1}\notag,
\end{align}
where $C$ is the scalar
$$
C=\frac{1}{(a-b)(a-c)(b-c)}\sqrt{1-\frac{2(a+b)}{(a-b)^2}}
\sqrt{1-\frac{2(a+c)}{(a-c)^2}} \sqrt{1-\frac{2(b+c)}{(b-c)^2}}
$$
with $a=q(i_k), b=q(i_{k+1}), c=q(i_{k+2})$.
Hence~(\ref{braidprime2}) follows from~(\ref{braidinter}).
\end{proof}
\begin{rem}\label{rem:admissible}
Suppose that $M\in\text{Rep}_{\I}\mhcn$ is completely splittable. By
Lemma~\ref{lem:newoperator}, if $M_{\underline{i}}\neq 0$ and
$\underline{j}\sim\underline{i}$, then $M_{\underline{j}}\neq 0$.
\end{rem}
\begin{lem} \label{lem:newbij.}Let $M\in\operatorname{Rep}_{\I}{\mhcn}$ be irreducible
completely splittable. Suppose that $M_{\underline{i}}\neq 0$ for
some $\underline{i}\in\I^n$ and $\tau\in P_{\underline{i}}$. Write
$\tau=s_{k_t}\cdots s_{k_1}$ so that $s_{k_l}$ is admissible with
respect to $s_{k_{l-1}}\cdots s_{k_1}\cdot \underline{i}$ for $1\leq
l\leq t$. Then
$$
\widehat{\Phi}_{\tau}:
=\widehat{\Phi}_{k_t}\cdots\widehat{\Phi}_{k_1}:
M_{\underline{i}}\longrightarrow M_{\tau\cdot\underline{i}}
$$
is a bijection satisfying
$x_k\widehat{\Phi}_{\tau}=\widehat{\Phi}_{\tau}x_{\tau(k)}$ and
$c_k\widehat{\Phi}_{\tau}=\widehat{\Phi}_{\tau}c_{\tau(k)}$ for
$1\leq k\leq n$. Moreover $\widehat{\Phi}_{\tau}$ does not depend on
the choice of the expression $s_{k_t}\cdots s_{k_1}$ for $\tau$.
\end{lem}
\begin{proof}
Since $s_{k_l}$ is admissible with respect to $s_{k_{l-1}}\cdots
s_{k_1}\cdot \underline{i}$ for $1\leq l\leq t$, each
$\widehat{\Phi}_{k_l}$ is a well-defined bijection from
$M_{s_{k_{l-1}}\cdot s_{k_1}\cdot\underline{i}}$ to
$M_{s_{k_{l}}\cdot s_{k_1}\cdot\underline{i}}$ by
Lemma~\ref{lem:newoperator} and hence $\widehat{\Phi}_{\tau}$ is
bijective. By~(\ref{braidprime1}) and~(\ref{braidprime2}),
$\widehat{\Phi}_{\tau}$ does not depend on the choice of the
expression $s_{k_t}\cdots s_{k_1}$ for $\tau$.
Using~(\ref{xinterprime}) and~(\ref{cinterprime}), we obtain
$x_k\widehat{\Phi}_{\tau}=\widehat{\Phi}_{\tau}x_{\tau(k)}$ and
$c_k\widehat{\Phi}_{\tau}=\widehat{\Phi}_{\tau}c_{\tau(k)}$ for
$1\leq k\leq n$.
\end{proof}
%
Suppose $\underline{i}\in W^{\prime}(\mhcn)$. Recall the definition
of $L(\underline{i})^{\tau}$ from Remark~\ref{rem:twist} for
$\tau\in P_{\underline{i}}$. Denote by $D^{\underline{i}}$ the
$P_n^{\mathfrak{c}}$-module defined by \begin{align}\ds
D^{\underline{i}}=\oplus_{\tau\in P_{\underline{i}}}
L(\underline{i})^{\tau}.\label{Dunderi}
\end{align}
The following theorem is the main result of this paper.
\begin{thm}\label{thm:Classficiation}
 Suppose $\underline{i}, \underline{j}\in W'(\mhcn)$. Then,
 \begin{enumerate}
 \item $D^{\underline{i}}$ affords an irreducible $\mhcn$-module via
\begin{align}
s_kz^{\tau}= \left \{
 \begin{array}{ll}
 \Xi_kz^{\tau}
 +\Omega_kz^{s_k\tau},
 & \text{ if } s_k \text{ is admissible with respect to } \tau\cdot\underline{i}, \\
 \Xi_kz^{\tau}
 , & \text{ otherwise },
 \end{array}
 \right.\label{actionformula}
\end{align}
 for  $1\leq k\leq n-1, z\in L(\underline{i})$ and $\tau\in
P_{\underline{i}}$.  It has the same type as the irreducible
$P_n^{\mathfrak{c}}$-module $L(\underline{i})$.

\item
$D^{\underline{i}}\cong D^{\underline{j}}$ if and only if
$\underline{i}\sim\underline{j}$.

\item Every irreducible completely splittable $\mhcn$-module in $\operatorname{Rep}_{\I}\mhcn$
is isomorphic to $D^{\underline{i}}$ for some $\underline{i}\in
W'(\mhcn)$. Hence the equivalence classes $W^{\prime}(\mhcn)/\sim$
parametrize irreducible completely splittable $\mhcn$-modules in the
category $\operatorname{Rep}_{\I}\mhcn$.

\end{enumerate}
\end{thm}
\begin{proof}
(1). To show the formula~(\ref{actionformula}) defines a
$\mhcn$-module structure on $D^{\underline{i}}$, we need to check
the defining relations~(\ref{braid}),~(\ref{px1}),~(\ref{px2})
and~(\ref{pc}) on $L(\underline{i})^{\tau}$ for each $\tau\in
P_{\underline{i}}$. One can show using~(\ref{xc}) that
\begin{align}
\Xi_kx_k-x_{k+1}\Xi_k=-(1+c_kc_{k+1}).\label{Deltax}
\end{align}
For $1\leq k\leq n-1$,
$(x_{\tau^{-1}(k)}z)^{s_k\tau}=x_{k+1}z^{s_k\tau}$ by
 Remark~\ref{rem:Ltau} and hence if
$s_k$ is admissible with respect to $ \tau\cdot\underline{i}$,
then
\begin{align}
s_kx_kz^{\tau}=s_k(x_{\tau^{-1}(k)}z)^{\tau}&=
\Xi_k(x_{\tau^{-1}(k)}z)^{\tau}+\Omega_k(x_{\tau^{-1}(k)}z)^{s_k\tau}\notag\\
&=\Xi_kx_kz^{\tau}+x_{k+1}\Omega_kz^{s_k\tau}\notag\\
&=(\Xi_kx_k-x_{k+1}\Xi_k)z^{\tau}+x_{k+1}(\Xi_kz^{\tau}+\Omega_kz^{s_k\tau})\notag\\
&=-(1+c_kc_{k+1})z^{\tau}+x_{k+1}s_kz^{\tau}\notag \quad\text{ by
}(\ref{Deltax}).
\end{align}
Otherwise we have
\begin{align}
s_kx_kz^{\tau}=s_k(x_{\tau^{-1}(k)}z)^{\tau}
&=\Xi_k(x_kz^{\tau})\notag\\
&=(\Xi_kx_k-x_{k+1}\Xi_k)z^{\tau}+x_{k+1}\Xi_kz^{\tau}\notag\\
&=-(1+c_kc_{k+1})z^{\tau}+x_{k+1}s_kz^{\tau}\notag \quad \text{ by
}(\ref{Deltax}).
\end{align}
Therefore~(\ref{px1}) holds. It is routine to check (\ref{px2})
and~(\ref{pc}).

It remains to prove~(\ref{braid}). It is clear by~(\ref{px2}) that
$s_ks_l=s_ls_k$ if $|l-k|>1$, so it suffices to prove $s_k^2=1$
and $s_ks_{k+1}s_k=s_{k+1}s_ks_{k+1}$. For the remaining of the
proof, let us fix $\tau\in P_{\underline{i}}$ and set
$\underline{j}=\tau\cdot\underline{i}$. One can check
using~(\ref{xc}) and ~(\ref{actionformula}) that
\begin{eqnarray*}
s_k^2z^{\tau}= \left \{
 \begin{array}{ll}
 (\Xi_k^2
 +\Omega_k^2)z^{\tau},
 & \text{ if } s_k \text{ is admissible with respect to } \underline{j}=\tau\cdot\underline{i} \\
\Xi_k^2z^{\tau}
 , & \text{ otherwise }.
 \end{array}
 \right.
\end{eqnarray*}
Hence if $s_k$ is admissible with respect to
$\underline{j}=\tau\cdot\underline{i}$, then
$$
s_k^2z^{\tau}=\Xi_k^2z^{\tau}
 +\Omega_k^2z^{\tau}=\Big(\frac{2(x_k^2+x_{k+1}^2)}{(x_k^2-x_{k+1}^2)^2}\Big)z^{\tau}
 +\Big(1-\frac{2(x_k^2+x_{k+1}^2)}{(x_k^2-x_{k+1}^2)^2}\Big)z^{\tau}=z^{\tau}.
$$
Otherwise we have $j_{k}=j_{k+1}\pm1$. This implies $
2(q(j_{k})+q(j_{k+1}))=(q(j_{k})-q(j_{k+1}))^2 $ and hence
$$
s_k^2z^{\tau}=\Xi_k^2z^{\tau}=\frac{2(q(j_{k})+q(j_{k+1}))}
{(q(j_{k})-q(j_{k+1}))^2}z^{\tau}=z^{\tau}.
$$
Therefore $s_k^2=1$ on $D^{\underline{i}}$ for $1\leq k\leq n-1$ .
Next we shall prove $s_ks_{k+1}s_k=s_{k+1}s_ks_{k+1}$ for $1\leq
k\leq n-2$. Set $\widehat{s}_k=s_k-\Xi_{k}$ for $1\leq k\leq n-1$.
It is clear by~(\ref{actionformula}) that
\begin{eqnarray*}
\widehat{s}_kz^{\tau}= \left \{
 \begin{array}{ll}
 \Omega_kz^{s_k\tau},
 & \text{ if } s_k \text{ is admissible with respect to } \underline{j}=\tau\cdot\underline{i}, \\
 0, & \text{ otherwise }.
 \end{array}
 \right.
\end{eqnarray*}
If $j_k-j_{k+1}=\pm1$,  $j_{k+1}-j_{k+2}=\pm1$ or
$j_{k}-j_{k+2}=\pm1$, then
$\widehat{s}_k\widehat{s}_{k+1}\widehat{s}_k
=0=\widehat{s}_{k+1}\widehat{s}_k\widehat{s}_{k+1}$ on
$L(\underline{i})^{\tau}$; otherwise, one can show
using~(\ref{omegak}) that
$$\widehat{s}_k\widehat{s}_{k+1}\widehat{s}_kz^{\tau}
=\Big(\sqrt{1-\frac{2(a+b)}{(a-b)^2}}\sqrt{1-\frac{2(b+c)}{(b-c)^2}}
\sqrt{1-\frac{2(a+c)}{(a-c)^2}}\Big)z^{\tau}
=\widehat{s}_{k+1}\widehat{s}_k\widehat{s}_{k+1}z^{\tau}, $$  for
any  $z\in L(\underline{i})$, where $a=q(j_k), b=q(j_{k+1}),
c=q(j_{k+2})$. Hence
\begin{align}
\widehat{s}_k\widehat{s}_{k+1}\widehat{s}_kz^{\tau}
=\widehat{s}_{k+1}\widehat{s}_k\widehat{s}_{k+1}z^{\tau}, \text{
for any }z\in L(\underline{i}), 1\leq k\leq n-2.\label{braid'}
\end{align}
%
Fix $1\leq k\leq n-2$. If $j_{k}\neq j_{k+2}$, then
$\ds\frac{1}{(x_k^2-x_{k+1}^2)(x_k^2-x_{k+2}^2)(x_{k+1}^2-x_{k+2}^2)}$
acts as the nonzero scalar $\frac{1}{(a-b)(a-c)(b-c)}$ on
$L(\underline{i})^{\tau}$. Recalling the intertwining elements
$\Phi_k$ from~(\ref{intertw}), we see that
\begin{align}
\widehat{s}_k=\Phi_k\frac{1}{x_k^2-x^2_{k+1}}.\notag
\end{align}
This together with~(\ref{braidinter}) shows that for any $z\in
L(\underline{i})$,
$$
\widehat{s}_k\widehat{s}_{k+1}\widehat{s}_kz^{\tau}
=\Phi_k\Phi_{k+1}\Phi_k\frac{1}{(x_k^2-x_{k+1}^2)(x_k^2-x_{k+2}^2)(x_{k+1}^2-x_{k+2}^2)}z^{\tau},
$$
and
$$
\widehat{s}_{k+1}\widehat{s}_k\widehat{s}_{k+1}z^{\tau}
=\Phi_{k+1}\Phi_k\Phi_{k+1}\frac{1}{(x_k^2-x_{k+1}^2)(x_k^2-x_{k+2}^2)(x_{k+1}^2-x_{k+2}^2)}z^{\tau}.
$$
Hence by~(\ref{braid'}) we see that for any $z\in
L(\underline{i})$,
$$(\Phi_k\Phi_{k+1}\Phi_k-\Phi_{k+1}\Phi_k\Phi_{k+1})
\frac{1}{(x_k^2-x_{k+1}^2)(x_k^2-x_{k+2}^2)(x_{k+1}^2-x_{k+2}^2)}z^{\tau}=0,
$$
A tedious calculation shows that
$$
\Phi_k\Phi_{k+1}\Phi_k-\Phi_{k+1}\Phi_k\Phi_{k+1}=
(s_ks_{k+1}s_k-s_{k+1}s_ks_{k+1})(x_k^2-x_{k+1}^2)(x_k^2-x_{k+2}^2)(x_{k+1}^2-x_{k+2}^2).
$$
Therefore we obtain that if $j_{k}\neq j_{k+2}$ then
$$
s_ks_{k+1}s_kz^{\tau}=s_{k+1}s_ks_{k+1}z^{\tau},\quad \text{ for
any }z\in L(\underline{i}).
$$
Now assume $j_k=j_{k+2}$, then by Lemma~\ref{lem:restr.1} we have
either $j_{k}=j_{k+2}=0,j_{k+1}=1$ or $j_{k}=j_{k+2}=\frac{p-3}{2},
j_{k+1}=\frac{p-1}{2}$. Hence $s_{k}=\Xi_k$ and $s_{k+1}=\Xi_{k+1}$
on $L(\underline{i})^{\tau}$. We see from the proof of
Lemma~\ref{lem:restr.1} that $s_ks_{k+1}s_k=s_{k+1}s_ks_{k+1}$.
Therefore $D^{\underline{i}}$ affords a $\mhcn$-module by the
formula~(\ref{actionformula}).

Suppose $N$ is a nonzero irreducible submodule of
$D^{\underline{i}}$, then $N_{\underline{j}}\neq 0$ for some
$\underline{j}\in\I^n$. This implies
$(D^{\underline{i}})_{\underline{j}}\neq 0$ and hence
$\underline{j}\sim\underline{i}$. Since
$\tau\cdot\underline{i}\sim\underline{i}\sim\underline{j}$, by
Remark~\ref{rem:admissible} we see that
$N_{\tau\cdot\underline{i}}\neq 0$ for all $\tau\in
P_{\underline{i}}$. Observe that
$(D^{\underline{i}})_{\tau\cdot\underline{i}}\cong
L(\tau\cdot\underline{i})$ is irreducible as a $\mpcn$-module for
$\tau\in P_{\underline{i}}$. Therefore
$N_{\tau\cdot\underline{i}}=(D^{\underline{i}})_{\tau\cdot\underline{i}}$
for $\tau\in P_{\underline{i}}$ and hence $N=D^{\underline{i}}$.
This means $D^{\underline{i}}$ is irreducible.

We shall show that $D^{\underline{i}}$ has the same type as
$L(\underline{i})$. Suppose
$\Psi\in\text{End}_{\mhcn}(D^{\underline{i}})$.  Note that for
each $\tau\in P_{\underline{i}}$ and $1\leq k\leq n-1$,  if $s_k$
is admissible with respect to $\tau\cdot\underline{i}$, then for
any $z\in L(\underline{i})$,
\begin{align}
\Omega_k\Psi(z^{s_k\tau})=\Psi(\Omega_kz^{s_k\tau})=\Psi(s_kz^{\tau}-\Xi_kz^{\tau})
=s_k\Psi(z^{\tau})-\Xi_k\Psi(z^{\tau}).\label{Psisk}
\end{align}
Since $s_k$ is admissible with respect to
$\underline{j}:=\tau\cdot\underline{i}$, $j_k\neq j_{k+1}\pm1$ and
hence $\Omega_k$ acts as a nonzero scalar on
$L(\underline{i})^{s_k\tau}$. By~(\ref{Psisk}) we see that
$\Psi(z^{s_k\tau})$ is uniquely determined by $\Psi(z^{\tau})$ for
any $\tau\in P_{\underline{i}}$. Since each $\tau$ can be written as
$\tau=s_{k_t}\cdots s_{k_1}$ so that $s_{k_l}$ is admissible with
respect to $s_{k_{l-1}}\cdots s_{k_1}\cdot\underline{i}$, we deduce
$\Psi(z^{\tau})$ is uniquely determined by $\Psi(z)$ for any $z\in
L(\underline{i})$. Therefore $\Psi$ is uniquely determined by its
restriction to the $\mpcn$-submodule $L(\underline{i})$. Clearly the
image of restriction of $\Psi$ to $L(\underline{i})$ is contained in
$L(\underline{i})$ by Lemma~\ref{lem:lambdaP}. This implies
\begin{align}
\dim_{\F}{\rm End}_{\mhcn}(D^{\underline{i}})\leq \dim_{\F}{\rm
End}_{\mpcn}(L(\underline{i})).\label{dimless}
\end{align}
One the other hand, it is routine to check that each
$\mathcal{P}_n^{\mathfrak{c}}$-endomorphism
$\rho:L(\underline{i})\rightarrow L(\underline{i})$ induces a
$\mhcn$-endomorphism $\oplus_{\tau\in
P_{\underline{i}}}\rho^{\tau}:D^{\underline{i}}\rightarrow
D^{\underline{i}}$, where $\rho^{\tau}(z^{\tau})=(\rho(z))^{\tau}$.
Therefore
\begin{align}
\dim_{\F}{\rm End}_{\mhcn}(D^{\underline{i}})\geq \dim_{\F}{\rm
End}_{\mpcn}(L(\underline{i})).\notag
\end{align}
This together with ~(\ref{dimless}) shows $\dim_{\F}{\rm
End}_{\mhcn}(D^{\underline{i}})= \dim_{\F}{\rm
End}_{\mpcn}(L(\underline{i}))$  and hence $D^{\underline{i}}$ has
the same type as $\mpcn$-module $L(\underline{i})$.

(2). If $D^{\underline{i}}\cong D^{\underline{j}}$, then
$(D^{\underline{i}})_{\underline{j}}\neq 0$ and hence
$\underline{i}\sim\underline{j}$. Conversely, by
Lemma~\ref{lem:lambdaP}, there exists $\sigma\in P_{\underline{i}}$
such that $\underline{j}=\sigma\cdot\underline{i}$. By
Remark~\ref{rem:Ltau}, we have $L(\underline{j})\cong
L(\underline{i})^{\sigma}$ and hence there exists a linear map
$\phi:L(\underline{j})\rightarrow L(\underline{i})$ such that the
map $L(\underline{j})\rightarrow L(\underline{i})^{\sigma},~
u\mapsto (\phi(u))^{\sigma}$ is a
$\mathcal{P}_n^{\mathfrak{c}}$-isomorphism. For each $\pi\in
P_{\underline{j}}$, set
\begin{align}
\phi^{\pi}: L(\underline{j})^{\pi}&\longrightarrow
L(\underline{i})^{\pi\sigma}\notag\\
u^{\pi}&\mapsto (\phi(u))^{\pi\sigma}.\notag
\end{align}
It is routine to check that
$$
\oplus_{\pi\in P_{\underline{j}}}\phi^{\pi}:D^{\underline{j}}
\longrightarrow D^{\underline{i}}
$$
is a nonzero $\mhcn$-homomorphism. This means
$D^{\underline{i}}\cong D^{\underline{j}}$  since both of them are
irreducible.

(3). Suppose $M\in\operatorname{Rep}_{\I}\mhcn$ is irreducible
completely splittable with $M_{\underline{i}}\neq 0$ for some
$\underline{i}\in\I^n$.  By Proposition~\ref{prop:equiv.cond.},
there exists a $\mpcn$-isomorphism $\psi:
M_{\underline{i}}\rightarrow L(\underline{i})$. By
Lemma~\ref{lem:newbij.}, for each $\tau\in P_{\underline{i}}$, there
exists a bijection $\widehat{\Phi}_{\tau}:
M_{\underline{i}}\rightarrow M_{\tau\cdot\underline{i}}$.  Now for
$\tau\in P_{\underline{i}}$, define
$$
\psi^{\tau}: L(\underline{i})^{\tau}\longrightarrow
M_{\tau\cdot\underline{i}},\quad z^{\tau}\mapsto
\widehat{\Phi}_{\tau}(\psi(z)).
$$
By Lemma~\ref{lem:newbij.}, the bijection $\widehat{\Phi}_{\tau}$
satisfies $
\widehat{\Phi}_{\tau}x_k=x_{\tau(k)}\widehat{\Phi}_{\tau},
\widehat{\Phi}_{\tau}c_k=c_{\tau(k)}\widehat{\Phi}_{\tau}$ for
$1\leq k\leq n$. Hence for $z\in L(\underline{i}), \tau\in
P_{\underline{i}}$ and $1\leq k\leq n$,
\begin{align}\psi^{\tau}(x_kz^{\tau})&=\psi^{\tau}((x_{\tau^{-1}(k)}z)^{\tau})
=\widehat{\Phi}_{\tau}(\psi(x_{\tau^{-1}(k)}z))\notag\\
&=\widehat{\Phi}_{\tau}(x_{\tau^{-1}(k)})\psi(z)=x_k\widehat{\Phi}_{\tau}(\psi(z))
=x_k\psi^{\tau}(z^{\tau})\notag.
\end{align}
Similarly one can show that
$\psi^{\tau}(c_kz^{\tau})=c_k\psi^{\tau}(z^{\tau})$. Therefore
$\psi^{\tau}$ is a $\mpcn$-homomorphism. By
Proposition~\ref{prop:restr.2} we have $W(\mhcn)\subseteq
W^{\prime}(\mhcn)$ and hence $\underline{i}\in W'(\mhcn)$. By the
fact that $\psi^{\tau}$ is a $\mpcn$-module homomorphism for each
$\tau\in P_{\underline{i}}$, one can easily check that
$$
\oplus_{\tau\in P_{\underline{i}}}\psi^{\tau}:
D^{\underline{i}}\longrightarrow M
$$
is a $\mhcn$-module isomorphism.
\begin{rem} Observe that Theorem~\ref{thm:Classficiation} confirms a
slightly modified version of~\cite[Conjecture 52]{Le}. Leclerc
defined a completely splittable representation to be one on which
the $x_k^2, 1\leq k\leq n$ act semisimply.
\end{rem}
%
\end{proof}
By Proposition~\ref{prop:restr.2} we have $W(\mhcn)\subseteq
W^{\prime}(\mhcn)$. By Theorem~\ref{thm:Classficiation} we obtain
the following.
\begin{cor} \label{cor:weigdes}We have
$ W(\mhcn)=W^{\prime}(\mhcn). $
\end{cor}

\section{A diagrammatic classification}\label{combinatorics}
In this section, we shall give a reinterpretation of irreducible
completely splittable $\mhcn$-modules in terms of Young diagrams.

Let $\lambda=(\lambda_1,\ldots, \lambda_l)$ be a partition of the
integer $|\lambda|=\lambda_1+\cdots+\lambda_l$, where $\lambda_1\geq
\dots \geq \lambda_l \geq 1$. Denote by $l(\lambda) $ the number of
nonzero parts in $\lambda$. It is known that the partition $\lambda$
can be drawn as Young diagrams.

A strict partition $\lambda$ (i.e. with distinct parts) can  be
identified with the shifted Young diagram which is obtained from the
ordinary Young diagram by shifting the $k$th row to the right by
$k-1$ squares, for all $k
> 1$. For example, let $\lambda=(4,2,1)$, the corresponding shifted Young
diagram is
$$
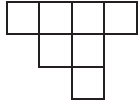
\young(\,\,\,\,,:\,\,,::\,)
$$

From now on, we shall always identify strict partitions with their
shifted Young diagrams. If $\lambda$ and $\mu$ are strict
partitions such that $\mu_k\leq \lambda_k$ for all $k$, we write
$\mu\subseteq\lambda$. A \textit{skew shifted Young diagram}
$\lambda/\mu$ is defined to be the diagram obtained by removing
 the shifted Young diagram $\mu$ from $\lambda$ for some strict
partitions $\mu\subseteq\lambda$ (see examples below). Note that any
skew shifted Young diagram is a union of connected components.
Moreover, different pairs of strict partitions may give an identical
skew shifted Young diagram.

A \textit{ placed skew shifted Young diagram } $(c,\lambda/\mu)$
consists of a skew shifted Young diagram $\lambda/\mu$ and a
\textit{ content function } $c:\{\text{ boxes of
}\lambda/\mu\}\longrightarrow \mathbb{Z}_+$ which is increasing
from southwest to northeast in each connected component of
$\lambda/\mu$ and satisfies the following:
\begin{enumerate}
\item $c(A)=c(B)$, \text{ if and only if  }$A$
\text{ and  }$B$ \text{ are on the same diagonal},\notag\\
\item $c(A)=c(B)+1$, \text{ if and only if  }$A
$\text{ and  }$B$ \text{ are on the adjacent diagonals},\notag\\
\item $c(A)=0$,  \text{ if the box }$A$ \text{ is located in
}$\lambda/\mu$\text{ as } $\young(A\,,:\,)$ \text{ and there is no
box below }$A$.\notag
\end{enumerate}
A standard tableau of the shape $\lambda/\mu$ is a labeling of the
skew shifted Young diagram $\lambda/\mu$ with the numbers $1,
2,\ldots, |\lambda|-|\mu|$ such that the numbers strictly increase
from left to right along each row and down each column. If $T$ is a
tableau of the shape $\lambda/\mu$, denote by $T(k)$ the box of
$\lambda/\mu$ labeled by $k$ in $T$ for  $1\leq k\leq
|\lambda|-|\mu|$.

\begin{example} Let $\lambda=(9,8,5,2,1)$ and $\mu=(7,5,4)$. The
skew shifted Young diagram $\lambda/\mu$ is as follows:
$$
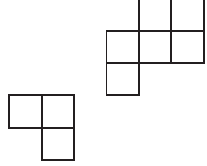
\young(:::::::\,\,,::::::\,\,\,,::::::\,,:::\,\,,::::\,)
$$

A standard tableau $T$ of shape $\lambda/\mu$:
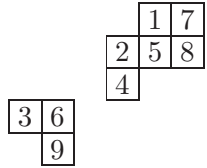
\begin{align}
\young(:::::::17,::::::258,::::::4,:::36,::::9)\notag
\end{align}
A placed skew shifted Young diagram $(c,\lambda/\mu)$:
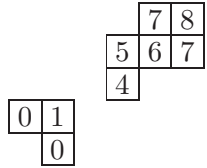
\begin{align}
&\young(:::::::78,::::::567,::::::4,:::01,::::0)\notag
\end{align}
satisfying $(c(T(1)),\ldots, c(T(9)))=(7,5,0,4,6,1,8,7,0)$.
\end{example}
\begin{rem}\label{rem:shiftedY} For each shifted Young diagram $\lambda$,
there exists one and only one content function $c_{\lambda}$ defined
by setting the contents of boxes on the first diagonal to be $0$.
Moreover, each placed skew shifted Young diagram
 can be obtained by removing a shifted
Young diagram $\mu$ associated with $c_{\mu}$ from the shifted Young
diagram $\lambda$ associated with $c_{\lambda}$ for some strict
partitions $\mu\subseteq\lambda$.

If we modify the definition of placed skew shifted Young diagram by
allowing non-integer contents and by adding that the difference
between contents of two boxes is an integer if and only if they
belong to the same connected component, then placed skew shifted
Young diagrams may be used for the study of ``non-integral"
$\mhcn$-modules.
\end{rem}
%
For each $n\in\mathbb{Z}_+$, denote by $\mathcal{PS}(n)$ the set
of placed skew shifted Young diagrams with $n$ boxes and set
$$
\Delta(n):=\{((c,\lambda/\mu),T)~|~(c,\lambda/\mu)\in\mathcal{PS}(n),
T \text{ is a standard tableau of shape }\lambda/\mu\}.
$$
For each $((c,\lambda/\mu),T)\in\Delta(n)$, define
\begin{align}
\mathcal{F}((c,\lambda/\mu),T):=(c(T(1)),\ldots,c(T(n))).\label{mathcalF}
\end{align}

A vector $\underline{i}\in\mathbb{Z}_+^n$ is said to be {\em
splittable} if it satisfies that if $i_k=i_l=u$ for some $1\leq
k<l\leq n$ then $u=0$ implies $1\in\{i_{k+1},\ldots,i_{l-1}\}$ and
$u\geq 1$ implies
$\{i_k-1,i_k+1\}\subseteq\{i_{k+1},\ldots,i_{l-1}\}$. Denote by
$\nabla(n)$ the subset of $\I^n$ consisting of splittable vectors.
\begin{lem}\label{lem:plYcont} The map $\mathcal{F}$
in~(\ref{mathcalF}) sends $\Delta(n)$ to $\nabla(n)$.
\end{lem}
\begin{proof} Suppose
$((c,\lambda/\mu),T)\in\Delta(n)$, we need to show that
$(c(T(1)),\ldots,c(T(n)))$ is splittable.  Suppose
$c(T(k))=c(T(l))=u$  for some $1\leq k<l\leq n$. Without loss of
generality, we can assume that
$u\notin\{c(T(k+1)),\ldots,c(T(l-1))\}$. This means that there is a
configuration in $T$ of the form
$$
\young(k,:l).
$$
Since $(c,\lambda/\mu)$ is a placed skew shifted Young diagram and
$T$ is standard, there exists a box labeled by $j$ located in $T$ as
in the configuration
$$
\young(kj,:l)
$$
for some $k<j<l$ and moreover $c(T(j))=u+1$. If $u=0$, then there is
no box below the box labeled by $k$ and $c(T(j))=1$. This implies
$1\in\{c(T(k+1)),\ldots,c(T(l-1))\}$. If $u\geq 1$, then there is a
box labeled by $t$ below the box labeled by $k$ and $c(T(t))=u-1$,
that is, $T$ contains the following configuration
$$
\young(ks,tl)
$$
for some $k<s\neq t<l$. This implies that
$\{u-1,u+1\}\subseteq\{c(T(k+1)),\ldots,c(T(l-1))\}$. Hence
$(c(T(1)),\ldots,c(T(n)))\in\nabla(n)$.
\end{proof}

Given $\underline{i}\in \nabla(n)$, by induction on $n$ we can
produce a pair
$\mathcal{G}(\underline{i})=((c,\lambda/\mu),T)\in\Delta(n)$
satisfying
 $c(T(k))=i_k$ for $1\leq k\leq n$ .
 If $n=1$, let $\mathcal{G}(\underline{i})$ be a box labeled by $1$
 with content $i_1$.  Assume inductively
that $\mathcal{G}(\underline{i}^{\prime})=((c',\lambda'/\mu'),T')\in
\Delta(n-1)$ is already defined, where
$\underline{i}^{\prime}=(i_1,\ldots,i_{n-1})\in \nabla(n-1)$. Set
$u=i_n$.

\noindent\textit{Case 1}:$(c',\lambda'/\mu')$ contains neither a box
with content $u-1$ nor a box with content $u+1$. Adding a new
component consisting of one box labeled by $n$ with content $u$ to
$T'$, we obtain a new placed skew shifted Young diagram
$(c,\lambda/\mu)$ and a standard tableau $T$ of shape $\lambda/\mu$.
Set $\mathcal{G}(\underline{i})=((c,\lambda/\mu), T)$.

\noindent\textit{Case 2}:  $(c',\lambda'/\mu')$ contains boxes with
content $u-1$ but no box with content $u+1$. This implies
$u+1\notin\{i_1,\ldots,i_n\}$. Since $(i_1,\ldots,i_n)$ is
splittable,  $u$ does not appear in $\underline{i}'$ and hence $u-1$
appears only once in $\underline{i}'$ by Lemma~\ref{lem:restr.2}.
Therefore there is no box of content $u$ and only one box denoted by
$A$ with content $u-1$ in $((c',\lambda'/\mu'),T')$. So we can add a
new box labeled by $n$ with content $u$ to the right of $A$ to
obtain a new tableau $T$ of shape $(c,\lambda/\mu)$. Set
$\mathcal{G}(\underline{i})=((c,\lambda/\mu), T)$ . Observe that
there is no box above $A$ in the column containing $A$ since there
is no box of content $u$ in $((c',\lambda'/\mu'),T')$. Hence
$\mathcal{G}(\underline{i})\in\Delta(n)$.

\noindent\textit{Case 3}: $(c',\lambda'/\mu')$ contains boxes with
content $u+1$ but no box with content $u-1$. This implies
$u-1\notin\{i_1,\ldots,i_n\}$. Since $(i_1,\ldots,i_n)$ is
splittable,  $u$ does not appear in $\underline{i}'$ and hence $u+1$
appears only once in $\underline{i}'$ by Lemma~\ref{lem:restr.2}.
Therefore $((c',\lambda'/\mu'),T')$ contains only one box denoted by
$B$ with content $u+1$ and no box with content $u$. This means there
is no box below $B$ in $((c',\lambda'/\mu'),T')$. Adding a new box
labeled by $n$ with content $u$ below $B$, we obtain a new tableau
$T$ of shape $(c,\lambda/\mu)$. Set
$\mathcal{G}(\underline{i})=((c,\lambda/\mu), T)$. Clearly
$\mathcal{G}(\underline{i})\in\Delta(n)$.

\noindent\textit{Case 4}: $(c',\lambda'/\mu')$ contains boxes with
contents $u-1$ and $u+1$. Let $C$ and $D$ be the last boxes on the
diagonals with content $u-1$ and $u+1$, respectively. Suppose $C$ is
labeled by $s$ and $D$ is labeled by $t$. Then $i_s=u-1, i_t=u+1$
and moreover $u-1\notin\{i_{t+1},\ldots,i_{n-1}\},
u+1\notin\{i_{s+1},\ldots,i_{n-1}\}$. Since $i_n=u$, by
Lemma~\ref{lem:restr.2} we see that
$u\notin\{i_{t+1},\ldots,i_{n-1}\}$ and
$u\notin\{i_{s+1},\ldots,i_{n-1}\}$.  This implies that there is no
box below $C$ and no box to the right of $D$ in
$((c',\lambda'/\mu'),T')$. Moreover $C$ and $D$ must be of the
following shape
$$
\young(:C,D).
$$
Add a new box labeled by $n$ to the right of $D$ and below $C$ to
obtain a new tableau $T$ of shape $(c,\lambda/\mu)$. Set
$\mathcal{G}(\underline{i})=((c,\lambda/\mu), T)$. It is clear that
$\mathcal{G}(\underline{i})\in\Delta(n)$.

Therefore we obtain a map
\begin{align}
\mathcal{G}: \nabla(n)\longrightarrow \Delta(n)\label{mathcalG}
\end{align}
satisfying $\underline{i}=(c(T(1)),\ldots, c(T(n)))$ if
$\mathcal{G}(\underline{i})=((c,\lambda/\mu), T)$. In this case,
we will say that $\mathcal{G}(\underline{i})$ affords the placed
skew shifted Young diagram $(c,\lambda/\mu)$.

\begin{example}\label{example1} Suppose $n=5$. The map $\mathcal{G}$ maps
the splittable vector $\underline{i}=(1,2,0,1,0)\in\nabla(5)$ to the
pair $((c,\lambda/\mu),T)\in\Delta(5)$ with
\begin{align}
(c,\lambda/\mu)=\young(12,01,:0),\qquad T=\young(12,34,:5).\notag
\end{align}

\end{example}
%
%
\begin{prop}\label{onto}
The map $\mathcal{G}$ in~(\ref{mathcalG}) is a bijection from
$\nabla(n)$ to $\Delta(n)$ with inverse $\mathcal{F}$.
\end{prop}
%
\begin{proof} It is clear that
$\mathcal{F}\circ\mathcal{G}(\underline{i})=\underline{i}$ for any
$\underline{i}\in \nabla(n)$ by~(\ref{mathcalG}). It remains to
prove that
$\mathcal{G}\circ\mathcal{F}((c,\lambda/\mu),T)=((c,\lambda/\mu),T)$
for any $((c,\lambda),T)\in\Delta(n)$. We shall proceed by induction
on $n$. Denote by $A$ the box labeled by $n$ in $T$. Removing $A$
from $(c,\lambda/\mu)$ and $T$, we obtain a new pair
$((c',\lambda'/\mu'),T')\in \Delta(n-1)$. By induction we see that
$$\mathcal{G}\circ\mathcal{F}(((c',\lambda'/\mu'),T'))
=((c',\lambda'/\mu'),T').$$ This means
$\mathcal{G}((c(T(1)),\ldots,c(T(n-1))))=((c',\lambda'/\mu'),T')$.
By adding a box denoted by $B$ labeled by $n$ with content $c(T(n))$
to $((c',\lambda'/\mu'),T')$ by the procedure for defining
$\mathcal{G}$, we obtain $\mathcal{G}((c(T(1)),\ldots,c(T(n))))$.
One can check case by case that $B$ coincides with $A$ and hence
$\mathcal{G}((c(T(1)),\ldots,c(T(n))))=((c,\lambda/\mu),T)$. This
means $\mathcal{G}\circ\mathcal{F}((c,\lambda/\mu),T)
=\mathcal{G}((c(T(1)),\ldots,c(T(n))))=((c,\lambda/\mu),T)$.
\end{proof}
\begin{lem}\label{lem:pocomb.des2}
Suppose $\underline{i},\underline{j}\in \nabla(n)$. Then
$\underline{i}\sim\underline{j}$ if and only if
$\mathcal{G}(\underline{i})$ and $\mathcal{G}(\underline{j})$ afford
the same placed skew shifted Young diagram.
\end{lem}
\begin{proof} Suppose that
$\mathcal{G}(\underline{i})$ and $\mathcal{G}(\underline{j})$ afford
the same placed skew shifted Young diagram $(c,\lambda/\mu)$. This
means that there exist standard tableaux $T$ and $S$ of shape
$\lambda/\mu$ such that $(i_1,\ldots,i_n)=(c(T(1)),\ldots,c(T(n)))$
and $(j_1,\ldots,j_n)=(c(S(1)),\ldots,c(S(n)))$. We shall prove
$\underline{i}\sim\underline{j}$ by induction on $n$. Let $T_0$ be
the tableau of shape $\lambda/\mu$ obtained by filling in the
numbers $1,\ldots,n$ from left to right along the rows, starting
from the first row and going down. Clearly $T_0$ is standard and
hence we have $(c(T_0(1)),\ldots,c(T_0(n)))\in \nabla(n)$ by
Lemma~\ref{lem:plYcont}. Let $A$ be the last box of the last row of
$\lambda/\mu$. Then in $T_0$, $A$ is occupied by $n$. Suppose in
$T$, $A$ is occupied by the number $k$. Clearly $k+1$ and $k$ do not
lie on adjacent diagonals in $T$, hence the transposition $s_k$ is
admissible with respect to $\underline{i}$. So we can apply $s_k$ to
swap $k$ and $k+1$, then to swap $k+1$ and $k+2$, and finally we
obtain a new standard tableau $T_1$ in which $A$ is occupied by $n$
and moreover $\underline{i}\sim(c(T_1(1)),\ldots,c(T_1(n)))$.
Observe that $A$ is occupied by $n$ in both $T_1$ and $T_0$. Hence
both $\mathcal{G}((c(T_1(1)),\ldots,c(T_1(n-1))))$ and
$\mathcal{G}((c(T_0(1)),\ldots,c(T_0(n-1))))$ contains the placed
skew shifted Young diagram obtained by removing $A$ from
$(c,\lambda/\mu)$. By induction we have
$(c(T_1(1)),\ldots,c(T_1(n-1)))$ is equivalent to
$(c(T_0(1)),\ldots,c(T_0(n-1)))$ and then
$(c(T_1(1)),\ldots,c(T_1(n)))\sim(c(T_0(1)),\ldots,c(T_0(n)))$.
Therefore we obtain $\underline{i}\sim(c(T_0(1)),\ldots,c(T_0(n)))$.
Similarly, we can apply the above argument to $P$ to obtain
$\underline{j}\sim(c(T_0(1)),\ldots,c(T_0(n)))$. Hence
$\underline{i}\sim\underline{j}$.

Conversely, it suffices to check the case when
$\underline{j}=s_k\cdot\underline{i}$, where $s_k$ is admissible
with respect to $\underline{i}$ for some $1\leq k\leq n-1$. This is
reduced to show that $\mathcal{G}((i_1,\ldots,i_{k-1},i_k,i_{k+1}))$
and $\mathcal{G}((i_1,\ldots,i_{k-1},i_{k+1},i_k))$ afford the same
placed skew shifted Young diagram. Suppose
$\mathcal{G}((i_1,\ldots,i_{k-1}))$ affords the placed skew shifted
Young diagram $(c,\lambda/\mu)$.  Since $s_k$ is admissible with
respect to $\underline{i}$, we have $i_k\neq i_{k+1}\pm1$ and hence
the resulting placed skew shifted Young diagram obtained by adding
two boxes with contents $i_k, i_{k+1}$ in two different orders to
$(c,\lambda/\mu)$ via the procedure for defining $\mathcal{G}$ are
identical.

\end{proof}

\subsection{A diagrammatic classification for
$p=0$}\label{combinatorics0}

In this subsection, we assume that $p=0$.
By Proposition~\ref{prop:restr.2},  $W'(\mhcn)$ consists of all
splittable vectors in $\mathbb{Z}_+^n$ and hence
$W'(\mhcn)=\nabla(n)$.  Recall the definition of $\mhcn$-module
$D^{\underline{i}}$ from Theorem~\ref{thm:Classficiation} for
$\underline{i}\in W'(\mhcn)$. Suppose
$(c,\lambda/\mu)\in\mathcal{PS}(n)$, by Proposition~\ref{onto} there
exists $\underline{i}\in W'(\mhcn)$ such that
$\mathcal{G}(\underline{i})$ affords $(c,\lambda/\mu)$. Let
\begin{align}
D(c,\lambda/\mu)=D^{\underline{i}}.\label{D(c,lambda)}
\end{align}
Note that if $\underline{j}\in W'(\mhcn)$ satisfies that
$\mathcal{G}(\underline{j})$ also affords $(c,\lambda/\mu)$, then
$\underline{i}\sim\underline{j}$ by Lemma~\ref{lem:pocomb.des2} and
hence the $\mhcn$-module $D(c,\lambda/\mu)$ is unique (up to
isomorphism) by Theorem~\ref{thm:Classficiation}(2).

For $(c,\lambda/\mu)\in\mathcal{PS}(n)$, denote
$\gamma_0(c,\lambda/\mu)$ by the number of boxes with content zero
in $(c,\lambda/\mu)$ and let $f^{\lambda/\mu}$ be the number of
standard tableaux of shape $\lambda/\mu$.

The following is a Young diagrammatic reformulation of
Theorem~\ref{thm:Classficiation} for $p=0$. Note that it confirms
\begin{thm}\label{thm:dimandparam.} Suppose that
$(c,\lambda/\mu)\in\mathcal{PS}(n)$ and write
$\gamma_0=\gamma_0(c,\lambda/\mu)$.

(1) $D(c,\lambda/\mu)$ is type
 $\texttt{M}$ if $\gamma_0$ is even and is type
$\texttt{Q}$ if $\gamma_0$ is odd. Moreover, $\dim
D(c,\lambda/\mu)=2^{n-\lfloor\frac{\gamma_0}{2}\rfloor}f^{\lambda/\mu}$.

(2) The $\mhcn$-modules $D(c,\lambda/\mu)$ for
$(c,\lambda/\mu)\in\mathcal{PS}(n)$ form a complete set of pairwise
non-isomorphic irreducible completely splittable $\mhcn$-modules in
$\operatorname{Rep}_{\I}\mhcn$.
\end{thm}
\begin{proof}
(1) Suppose  $(c,\lambda/\mu)\in\mathcal{PS}(n)$ and
$\mathcal{G}(\underline{i})$ affords $(c,\lambda/\mu)$ for some
$\underline{i}\in W(\mhcn)$.  By Proposition~\ref{onto}, we have
$\underline{i}=(c(T(1)),\ldots, c(T(n)))$ and hence the number of
$1\leq k\leq n$ with $i_k=0$ is equal to $\gamma_0$. This together
with Lemma~\ref{lem:irrepPn} and Theorem~\ref{thm:Classficiation}(1)
shows that $D^{\underline{i}}$ is type  $\texttt{M}$ if $\gamma_0$
is even and is type $\texttt{Q}$ if $\gamma_0$ is odd. Denote by
$|P_{\underline{i}}|$ the number of elements contained in
$P_{\underline{i}}$. By Lemma~\ref{lem:pocomb.des2}, there exists a
one-to-one correspondence between the set of weights in $W(\mhcn)$
equivalent to $\underline{i}$ and the set of standard tableaux of
shape $\lambda/\mu$. Hence $f^{\lambda/\mu}=|P_{\underline{i}}|$ by
Lemma~\ref{lem:lambdaP}. This together with Lemma~\ref{lem:irrepPn}
and Theorem~\ref{thm:Classficiation} shows that
$$\dim D(c,\lambda/\mu)=\dim D^{\underline{i}}=
2^{n-\lfloor\frac{\gamma_0}{2}\rfloor}|P_{\underline{i}}|=2^{n-\lfloor\frac{\gamma_0}{2}\rfloor}f^{\lambda/\mu}.$$

(2) It follows from Proposition~\ref{onto},
 Lemma~\ref{lem:pocomb.des2} and Theorem~\ref{thm:Classficiation}.

\end{proof}
\subsection{A diagrammatic classification for $p\geq
3$}\label{combinatoricsp}

In this subsection, we assume $p\geq 3$.
Set
\begin{align}
\ds W_1(\mhcn)=\Big\{\underline{i}\in
W(\mhcn)~|~&i_k-1\in\{i_{k+1},\ldots,i_{l-1}\}
\text{ whenever }\notag\\
& 1\leq i_k=i_l\leq \frac{p-3}{2}\text{ with }1\leq k<l\leq n\Big\}\notag,\\
\ds W_2(\mhcn)=\Big\{\underline{i}\in W(\mhcn)~|~&\text{there
exist }1\leq
k<l\leq n \text{ such that }1\leq i_k=i_l\leq\frac{p-3}{2}\notag,\\
&i_k-1\notin\{i_{k+1},\ldots,i_{l-1}\}\Big\}\notag.
\end{align}
Observe that $W(\mhcn)$ is the disjoint union of $W_1(\mhcn)$ and
$W_2(\mhcn)$. Moreover if $\underline{i}\in W_k(\mhcn)$ and
$\underline{j}\sim\underline{i}$, then $\underline{j}\in
W_k(\mhcn)$ for $k=1, 2$.  For each $u\in\mathbb{Z}_+$ and $m\geq
1$, let $\mathcal{PS}_u(m)$ be the set of placed skew shifted
Young diagrams $(c,\lambda/\mu)$ with $m$ boxes such that the
contents of boxes of $\lambda/\mu$ are smaller than or equal to
$u$. For $n\in\mathbb{Z}_+$, set
\begin{align}
\Delta_1(n)=\{((c,\lambda/\mu),T)~|&~(c,\lambda/\mu)\in\mathcal{PS}_{\frac{p-1}{2}}(n),
T \text{ is a standard tableau of shape  } \lambda/\mu \}.\notag
\end{align}

By Lemma~\ref{lem:restr.2}, we see that
$W_1(\mhcn)\subseteq\nabla(n)$.
\begin{prop}\label{prop:ppcomb0}
The restriction of the  map $\mathcal{G}$ in~(\ref{mathcalG}) to
$W_1(\mhcn)$ gives a bijection
$\mathcal{G}_1:W_1(\mhcn)\rightarrow\Delta_1(n)$. Moreover,
$\underline{i}\sim\underline{j}\in W_1(\mhcn)$ if and only if
$\mathcal{G}_1(\underline{i})$ and $\mathcal{G}_1(\underline{j})$
afford the same placed skew shifted Young diagram.
\end{prop}
\begin{proof}
Observe that $W_1(\mhcn)$ can be identified with the subset of
$\nabla(n)$ consisting of splittable vectors whose parts are less
than or equal to $\frac{p-1}{2}$.  Hence by Proposition~\ref{onto},
the restriction $\mathcal{G}_1$ of the map $\mathcal{G}$ establishes
a bijection between $W_1(\mhcn)$ and $\Delta_1(n)$. Now the rest of
the Proposition follows from Lemma~\ref{lem:pocomb.des2}.
\end{proof}
%
%
For each $\underline{i}\in W_2(\mhcn)$, denote by $1\leq
u_{\underline{i}}\leq\frac{p-3}{2}$ the minimal integer such that
there exist $1\leq k<l\leq n$ satisfying $i_k=i_l=u_{\underline{i}}$
and $u_{\underline{i}}-1\notin\{i_{k+1},\ldots,i_{l-1}\}$. By the
definition of $W_2(\mhcn)$, we see that $u_{\underline{i}}$ always
exists.

\begin{lem}\label{lem:ppcom1}
Let $\underline{i}\in W_2(\mhcn)$ and write $u=u_{\underline{i}}$.
\begin{enumerate}

 \item There exists a unique sequence of integers
  $1\leq r_0<r_1<\ldots<r_{\frac{p-3}{2}-u}
<q<t_{\frac{p-3}{2}-u}<\ldots<t_1<t_0\leq n$ such that
\begin{enumerate}
\item $i_q=\frac{p-1}{2}, i_{r_j}=i_{t_j}=u+j$ for $0\leq
j\leq\frac{p-3}{2}$,

\item $i_a\neq u-1$ for all $r_0\leq a\leq t_0$,

\item $i_b\leq u-1$ for all $b\neq r_0, r_1, \ldots,
r_{\frac{p-3}{2}-u} ,q, t_{\frac{p-3}{2}-u}, \ldots, t_1, t_0$.

\end{enumerate}

\item $\underline{i}\sim (\underline{i}',
u,u+1,\ldots,\frac{p-3}{2},\frac{p-1}{2},\frac{p-3}{2},\ldots,
u+1,u,u-1,\ldots, u-m)$ for some $\underline{i}'\in\nabla(n-p+2u-m)$
whose parts are less than $u$ and some $0\leq m\leq u$.
\end{enumerate}
\end{lem}
\begin{proof}
(1)  By Proposition~\ref{prop:restr.2}, there exists a sequence of
integers $r_0< r_1<\cdots< r_{\frac{p-3}{2}-u}< q<
t_{\frac{p-3}{2}-u}<\cdots< t_1<t_0$ such that $i_q=\frac{p-1}{2},
i_{r_j}=i_{t_j}=u+j$,  and $u+j$ does not appear between $i_{r_j}$
and $i_{t_j}$ in $\underline{i}$ for each $0\leq j\leq
\frac{p-3}{2}-u$. Hence it suffices to prove $(1)(c)$. Assume that
$i_b=u+k$ for some $0\leq k\leq \frac{p-3}{2}-u$ and some
$b\notin\{r_0,\ldots,r_{\frac{p-3}{2}-u},q,t_{\frac{p-3}{2}-u},\ldots,t_0\}$.
Since $u+k$ does not appear between $i_{r_k}$ and $i_{t_k}$,  we see
that either $b<r_k$ or $b>t_k$. Now assume $b<r_k$. Since
$i_b=u+k=i_{r_k}$, by Lemma~\ref{lem:restr.2} there exists $b_1$
with $b<b_1<r_k$ and $i_{b_1}=u+k+1$. Again since
$i_{b_1}=u+k+1=i_{r_{k+1}}$, using Lemma~\ref{lem:restr.2} there
exists $b_2$ with $b_1<b_2<r_{k+1}$ such that $i_{b_2}=u+k+2$.
Continuing in this way, we finally obtain an integer $f$ satisfying
$f<r_{\frac{p-3}{2}-u-1}$ and $i_f=\frac{p-3}{2}$. By
Lemma~\ref{lem:restr.2}, $\frac{p-1}{2}$  appears between $i_f$ and
$i_{r_{\frac{p-3}{2}-u}}$. So  $\frac{p-1}{2}$ appears at least
twice in $\underline{i}$ since $i_q=\frac{p-3}{2}$ and
$q>r_{\frac{p-3}{2}-u}$. This contradicts
Proposition~\ref{prop:restr.2}. An identical argument holds for the
case when $b>t_k$. Therefore $i_b\leq u-1$ for all
$b\notin\{r_0,\ldots,r_{\frac{p-3}{2}-u},q,t_{\frac{p-3}{2}-u},\ldots,t_0\}$.

(2) As shown in (1), there exists a sequence of integers
$r_0<r_1<\ldots<r_{\frac{p-3}{2}-u}
<q<t_{\frac{p-3}{2}-u}<\ldots<t_1<t_0$ such that $i_q=\frac{p-1}{2},
i_{r_j}=i_{t_j}=u+j$ for $0\leq j\leq\frac{p-3}{2}$. If $u-1$ does
not appear after $i_{t_0}$ in $\underline{i}$, then $i_a\leq u-2$
for all $a\in\{r_0, r_{0}+1,\ldots, n\}\setminus\{
r_0,r_1,\ldots,r_{\frac{p-3}{2}-u}, q, t_{\frac{p-3}{2}-u},
\ldots,t_0\}$ by $(1)(c)$. By applying admissible transpositions we
can swap $i_k$ with $i_l$ in $\underline{i}$ for all $k\in\{r_0,
r_{0}+1,\ldots, n\}\setminus\{ r_0,r_1,\ldots,r_{\frac{p-3}{2}-u},
q, t_{\frac{p-3}{2}-u}, \ldots,t_0\} $ and $l\in\{
r_0,r_1,\ldots,r_{\frac{p-3}{2}-u}, q, t_{\frac{p-3}{2}-u},
\ldots,t_0\}$. Finally we obtain an element of the form $(\ldots,
u,u+1,\ldots,\frac{p-3}{2},\frac{p-1}{2},\frac{p-3}{2},\ldots,
u+1,u)$.

Now assume $u-1$ appears after $i_{t_0}$ in $\underline{i}$. Since
$i_b\neq u$ for all $b>t_0$ by $(1)(c)$, we see that $u-1$ appears
at most once after $i_{t_0}$ in $\underline{i}$ by
Lemma~\ref{lem:restr.2}. Therefore there exists a unique $l_1>t_0$
such that $i_{l_1}=u-1$. If $u-2$ does not appear after $i_{l_1}$,
then $i_a\leq u-3$ for all $a\in\{r_0, r_{0}+1,\ldots,
n\}\setminus\{ r_0,r_1,\ldots,r_{\frac{p-3}{2}-u}, q,
t_{\frac{p-3}{2}-u}, \ldots,t_0, l_1\} $ by $(1)(c)$. Hence we can
apply admissible transpositions to $\underline{i}$ to obtain an
element of the form $(\ldots,
u,u+1,\ldots,\frac{p-3}{2},\frac{p-1}{2},\frac{p-3}{2},\ldots,
u+1,u, u-1)$.

Now assume  $u-2$ appears after $i_{l_1}$. Since $u-1$ appears
only once with $i_{l_1}=u-1$ after $i_{t_0}$ in $\underline{i}$,
we see that $i_b\neq u-1$ for all $b>l_1$ and hence $u-2$ appears
at most once after $i_{l_1}$ in $\underline{i}$ by
Lemma~\ref{lem:restr.2}. This means there exists a unique
$l_2>l_1$ such that $i_{l_2}=u-2$. By repeating the above process,
we arrive at the claim in (2).
\end{proof}

By Lemma~\ref{lem:ppcom1}, for $\underline{i}\in W_2(\mhcn)$,
there exists a unique vector $\widehat{\underline{i}}$ as follows
\begin{align}
\widehat{\underline{i}}=(i_1,\ldots,i_{r_0},\ldots,
\widehat{i}_{r_1},\ldots,\widehat{i}_q,\ldots,\widehat{i}_{t_1},
\ldots,\widehat{i}_{t_0},\ldots,i_n),\label{ihat}
\end{align}
which is obtained by removing
$i_{r_1},\ldots,i_{r_{\frac{p-3}{2}-u_{\underline{i}}}},i_q,i_{t_{\frac{p-3}{2}-u_{\underline{i}}}},
\ldots, i_{t_1}, i_{t_0}$ from $\underline{i}$.

\begin{lem}\label{lem:ppcom2}
The following holds for $\underline{i}, \underline{j}\in
W_2(\mhcn)$.
\begin{enumerate}

\item $\widehat{\underline{i}}$ has a unique part equal to
$u_{\underline{i}}$ and all other parts are less than
$u_{\underline{i}}$.

\item $\widehat{\underline{i}}$ is splittable.

\item $\widehat{\underline{i}}\sim\widehat{\underline{j}}$ if
$\underline{i}\sim\underline{j}$.
\end{enumerate}
\end{lem}
\begin{proof}
(1). It follows from the definition of $\widehat{\underline{i}}$.

(2). Suppose $\underline{i}\in W_2(\mhcn)$. By
Lemma~\ref{lem:ppcom1}, there exists a unique sequence of integers
$r_0<r_1<\ldots<r_{\frac{p-3}{2}-u}
<q<t_{\frac{p-3}{2}-u}<\ldots<t_1<t_0$ such that $i_q=\frac{p-1}{2},
i_{r_j}=i_{t_j}=u+j$ for $0\leq j\leq\frac{p-3}{2}$ and $i_a\neq
u-1, i_b\leq u-1$ for all $r_0\leq a\leq t_0$ and $b\neq r_0, r_1,
\ldots, r_{\frac{p-3}{2}-u} ,q, t_{\frac{p-3}{2}-u}, \ldots, t_1,
t_0$. Assume $i_k=i_l=v$ for some  $k<l\notin\{
r_1,\ldots,r_{\frac{p-3}{2}-u}, q, t_{\frac{p-3}{2}-u},
\ldots,t_0\}$. To show $\widehat{\underline{i}}$ is splittable, we
need to show that if $v=0$ then $1$ appears between $i_k$ and $i_l$
in $\widehat{\underline{i}}$ and if $1\leq v\leq u-1$ then $v-1$ and
$v+1$ appear between $i_k$ and $i_l$ in $\widehat{\underline{i}}$.
One can easily check the case when $v=0$. Now assume $v\geq1$. If
$1\leq v<u-1$, by the choice of $u$ there exist $k<s, t<l$ such that
$i_s=v-1, i_t=v+1$. Observe that $v-1<u-2,v+1\leq u-1$. Hence $s,
t\neq r_1,\ldots,r_{\frac{p-3}{2}-u}, q, t_{\frac{p-3}{2}-u},
\ldots,t_0$. This means $v-1$ and $v+1$ appear between $i_k$ and
$i_l$ in $\widehat{\underline{i}}$.  Now assume $v=u-1$. Since there
are no parts equal to $u$ before $i_{r_0}$ and after $i_{t_0}$ in
$\underline{i}$, it follows from Lemma~\ref{lem:restr.2} that $u-1$
appears at most once before $i_{r_0}$ and after $i_{t_0}$ in
$\underline{i}$, respectively. This together with the fact that
$i_a\neq u-1$ for $r_0<a<t_0$ shows  $a<r_0$ and $b>t_0$ . By the
choice of $u$, there exists $a<c<b$ such that $i_c=u-2$. This
together with $i_{r_0}=u$ shows that $v-1=u-2$ and $v+1=u$ appear
between $i_k$ and $i_l$ in $\widehat{\underline{i}}$.

(2). It suffices to check the case when
$\underline{j}=s_k\cdot\underline{i}\in W_2(\mhcn)$, where $s_k$
is admissible with respect to $\underline{i}$.  If $\{k,
k+1\}\cap\{r_1,\ldots,r_{\frac{p-3}{2}-u}, q, t_{\frac{p-3}{2}-u},
\ldots, t_1, t_0\}=\emptyset $, then
$\widehat{\underline{j}}=s_k\cdot \widehat{\underline{i}}$ and
hence $\widehat{\underline{j}}\sim\widehat{\underline{i}}$.
Otherwise we see that
$\widehat{\underline{j}}=\widehat{\underline{i}}$.
\end{proof}
 Suppose $1\leq
u\leq\frac{p-3}{2}$ and $(c,\lambda/\mu)\in\mathcal{PS}_{u}(m)$ for
some $m\in\mathbb{Z}_+$. Observe that there exists at most one box
with content $u$ in $(c,\lambda/\mu)$. Let us denote by
$\mathcal{PS}^*_{u}(m)\subseteq\mathcal{PS}_{u}(m)$ the subset
consisting of placed skew shifted Young diagrams which contain a
unique box of content $u$. Suppose
$(c,\lambda/\mu)\in\mathcal{PS}^*_{u}(m)$ and let
$A_{(c,\lambda/\mu)}$ be the unique box of content $u$. Add $p-2u-1$
boxes to the right of $A_{(c,\lambda/\mu)}$ in the row containing
$A_{(c,\lambda/\mu)}$ to obtain a skew shifted Young diagram denoted
by $\overline{\lambda/\mu}$. A standard tableau of shape
$\overline{\lambda/\mu}$ is said to be {\em $p$-standard} if it
satisfies that if there exists a box below $A_{(c,\lambda/\mu)}$
then it is labeled by a number greater than the one in the last box
in the row containing $A_{(c,\lambda/\mu)}$. For $n\in\mathbb{Z}_+$,
set
\begin{align}
\Delta_2(n)=\Big\{((c,\lambda/\mu),S)~|&~(c,\lambda/\mu)\in\mathcal{PS}^*_{u}(n-p+2u+1),
S \text{ is a }p\text{-standard }\notag\\
&\text{ tableau of shape }\overline{\lambda/\mu}, 1\leq
u\leq\frac{p-3}{2} \Big\}.\notag
\end{align}
Suppose $\underline{i}\in W_2(\mhcn)$ and set $u=u_{\underline{i}}$.
By Lemma~\ref{lem:ppcom1}, there exists a unique sequence of
integers $r_0<r_1<\ldots<r_{\frac{p-3}{2}-u}
<q<t_{\frac{p-3}{2}-u}<\ldots<t_1<t_0$ such that $i_q=\frac{p-1}{2},
i_{r_j}=i_{t_j}=u+j$ for $0\leq j\leq\frac{p-3}{2}$. Since the
vector $\widehat{\underline{i}}$ in~(\ref{ihat}) is splittable, by
Lemma~\ref{lem:ppcom2}(2) we apply the map $\mathcal{G}$
in~(\ref{mathcalG}) to $\widehat{\underline{i}}$ to get a placed
skew shifted Young diagram $(c,\lambda/\mu)$ and a standard tableau
$T$ of shape $\lambda/\mu$ whose boxes are labeled by
$\{1,\ldots,n\}\setminus\{r_1,\ldots,r_{\frac{p-3}{2}-u},
\ldots,t_1,t_0\}$. By Lemma~\ref{lem:ppcom2}(1), we have
$(c,\lambda/\mu)\in\mathcal{PS}^*_{u}(n-p+2u+1)$. Label the boxes in
$\overline{\lambda/\mu}$ on the right of $A_{(c,\lambda/\mu)}$ by
$r_1,\ldots,r_{\frac{p-3}{2}-u},q,t_{\frac{p-3}{2}-u},\ldots, t_1,
t_0$ consecutively and denote the resulting tableau by $S$. Observe
that $A_{(c,\lambda/\mu)}$ is labeled by $r_0$ and hence the row
containing $A_{(c,\lambda/\mu)}$ in $S$ is increasing.  If there
exists a box denoted by $B$ below $A_{(c,\lambda/\mu)}$, then $B$
has content $u-1$. Suppose $B$ is labeled by $e$, then $i_e=u-1$ and
$e>r_0$. Hence $e>t_0$ since $i_k\neq u-1$ for $r_0\leq k\leq t_0$.
Therefore $S$ is $p$-standard. Set
$\mathcal{G}_2(\underline{i}):=((c,\lambda/\mu),S)$. Hence we obtain
a map
\begin{align}
\mathcal{G}_2:
W_2(\mhcn)\longrightarrow\Delta_2(n).\label{mathcalGover}
\end{align}
If $\mathcal{G}_2(\underline{i})=((c,\lambda/\mu),S)$, we say
$\mathcal{G}_2(\underline{i})$ affords the placed skew shifted Young
diagram $(c,\lambda/\mu)$.

\begin{example} Suppose $p=7$ and $n=7$. Note that the vector
$\underline{i}=(1,2,0,3,2,1,0)$ belongs to
$W_2(\mathfrak{H}_7^{\mathfrak{c}})$ with $u_{\underline{i}}=2$ and
$\widehat{\underline{i}}=(1,2,0,1,0)$. By Example~\ref{example1}, we
see that the map $\mathcal{G}_2$ sends $\underline{i}$ to the pair
$((c,\lambda/\mu), S)$ with
\begin{align}
(c,\lambda/\mu)=\young(12,01,:0),\qquad
\overline{\lambda/\mu}=\young(\,\,\,\,,\,\,,:\,),\qquad
S=\young(1245,36,:7).\notag
\end{align}
\end{example}

 On the other hand, suppose
$(c,\lambda/\mu)\in\mathcal{PS}^*_{u}(n-p+2u+1)$ for $1\leq
u\leq\frac{p-3}{2}$ and  $S$ is $p$-standard tableau of shape
$\overline{\lambda/\mu}$. Assume that the boxes on the right of
$A_{(c,\lambda/\mu)}$ in $S$ are labeled by $r_1,\ldots,
r_{\frac{p-3}{2}-u},q,t_{\frac{p-3}{2}-u},\ldots,t_1, t_0$. Define
the contents of these additional boxes by setting
$c(S(q))=\frac{p-1}{2}, c(S(t_0))=u$ and $c(S(r_j))=u+j=c(S(t_j))$
for $1\leq j\leq\frac{p-3}{2}-u$. Set
\begin{align}
\mathcal{F}_2((c,\lambda/\mu),S):=(c(S(1)),\ldots,
c(S(n))).\label{mathcalFover}
\end{align}
\begin{lem} The map $\mathcal{F}_2$ in~(\ref{mathcalFover}) sends $\Delta_2(n)$ to
$W_2(\mhcn)$.
\end{lem}
\begin{proof} Suppose $((c,\lambda/\mu),S)\in \Delta_2(n)$ so that
$(c,\lambda/\mu)\in\mathcal{PS}^*_{u}(n-p+2u+1)$ for $1\leq
u\leq\frac{p-3}{2}$. Assume that $c(S(k))=c(S(l))=v$ for some $1\leq
k<l\leq n$.

If $0\leq v\leq u-1$, both  boxes $S(k)$ and $S(l)$ belong to
$(c,\lambda/\mu)$. Hence $1\in\{c(S(k+1)),\ldots, c(S(l-1))\}$ if
$v=0$, and $\{v-1, v+1\}\subseteq\{c(S(k+1)),\ldots, c(S(l-1))\}$ if
$1\leq v\leq u-1$. Now $v=u+m$ for some $0\leq
m\leq\frac{p-3}{2}-u$. Suppose the box $A_{(c,\lambda/\mu)}$ is
labeled by $r_0$ and the boxes on its right in $S$ are labeled by
$r_1,\ldots, r_{\frac{p-3}{2}-u},q,t_{\frac{p-3}{2}-u},\ldots,t_1,
t_0$. By the definition of $\mathcal{F}_2$, the boxes $S(k)$ and
$S(l)$ coincide with $S(r_m)$ and $S(t_m)$, respectively. Therefore
$(c(S(k+1)),\ldots, c(S(l-1)))$ contains the subsequence
$(v+1,\ldots,\frac{p-3}{2},\frac{p-1}{2},\frac{p-3}{2},\ldots,v+1)$,
and $\mathcal{F}_2((c,\lambda/\mu), S)\in W_2(\mhcn)$.

\end{proof}

\begin{prop}\label{prop:ppcomb2}
The map $\mathcal{G}_2:W_2(\mhcn)\rightarrow\Delta_2(n)$ is a bijection
 with inverse $\mathcal{F}_2$.
\end{prop}
\begin{proof}
It is clear that $\mathcal{F}_2\circ
\mathcal{G}_2(\underline{i})=\underline{i} $ for $\underline{i}\in
W_2(\mhcn)$. Conversely, suppose $(c,\lambda/\mu)\in
\mathcal{PS}^*_u(n-p+2u+1)$ and $S$ is a $p$-standard tableau of
shape $\overline{\lambda/\mu}$ for some $1\leq u\leq \frac{p-3}{2}$.
Set $\underline{i}=\mathcal{F}_2((c,\lambda/\mu),S)\in W_2(\mhcn)$.
Denote by $T$ the standard tableau of shape $\lambda/\mu$ obtained
by removing the $p-2u-1$ boxes on the right of $A_{(c,\lambda/\mu)}$
from $S$. Suppose the boxes of $T$ are labeled by
$l_1<l_2<\cdots<l_{n-p+2u+1}$. By the definition of $\mathcal{F}_2$
we have
$$\widehat{\underline{i}}=(c(T(l_1)),\ldots,c(T(l_{n-p+2u+1})))
=\mathcal{F}((c,\lambda/\mu),T).
$$
Therefore
$\mathcal{G}(\widehat{\underline{i}})=\mathcal{G}\circ\mathcal{F}((c,\lambda/\mu),T)
=((c,\lambda/\mu),T)$ by Proposition~\ref{onto}. Note that
$\mathcal{G}_2(\underline{i})$ is obtained by adding $p-2u-1$ boxes
labeled by $\{1,\ldots\}\setminus\{l_1,\ldots,l_{n-p+2u+1}\}$ to the
right of $A_{(c,\lambda/\mu)}$ in $T$. Since $T$ is obtained by
removing the $p-2u-1$ boxes on the right of $A_{(c,\lambda/\mu)}$
from $S$, $\mathcal{G}_2(\underline{i})=((c,\lambda/\mu),S)$ and
hence
$\mathcal{G}_2\circ\mathcal{F}_2((c,\lambda/\mu),S)=((c,\lambda/\mu),S)$.
\end{proof}
\begin{lem}\label{lem:ppcom3}
 $\underline{i}\sim\underline{j}\in W_2(\mhcn)$ if and
only if $\mathcal{G}_2(\underline{i})$ and
$\mathcal{G}_2(\underline{j})$ afford the same placed skew shifted
Young diagram in $\mathcal{PS}^*_{u}(n-p+2u+1)$ for some $1\leq
u\leq\frac{p-3}{2}$.
\end{lem}
\begin{proof}
 By Lemma~\ref{lem:ppcom2}, if
$\underline{i}\sim\underline{j}$, then
$\widehat{\underline{i}}\sim\widehat{\underline{j}}$.  By
Lemma~\ref{lem:pocomb.des2}, $\mathcal{G}(\widehat{\underline{i}})$
and $\mathcal{G}(\widehat{\underline{j}})$ afford the same skew
shifted Young diagram. Hence  $\mathcal{G}_2(\underline{i})$ and
$\mathcal{G}_2(\underline{j})$ afford the same placed skew shifted
Young diagram $(c,\lambda/\mu)\in\mathcal{PS}^*_u(n-p+2u+1)$ for
some $1\leq u\leq\frac{p-3}{2}$.

Conversely, suppose $\mathcal{G}_2(\underline{i})$ and
$\mathcal{G}_2(\underline{j})$ afford the same placed skew shifted
Young diagram $(c,\lambda/\mu)\in\mathcal{PS}^*_u(n-p+2u+1)$ for
some $1\leq u\leq\frac{p-3}{2}$. Suppose there are $m$ boxes below
$A_{(c,\lambda/\mu)}$ in $(c,\lambda/\mu)$. By
Lemma~\ref{lem:ppcom1}, we see that
\begin{align}
\underline{i}&\sim
(\underline{i}',u,u+1,\cdots,\frac{p-3}{2},\frac{p-1}{2},\frac{p-3}{2},\cdots,
u+1, u, u-1, \cdots, u-m),\label{isim'}\\
\underline{j}&\sim
(\underline{j}',u,u+1,\cdots,\frac{p-3}{2},\frac{p-1}{2},\frac{p-3}{2},\cdots,
u+1, u, u-1, \cdots, u-m).\label{jsim'}
\end{align}
 for some $\underline{i}', \underline{j}'\in \nabla(n-p+2u-m)$. This
 together with~Lemma~\ref{lem:ppcom2} shows that
\begin{align}
\widehat{\underline{i}}&\sim
(\underline{i}', u, u-1, \cdots, u-m)\notag\\
\widehat{\underline{j}}&\sim (\underline{j}',u, u-1, \cdots,
u-m).\notag
\end{align}
Observe that $\mathcal{G}(\widehat{\underline{i}})$ and
$\mathcal{G}(\widehat{\underline{j}})$ afford the placed skew
shifted Young diagram $(c,\lambda/\mu)$. Therefore
$\mathcal{G}(\underline{i}')$ and $\mathcal{G}(\underline{j}')$
afford the same placed skew shifted Young diagram and hence
 $\underline{i}'\sim\underline{j}'$ by
Lemma~\ref{lem:pocomb.des2}. This together
with~(\ref{isim'})~and~(\ref{jsim'}) shows that
$\underline{i}\sim\underline{j}$.
\end{proof}

Suppose
$(c,\lambda/\mu)\in\mathcal{PS}_{\frac{p-1}{2}}(n)\cup\big(\cup_{1\leq
u\leq\frac{p-3}{2}}\mathcal{PS}^*_{u}(n-p+2u+1)\big)$.  By
Proposition~\ref{prop:ppcomb0} and Proposition~\ref{prop:ppcomb2},
there exists $\underline{i}\in W_k(\mhcn)$ such that
$\mathcal{G}_k(\underline{i})$ affords $(c,\lambda/\mu)$ for
$k=1,2$. Let
\begin{align}
D_p(c,\lambda/\mu):=D^{\underline{i}}.\label{D(c,lambda)p}
\end{align}
Note that if there exists $\underline{j}\in W'(\mhcn)$ satisfying
that $\mathcal{G}_k(\underline{j})$ also affords $(c,\lambda/\mu)$
for $k=1,2$, then  $\underline{i}\sim\underline{j}$ by
Proposition~\ref{prop:ppcomb0} and Lemma~\ref{lem:ppcom3} and hence
the $\mhcn$-module $D_p(c,\lambda/\mu)$ is unique (up to
isomorphism) by Theorem~\ref{thm:Classficiation}(2).

For $\ds(c,\lambda/\mu)\in
\mathcal{PS}_{\frac{p-1}{2}}(n)\cup\big(\cup_{1\leq
u\leq\frac{p-3}{2}}\mathcal{PS}^*_{u}(n-p+2u+1)\big)$, denote by
$\gamma_0(c,\lambda/\mu)$ the number of boxes with content zero in
$(c,\lambda/\mu)$. If $(c,\lambda/\mu)\in
\mathcal{PS}_{\frac{p-1}{2}}(n)$, set $f^{\lambda/\mu}$ to be the
number of standard tableaux of shape $\lambda/\mu$. If
$(c,\lambda/\mu)\in\cup_{1\leq
u\leq\frac{p-3}{2}}\mathcal{PS}^*_{u}(n-p+2u+1)$, let
$f_p^{\lambda/\mu}$  be the number of $p$-standard tableaux of shape
$\overline{\lambda/\mu}$.

The following is a Young diagrammatic reformulation of
Theorem~\ref{thm:Classficiation} for $p\geq 3$.

\begin{thm}\label{thm:dimpandparam.}
Suppose $(c,\lambda/\mu)\in\mathcal{PS}_{\frac{p-1}{2}}(n)\cup\big(\cup_{1\leq
u\leq\frac{p-3}{2}}\mathcal{PS}^*_{u}(n-p+2u+1)\big)$ and write
$\gamma_0=\gamma_0(c,\lambda/\mu)$. Then,

\text{(1)} $D_p(c,\lambda/\mu)$ is type
 $\texttt{M}$ if $\gamma_0$ is even and is type
$\texttt{Q}$ if $\gamma_0$ is odd. Moreover if
$(c,\lambda/\mu)\in\mathcal{PS}_{\frac{p-1}{2}}(n)$, then $\dim
D_p(c,\lambda/\mu)=2^{n-\lfloor\frac{\gamma_0}{2}\rfloor}f^{\lambda/\mu}$;
if $(c,\lambda/\mu)\in\big(\cup_{1\leq
u\leq\frac{p-3}{2}}\mathcal{PS}^*_{u}(n-p+2u+1)\big)$, then $\dim
D_p(c,\lambda/\mu)=2^{n-\lfloor\frac{\gamma_0}{2}\rfloor}f_p^{\lambda/\mu}$.

(2) The $\mhcn$-modules $D_p(c,\lambda/\mu)$ for
$(c,\lambda/\mu)\in\ds
\mathcal{PS}_{\frac{p-1}{2}}(n)\cup\big(\cup_{1\leq
u\leq\frac{p-3}{2}}\mathcal{PS}^*_{u}(n-p+2u+1)\big)$ form a set of
pairwise non-isomorphic irreducible completely splittable
$\mhcn$-modules in $\text{Rep}_{\I}\mhcn$.
\end{thm}
\begin{proof}
(1) By Proposition~\ref{prop:ppcomb0} and
Proposition~\ref{prop:ppcomb2}, the number of $1\leq k\leq n$ with
$i_k=0$ equals to $\gamma_0$. Hence by Lemma~\ref{lem:irrepPn} and
Theorem~\ref{thm:Classficiation}, $D_p(c,\lambda/\mu)$ is type
 $\texttt{M}$ if $\gamma_0$ is even and is type
$\texttt{Q}$ if $\gamma_0$ is odd.

Set $|P_{\underline{i}}|$ to be the number of elements contained in
$P_{\underline{i}}$.  If
$(c,\lambda/\mu)\in\mathcal{PS}_{\frac{p-1}{2}}(n)$, then by
Proposition~\ref{prop:ppcomb0} there exists a one-to-one
correspondence between the set of weights in $W(\mhcn)$ equivalent
to $\underline{i}$ and the set of standard tableaux of shape
$\lambda/\mu$. This implies $|P_{\underline{i}}|=f^{\lambda/\mu}$ by
Lemma~\ref{lem:lambdaP}. If
$(c,\lambda/\mu)\in\mathcal{PS}^*_{u}(n-p+2u+1)$ for some $1\leq
u\leq\frac{p-3}{2}$. By Lemma~\ref{lem:ppcom3}, there exists a
one-to-one correspondence between the set of weights in $W(\mhcn)$
equivalent to $\underline{i}$ and the set of splittable standard
tableaux of shape $\overline{\lambda/\mu}$. This implies
$|P_{\underline{i}}|=f_p^{\lambda/\mu}$ by Lemma~\ref{lem:lambdaP}.
Now the Proposition follows from  Lemma~\ref{lem:irrepPn} and
Theorem~\ref{thm:Classficiation}.

(2) It follows from Proposition~\ref{prop:ppcomb0},
Proposition~\ref{prop:ppcomb2}, Lemma~\ref{lem:ppcom3} and
Theorem~\ref{thm:Classficiation}(3).


\end{proof}

\begin{rem}\label{empty} Note that for fixed $p\geq 3$,
$\mathcal{PS}_{\frac{p-1}{2}}(n)\neq\emptyset$ if and only if
$n\leq\frac{(p+1)(p+3)}{8}$. Moreover if $n>\frac{(p+1)(p+3)}{8}$,
then $\mathcal{PS}^*_{u}(n-p+2u+1)=\emptyset$ for $1\leq
u\leq\frac{p-3}{2}$. Hence there is no irreducible completely
splittable supermodule in $\operatorname{Rep}_{\I}\mhcn$ if
$n>\frac{(p+1)(p+3)}{8}$ for fixed $p\geq 3$.
\end{rem}

\section{Completely splittable representations of finite Hecke-Clifford
algebras}\label{fHCa} Denote by $\mathcal{C}_n$ the subalgebra of
$\mhcn$ generated by $c_1,\ldots, c_n$, which is known as the
Clifford algebra. The finite Hecke-Clifford algebra
$\mathcal{Y}_n=\mathcal{C}_n\rtimes\F S_n$ is isomorphic to the
subalgebra of $\mhcn$ generated by $c_1,\ldots, c_n, s_1,\ldots,
s_{n-1}$.  The Jucys-Murphy elements $L_k(1\leq k\leq n)$ in
$\mathcal{Y}_n$ are defined as
\begin{align}
L_k=\sum_{1\leq j< k}(1+c_jc_k)(jk),\label{JM}
\end{align}
where $(jk)$ is the transposition exchanging $j$ and $k$ and keeping
all others fixed.

\begin{defn} A $\mathcal{Y}_n$-module is called
{\em completely splittable} if the Jucys-Murphy elements $L_k (1\leq
k\leq n)$ act semisimply.
\end{defn}

It is well known that there exists a surjective homomorphism
\begin{align}
\digamma: \mhcn&\rightarrow \mathcal{Y}_n\notag\\
c_k\mapsto c_k, s_l&\mapsto s_l, x_k\mapsto L_k, \quad (1\leq
k\leq n, 1\leq l\leq n-1)\notag
\end{align}
whose kernel coincides with the ideal of $\mhcn$ generated by $x_1$.
Hence the category of finite dimensional $\mathcal{Y}_n$-modules can
be identified as the category of finite dimensional $\mhcn$-modules
which are annihilated by $x_1$. By~\cite[Lemma 4.4]{BK} (cf.
\cite[Lemma 15.1.2]{K2}), a $\mhcn$-module $M$ belongs to the
category $\text{Rep}_{\I}\mhcn$ if all of eigenvalues of $x_j$ on
$M$ are of the form $q(i)$ for some $1\leq j\leq n$. Hence the
category of finite dimensional completely splittable
$\mathcal{Y}_n$-module can be identified with the subcategory of
$\text{Rep}_{\I}\mhcn$ consisting of completely splittable
$\mhcn$-modules on which $x_1=0$. By~(\ref{decomp.1}), we can
decompose any finite dimensional $\mathcal{Y}_n$-module $M$ as
$$
M=\oplus_{\underline{i}\in\I^n}M_{\underline{i}},
$$
where $M_{\underline{i}}=\{z\in M~|~(L_k^2-q(i_k))^Nz=0,\text{ for
}N\gg0, 1\leq k\leq n \}$. If $M_{\underline{i}}\neq 0$, then
$\underline{i}$ is called a {\em weight} of $M$.

\begin{defn} Define $W(\mathcal{Y}_n)$ to be the set of weights
$\underline{i}=(i_1,\ldots,i_n)\in W(\mhcn)$ satisfying the
following additional conditions:
\begin{align}
i_1=0, \quad\{i_k-1,i_k+1\}\cap\{i_1,\ldots,i_{k-1}\}\neq\emptyset
\text{ for  } 2\leq k\leq n. \label{weigofSerg}
\end{align}
\end{defn}

\begin{prop}\label{prop:wei.of sergeev}
$W(\mathcal{Y}_n)$ is the set of weights occurring in irreducible
completely splittable $\mathcal{Y}_n$-modules.
\end{prop}
\begin{proof} Suppose $\underline{i}$ occurs in some irreducible
completely splittable representation $M$ of $\mathcal{Y}_n$, then
$i_1=0$ since $L_1=0$ on $M$. For $2\leq k\leq n$, if $i_k=0$, then
by Lemma~\ref{lem:restr.2} we have $1\in\{i_1,\ldots,i_{k-1}\}$ and
hence $\{i_k-1,i_k+1\}\cap\{i_1,\ldots,i_{k-1}\}\neq\emptyset$. Now
assume $i_k\geq 1$ and suppose
$\{i_k-1,i_k+1\}\cap\{i_1,\ldots,i_{k-1}\}=\emptyset$. Then $s_l$ is
admissible with respect to $s_{l+1}\cdots s_{k-1}\cdot\underline{i}$
for $1\leq l\leq k-1$ and hence $M_{s_1\cdots
s_{k-1}\cdot\underline{i}}\neq 0$. Set $\underline{j}=s_1\cdots
s_{k-1}\cdot\underline{i}$. Note $j_1=i_k\neq 0$ and this
contradicts the fact that $L_1=0$ on $M$.

Conversely, let $\underline{i}\in W(\mathcal{Y}_n)$. Recall
$P_{\underline{i}}$ and $D^{\underline{i}}$ from~(\ref{Punderi})
and~(\ref{Dunderi}), respectively. It can be easily checked that
$\tau\cdot\underline{i}\in W(\mathcal{Y}_n)$ for each $\tau\in
P_{\underline{i}}$ and hence $x_1=0$ on $D^{\underline{i}}$. This
implies that $D^{\underline{i}}$ can be factored through the
surjective map $\digamma$ and hence it gives an irreducible
completely splittable $\mathcal{Y}_n$-module. The Proposition
follows from the fact that $\underline{i}$ is a weight of
$D^{\underline{i}}$.
\end{proof}

%
%

Denote by $\nabla^{\circ}(n)$ the subset of $\nabla(n)$ consisting
of $\underline{i}$ satisfying~(\ref{weigofSerg}).
\begin{lem}\label{lem:shiftedY}
The restriction $\mathcal{G}^{\circ}$ of the map $\mathcal{G}$ in
(\ref{mathcalG}) induces a bijection between $\nabla^{\circ}(n)$ and
the set of pairs $(\lambda,T)$ of strict partitions $\lambda$ and
standard tableaux $T$ of shape $\lambda$.
\end{lem}
\begin{proof} Let us proceed by induction on $n$. Clearly the statement holds for $n=1$.
Let $\underline{i}\in \nabla^{\circ}(n)$. Then
$\underline{i}^{\prime}:=(i_1,\ldots,i_{n-1})\in
\nabla^{\circ}(n-1)$ and by induction we have
$\mathcal{G}(\underline{i}^{\prime})=(\widetilde{\lambda},S)$ for
some shifted Young diagram $\widetilde{\lambda}$ with $n-1$ boxes
and a standard tableau $S$ of shape $\widetilde{\lambda}$. Note that
$\mathcal{G}(\underline{i})$ is obtained by adding a box labeled by
$n$ to the diagonal of content $i_n$ in $S$. Since
$\{i_n-1,i_n+1\}\cap\{i_1,\ldots,i_{n-1}\}\neq\emptyset$, the
resulting diagram is still a shifted Young diagram.
\end{proof}

Note that if $p=0$, then $W(\mathcal{Y}_n)$ coincides with
$\nabla^{\circ}(n)$. Hence by Theorem~\ref{thm:dimandparam.} we have
the following which recovers Nazarov's result in~\cite{Na1}.

\begin{cor}\label{cor:p0Serg} Suppose that $p=0$ and that $\lambda$ is a strict partition of $n$.
Then,

(1) There exists an irreducible $ \mathcal{Y}_n$-module $D(\lambda)$
satisfying that $\dim
D(\lambda)=2^{n-\lfloor\frac{l(\lambda)}{2}\rfloor}f^{\lambda}$,
where $f^{\lambda}$ is the number of standard $\lambda$-tableaux.
Moreover, $D(\lambda)$ is type $\texttt{M}$ if $l(\lambda)$ is even
and is type $\texttt{Q}$ if $l(\lambda)$ is odd.

(2) The set of shifted Young diagrams with $n$ boxes parameterizes
the irreducible completely splittable $\mathcal{Y}_n$-modules.
\end{cor}
\begin{proof}
Suppose $\lambda$ is a strict partition of $n$. Recall the content
function $c_{\lambda}$ from Remark~\ref{rem:shiftedY}. Note that
$(c_{\lambda},\lambda)\in\mathcal{PS}(n)$. Recall the $\mhcn$-module
$D(c_{\lambda},\lambda)$ from~(\ref{D(c,lambda)}) and let
$$
D(\lambda)=D(c_{\lambda},\lambda).
$$
Now the Proposition follows from Theorem~\ref{thm:dimandparam.}.
\end{proof}
%
%
{\bf In the remaining part of this section, let us assume that
$p\geq3$}. Set $W_k(\mathcal{Y}_n):=W(\mathcal{Y}_n)\cap W_k(\mhcn)$
for $k=1, 2$.
\begin{lem}\label{Sergpp1}
The restriction $\mathcal{G}_1^{\circ}$ of $\mathcal{G}_1$ to
$W_1(\mathcal{Y}_n)$ gives a bijection from $W_1(\mathcal{Y}_n)$ to
the set of pairs $(\lambda, T)$ of strict partitions $\lambda$ of
$n$ boxes whose first part is less than or equal to $\frac{p+1}{2}$
and standard tableaux $T$ of shape $\lambda$.
\end{lem}
\begin{proof}
Observe that $W_1(\mathcal{Y}_n)\subseteq \nabla^{\circ}(n)$. By
Lemma~\ref{lem:shiftedY} and Proposition~\ref{prop:ppcomb0}, there
exists a one-to-one correspondence between $W(\mathcal{Y}_n)\cap
W_1(\mhcn)$ and the set consisting of pairs of shifted Young
diagrams
$\lambda=(\lambda_1,\ldots,\lambda_n)\in\mathcal{PS}_{\frac{p-1}{2}}(n)$
and standard tableaux of shape $\lambda$ with $c(T(k))=i_k$ for each
$1\leq k\leq n$.  Suppose the last box in the first row of $T$ is
labeled by $l$, then $c(T(l))=\lambda_1-1$ and hence
$\lambda_1\leq\frac{p+1}{2}$ since $c(T(l))=i_l\leq\frac{p-1}{2}$.
\end{proof}
\begin{lem}\label{Sergpp2}
The restriction $\mathcal{G}_2^{\circ}$ of the map $\mathcal{G}_2$
to $W_2(\mathcal{Y}_n)$ gives a bijection from $W_2(\mhcn)$ to the
set consisting of pairs $(\lambda, T)$, where $\lambda$ is a strict
partition whose first part is equal to $p-u$ and second part is less
than or equal to $u$ for some $1\leq u\leq\frac{p-3}{2}$, and $T$ is
a standard tableau of shape $\lambda$ satisfying that if
$\lambda_2=u$ then the number in last box of the second row is
greater than the number in the last box of the first row in $T$.
\end{lem}
\begin{proof}

Suppose $\underline{i}\in W_2(\mathcal{Y}_n)$.  It is clear that
$\widehat{\underline{i}}\in \nabla^{\circ}(n-p+2u+1)$ for some
$1\leq u\leq\frac{p-3}{2}$. By Lemma~\ref{lem:shiftedY}, we have
$\mathcal{G}^{\circ}(\widehat{\underline{i}})=(\mu, S)$ for some
shifted Young diagram $\mu\in \mathcal{PS}^*_u(n-p+2u+1)$ and
splittable standard tableau $S$ of shape $\mu$. Suppose
$\mu=(\mu_1,\ldots,\mu_m)$ and set
$\lambda=(\lambda_1,\ldots,\lambda_m):=\overline{\mu}$. Observe that
the last box in the first row of $\mu$ has content $u$. This implies
$\mu_1-1=u$ and hence $\mu_1=u+1, \mu_2\leq u$. Therefore
$\lambda_1=\mu_1+p-2u-1=p-u$ and $\lambda_2=\mu_2\leq u$.  Note that
if $\lambda_2<u$, then the set of splittable standard tableaux of
shape $\lambda$ coincides with the set of standard
$\lambda$-tableaux; otherwise the set of  splittable standard
tableaux of shape $\lambda$ coincides with the set of standard
$\lambda$-tableaux in which the number in last box in the second row
is greater than the number in the last box in the first row.
\end{proof}
If $\lambda=(\lambda_1,\ldots,\lambda_l)$ is strict partition of $n$
satisfying $\lambda_1\leq\frac{p-3}{2}$, then
$(c_{\lambda},\lambda)\in \mathcal{PS}_{\frac{p-1}{2}}(n)$, where
$c_{\lambda}$ is the unique content function on $\lambda$ by
Remark~\ref{rem:shiftedY}. Recall the $\mhcn$-module
$D_p(c_{\lambda},\lambda)$ from~(\ref{D(c,lambda)p}) and let
$$
D_p(\lambda)=D_p(c_{\lambda},\lambda).
$$ Let $f^{\lambda}$ be the number of standard tableaux of shape
$\lambda$. Recall $f^{\lambda/\emptyset}$ and
$\gamma_0(c_{\lambda},\lambda)$ from
Theorem~\ref{thm:dimpandparam.}.  Clearly
$f^{\lambda}=f^{\lambda/\emptyset}$ and moreover
$\gamma_0(c_{\lambda},\lambda)=l(\lambda)$.

If $\lambda=(\lambda_1,\ldots,\lambda_l)$ is strict partition of $n$
satisfying $\lambda_1=p-u$ and $ \lambda_2\leq u$ for some $1\leq
u\leq\frac{p-3}{2}$. Denote by $\widehat{\lambda}$ the strict
partition obtained by removing the last $p-2u-1$ boxes in the first
row of $\lambda$.  Recall $c_{\widehat{\lambda}}$ from
Remark~\ref{rem:shiftedY}. Note that
$(c_{\widehat{\lambda}},\widehat{\lambda})\in \cup_{1\leq
u\leq\frac{p-3}{2}}\mathcal{PS}_u^*(n-p+2u+1)$. Recall the
$\mhcn$-module $D_p(c_{\widehat{\lambda}},\widehat{\lambda})$
from~(\ref{D(c,lambda)p}) and let
$$
D_p(\lambda)=D_p(c_{\widehat{\lambda}},\widehat{\lambda}).
$$
Let $f_p^{\lambda}$ be the number of standard $\lambda$-tableau $T$
if $\lambda_1=p-u, \lambda_2< u$ for some $1\leq
u\leq\frac{p-3}{2}$; if $\lambda_1=p-u, \lambda_2=u$ for some $1\leq
u\leq\frac{p-3}{2}$ let $f_p^{\lambda}$ be the number of standard
$\lambda$-tableau $T$ in which the number in last box of the second
row is greater than the number in the last box of the first row.
Recall $f^{\widehat{\lambda}/\emptyset}$ and
$\gamma_0(c_{\widehat{\lambda}},\widehat{\lambda})$ from
Theorem~\ref{thm:dimpandparam.}. One can easily check that
$f_p^{\lambda}=f_p^{\widehat{\lambda}/\emptyset}$ and moreover
$\gamma_0(c_{\widehat{\lambda}},\widehat{\lambda})=l(\lambda)$.

Combining the above observations and Lemma~\ref{Sergpp1},
Lemma~\ref{Sergpp2} and Theorem~\ref{thm:dimpandparam.}, we have the
following.
\begin{thm}\label{thm:Serg} Let $p\geq 3$. Suppose that
$\lambda=(\lambda_1,\ldots, \lambda_m)$ is strict partition
with $n$ boxes
 satisfying either $(\lambda_1=p-u$ and $
\lambda_2\leq u$ for some $1\leq u\leq\frac{p-3}{2})$ or
$(\lambda_1\leq \frac{p+1}{2})$.

(1) $D_p(\lambda)$ is type $\texttt{M}$ if $l(\lambda)$ is even and
is type $\texttt{Q}$ if $l(\lambda)$ is odd. If $\lambda_1\leq
\frac{p+1}{2}$, then $\dim
D_p(\lambda)=2^{n-\lfloor\frac{l(\lambda)}{2}\rfloor}f^{\lambda}$;
if $\lambda_1=p-u$ and $ \lambda_2\leq u$, then $\dim
D_p(\lambda)=2^{n-\lfloor\frac{l(\lambda)}{2}\rfloor}f_p^{\lambda}$.

(2) The $\mathcal{Y}_n$-modules $D_p(\lambda)$ for strict partitions
$\lambda=(\lambda_1,\ldots, \lambda_m)$ with $n$ boxes
 satisfying either $(\lambda_1=p-u,
\lambda_2\leq u$ for some $1\leq u\leq\frac{p-3}{2})$ or
$(\lambda_1\leq \frac{p+1}{2})$ form a complete set of
non-isomorphic irreducible completely splittable
$\mathcal{Y}_n$-modules.
\end{thm}
\begin{rem}\label{rem:Serg}
(1) A partition $\lambda=(\lambda_1,\lambda_2,\ldots)$ is called
$p$-restricted $p$-strict if $p$ divides $\lambda_r$ whenever
$\lambda_r=\lambda_{r+1}$ for $r\geq 1$ and in addition
$\lambda_r-\lambda_{r+1}<p$ if $p\mid\lambda_r$ and
$\lambda_r-\lambda_{r+1}\leq p$ if $p\nmid\lambda_r$ (cf. \cite[\S
9-a]{BK}). It is known from~\cite[\S 9-b]{BK} that there exists an
irreducible $\mathcal{Y}_n$-module $M(\lambda)$ associated to each
$p$-restricted $p$-strict partition $\lambda$  of $n$ and moreover
$\{M(\lambda)~|~\lambda \text{ is a }p\text{-restricted
}p\text{-strict partition of }n \}$ forms a complete set of pairwise
non-isomorphic irreducible $\mathcal{Y}_n$-modules. If $\lambda$ is
a strict partition with either $\lambda_1=p-u$ and $ \lambda_2\leq
u$ for some $1\leq u\leq\frac{p-3}{2}$ or $\lambda_1\leq
\frac{p+1}{2}$, then $\lambda$ is $p$-restricted $p$-strict and
moreover $D(\lambda)\cong M(\lambda)$ by claiming that they have the
same set of weights.


(2) It is well known that the representation theory of the spin
symmetric group algebra $\F S^-_n$ is essentially equivalent to that
of $\mathcal{Y}_n$ due to the isomorphism $\mathcal{C}_n\otimes\F
S_n^-\cong\mathcal{Y}_n$. Applying the representation theory of
$\mathcal{Y}_n$ established so far, we can obtain a family of
irreducible representations of the spin symmetric group algebra $\F
S_n^-$ for which dimensions and characters can be explicitly
described. Over the complex field $\mathbb{C}$, these modules were
originally constructed by Nazarov in~\cite{Na1}.
\end{rem}
\section{A larger category}\label{aLc}

Recall that $\mathcal{C}_n$ is the Clifford algebra generated by
$c_1,\ldots, c_n$ subject to the relation~(\ref{clifford}) and
$\mathcal{Y}_n=\mathcal{C}_n\rtimes\F S_n$. The basic spin
$\mathcal{Y}_n$-module $I(n)$ (cf. \cite[(9.11)]{BK}) is defined by
\begin{align}
I(n):={\rm ind}^{\mathcal{Y}_n}_{\F S_n}\textbf{1},\label{basicspin}
\end{align}
where $\textbf{1}$ is the trivial $1$-dimensional $\F S_n$-module.
Note that $\{c_1^{\alpha_1}\cdots
c_n^{\alpha_n}~|~(\alpha_1,\ldots,\alpha_n)\in\mathbb{Z}_2^n\}$
forms a basis of $I(n)$. It can be easily checked that each element
$c_1^{\alpha_1}\cdots c_n^{\alpha_n}$ is a simultaneous eigenvector
for $L_1^2,\ldots, L^2_n$. Hence all $L_k^2, 1\leq k\leq n, $ act
semisimply on $I(n)$. Define the $p$-restricted $p$-strict partition
$\omega_n$ by
\begin{eqnarray*}
\omega_n= \left \{
 \begin{array}{ll}
 (p^{a},b),
 & \text{ if } b\neq0 \\
 (p^{a-1},p-1,1)
 , & \text{ otherwise },
 \end{array}
 \right.
\end{eqnarray*}
where $n=ap+b$ with $0\leq b<p$. By \cite[Lemma 9.7]{BK}, we have
$I(n)\cong M(\omega_n)$ if $p\nmid n$ and if $p\mid n$ then $I(n)$
is an indecomposable module with two composition factors both
isomorphic to $M(\omega_n)$.  By Remark~\ref{rem:Serg}, the
Jucys-Murphy elements $L_k$ do not act semisimply on $M(\omega_n)$.
Hence $L_k^2, 1\leq k\leq n,$ act semisimply on $M(\omega_n)$ which
is not completely splittable.  On the other hand, Wang~\cite{W}
introduced the degenerate spin affine Hecke-Clifford algebra
$\mathfrak{H}^-$, which is the superalgebra with odd generators
$b_i(1\leq i\leq n)$ and $t_i(1\leq i\leq n-1)$ subject to the
relations
\begin{align}
t_i^2=1, t_it_{i+1}t_i&=t_{i+1}t_it_{i+1}, t_it_i=-t_it_i, \quad
|i-j|>1,\notag\\
b_ib_j&=-b_jb_i,\quad i\neq j,\notag\\
t_ib_i=-b_{i+1}t_i+1, t_ib_j&=-b_jt_i,\quad j\neq i,i+1.\notag
\end{align}
Moreover, an algebra isomorphism between $\mhcn$ and
$\mathcal{C}_n\otimes\mathfrak{H}^-$ which maps $x_k$ to
$\sqrt{-2}c_kb_k$ is established.  Since $b_1,\ldots, b_n$ are
anti-commutative, it is reasonable to study $\mathfrak{H}^-$-modules
on which the commuting operators $b_1^2,\ldots,b_n^2$ act
semisimply. As $x_k^2$ is sent to $2b_k^2$, it is reduced to study
the $\mhcn$-modules on which $x_k^2$ act semisimply.

Motivated by the above observations, in this section we shall study
the category of $\mhcn$-modules on which all $x_k^2, 1\leq k\leq n,$
act semisimply.

\subsection{The case for $n=2, 3$.} Recall the irreducible
$\mathfrak{H}_2^{\mathfrak{c}}$-module $V(i,j)$ for $i,j\in\I$
from Proposition~\ref{prop:rank 2}.
%
\begin{lem}\label{lem:n=2} Let $i,j\in\I$. Then $x_1^2, x_2^2$
act semisimply on the $\mathfrak{H}_2^{\mathfrak{c}}$-module
$V(i,j)$ if and only if $i\neq j$ or $i=j=0$.
\end{lem}
\begin{proof} By Proposition~\ref{prop:rank 2}, if $i\neq j$ then $V(i,j)$ is completely splittable and hence
$x^2_1, x^2_2$ act semisimply. It suffices to prove that if $i=j$,
then $x_1^2, x_2^2$ act semisimply on $V(i,j)$ if and only if
$i=j=0$. Now assume $i=j$. By Proposition~\ref{prop:rank 2},
$V(i,j)={\rm
ind}^{\mathfrak{H}^{\mathfrak{c}}_2}_{\mathcal{P}_2^{\mathfrak{c}}}L(i)\circledast
L(j)$. Suppose $x_1^2, x_2^2$ act semisimply on $V(i,j)$ and let
$0\neq z\in V(i,j)$. Then $x^2_1z=q(i)z= x^2_2z$. This together
with~(\ref{reln:x2s1}) shows that
$$
\big(x_1(1-c_1c_{2})+(1-c_1c_{2})x_{2}\big)z=0.
$$
This implies
$$
4q(i)z=2(x_1^2+x^2_{2})z=
\big(x_1(1-c_1c_{2})+(1-c_1c_{2})x_{2}\big)^2z=0.
$$
This means $q(i)=0$ and hence $i=0$ since $p\neq 2$.

Conversely if $i=j=0$, then $x_1=0=x_2$ on $L(i)\circledast L(j)$
and hence $x_1^2=0=x_2^2$ on $V(i,j)$ by the fact that $V(0,0)$
has two composition factors isomorphic to $L(0)\circledast L(0)$
as $\mathcal{P}_2^{\mathfrak{c}}$-modules.
\end{proof}
Observe that the subalgebra generated by $x_k, x_{k+1}, c_k,
c_{k+1}, s_k$ is isomorphic to $\mathfrak{H}_2^{\mathfrak{c}}$ for
each fixed $1\leq k\leq n-1$, . By Lemma~\ref{lem:n=2}, we have
the following.
\begin{cor}\label{cor:n=2} Suppose that $M\in\operatorname{Rep}_{\I}\mhcn$
and $x_k^2, 1\leq k\leq n$ act semisimply. Let
$\underline{i}\in\I^n$ be a weight of $M$. If $i_k=i_{k+1}$ for some
$1\leq k\leq n-1$, then $i_k=i_{k+1}=0$.
\end{cor}
\begin{lem}\label{lem:NP}
 For any $z\in V(0,0)$, we have
$$
\big((1+c_1c_2)x_1+(1-c_1c_2)x_2\big)z=0,\quad x_1x_2z=0.
$$
\end{lem}
\begin{proof} Let $z\in V(0,0)$. By Lemma~\ref{lem:n=2},  $x_1^2=0=x_2^2$ on
$V(0,0)$. This together with~(\ref{reln:x2s1}) shows that
\begin{align}
 \big((1+c_1c_2)x_1+(1-c_1c_2)x_2\big)z=0.\label{xccx}
\end{align}
Multiplying both sides of (\ref{xccx}) by $x_1(1+c_1c_2)$, we obtain
that
$$
(2x_1c_1c_2x_1+2x_1x_2)z=0.
$$
This implies that $x_1x_2z=0$ since $x_1^2z=0$.
\end{proof}
Recall that $\mathfrak{H}^{\mathfrak{c}}_{2,1}$ is the subalgebra
of $\mathfrak{H}_3^{\mathfrak{c}}$ generated by
$\mathcal{P}_3^{\mathfrak{c}}$ and $S_2$.
\begin{lem}\label{lem:n=3case1}

The irreducible $\mathfrak{H}^{\mathfrak{c}}_{2,1}$-module
$V(0,0,1):=V(0,0)\circledast L(1)$ affords an irreducible
$\mathfrak{H}_3^{\mathfrak{c}}$-module via $s_2=\Xi_2$.
\end{lem}
\begin{proof}
Since $L(1)$ is of type $\texttt{M}$, by Lemma~\ref{tensorsmod}
$V(0,0,1)=V(0,0)\boxtimes L(1)$. It is routine to check that
$s_2^2=1, s_2x_1=x_1s_2, s_2x_2=x_3s_2-(1+c_2c_3)$ and
$s_2c_1=c_1s_2, s_2c_2=c_3s_2$ on $V(0,0,1)$. It remains to prove
$s_1s_2s_1=s_2s_1s_2$. Let $0\neq z\in V(0,0,1)$.  Note that
$$
x_2^2z=0, x_3^2z=2z
$$ and hence
\begin{align}
s_2z=\frac{1}{2}\big((x_2+x_3)+c_2c_3(x_2-x_3)\big)z
=\frac{1}{2}\big((1+c_2c_3)x_2+(1-c_2c_3)x_3\big)z.\label{n3s2z}
\end{align}
Using~(\ref{n3s2z}) with $z$ replaced by $s_1z$ and (\ref{px1}),
we show by a straightforward calculation that
\begin{align}
s_2s_1z=\frac{1}{2}s_1\big((1+c_1c_3)x_1+(1-c_1c_3)x_3\big)z+\frac{1}{2}(1+c_1c_2+c_2c_3-c_1c_3)z.
\label{n3s1s2}
\end{align}
This implies that
\begin{align}
s_1s_2s_1z=\frac{1}{2}\big((1+c_1c_3)x_1+(1-c_1c_3)x_3\big)z+\frac{1}{2}s_1(1+c_1c_2+c_2c_3-c_1c_3)z\label{s1s2s1}.
\end{align}
On the other hand, it follows from~(\ref{n3s2z}) with $z$ replaced
by $s_2z$ and~(\ref{n3s1s2}) that
\begin{align}
s_2s_1s_2z=&\frac{1}{4}s_1\big((1+c_1c_3)x_1+(1-c_1c_3)x_3\big)\big((1+c_2c_3)x_2+ (1-c_2c_3)x_3\big)z\notag\\
&+\frac{1}{4}(1+c_1c_2+c_2c_3-c_1c_3)\big((1+c_2c_3)x_2+
(1-c_2c_3)x_3\big)z\notag\\
=&\frac{1}{4}s_1\big((1+c_1c_3)x_1+(1-c_1c_3)x_3\big)\big((1+c_2c_3)x_2+ (1-c_2c_3)x_3\big)z\notag\\
&+\frac{1}{2}(c_1c_2+c_2c_3)x_2z+\frac{1}{2}(1-c_1c_3)x_3z.\label{s2s1s2}
\end{align}
The first term on the right hand side of~(\ref{s2s1s2}) can be
simplified as follows
\begin{align}
\frac{1}{4}s_1&\big((1+c_1c_3)x_1+(1-c_1c_3)x_3\big)\big((1+c_2c_3)x_2+ (1-c_2c_3)x_3\big)z\notag\\
=&\frac{1}{4}s_1\Big((1+c_1c_3)(1+c_2c_3)x_1x_2z+
(1-c_2c_3)x_3\big((1+c_1c_2)x_1+(1-c_1c_2)x_2\big)z\Big)\notag\\
&+\frac{1}{4}s_1(1+c_1c_2+c_2c_3-c_1c_3)x_3^2z\notag\\
=&\frac{1}{2}s_1(1+c_1c_2+c_2c_3-c_1c_3)z \quad\text{ by
Lemma~\ref{lem:NP}}\notag.
\end{align}
This together with~(\ref{s1s2s1})~and~(\ref{s2s1s2}) shows that
\begin{align}
(s_1s_2s_1-s_2s_1s_2)z=&\frac{1}{2}((x_1+x_3)+c_1c_3(x_1-x_3))z\notag\\
&-\frac{1}{2}(c_1c_2+c_2c_3)x_2z-\frac{1}{2}(1-c_1c_3)x_3z\notag\\
=&\frac{1}{2}(1+c_1c_3)x_1z-\frac{1}{2}(c_1c_2+c_2c_3)x_2z\notag\\
=&\frac{1}{4}(1-c_1c_2-c_2c_3+c_1c_3)\Big((1+c_1c_2)x_1+(1-c_1c_2)x_2\Big)z\notag
\end{align}
which is zero by Lemma~\ref{lem:NP}.
\end{proof}
An identical argument used for proving Lemma~\ref{lem:n=3case1}
shows that $\mathfrak{H}^{\mathfrak{c}}_{2,1}$-module
$V(1,0,0):=L(1)\circledast V(0,0) $ affords an irreducible
$\mathfrak{H}_3^{\mathfrak{c}}$-module via $s_1=\Xi_1$.

\begin{prop}\label{prop:n=3}
Each irreducible $\mathfrak{H}_3^{\mathfrak{c}}$-module in
$\operatorname{Rep}_{\I}\mathfrak{H}_3^{\mathfrak{c}}$ on which
$x_1^2, x_2^2, x_3^2$ act semisimply is isomorphic to one of the
following.
\begin{enumerate}

\item A completely splittable
$\mathfrak{H}_3^{\mathfrak{c}}$-module $D^{\underline{i}}$ for
$\underline{i}\in W'(\mathfrak{H}_3^{\mathfrak{c}})$ (see
Theorem~\ref{thm:Classficiation}).

\item $V(0,0,1)$.

\item $V(1,0,0)$.


\item
${\rm
ind}^{\mathfrak{H}_3^{\mathfrak{c}}}_{\mathfrak{H}_{2,1}^{\mathfrak{c}}}V(0,0)\circledast
L(j)$ with $j\neq 0,1$.
\end{enumerate}
\end{prop}
\begin{proof}
We first show that listed pairwise non-isomorphic modules are
irreducible and all $x_k^2$ act semisimply. The case (1), (2) and
(3) are taken care of by Theorem~\ref{thm:Classficiation} and
Lemma~\ref{lem:n=3case1}.
 Using~\cite[Theorem 5.18]{BK}, we have
${\rm
ind}^{\mathfrak{H}_3^{\mathfrak{c}}}_{\mathfrak{H}_{2,1}^{\mathfrak{c}}}V(0,0)\circledast
L(j)$ is irreducible if $j\neq 0, 1$. It is known that as vector
spaces
$${\rm ind}^{\mathfrak{H}_3^{\mathfrak{c}}}_{\mathfrak{H}_{2,1}^{\mathfrak{c}}}V(0,0)\circledast L(j)
=V(0,0)\circledast L(j)\oplus s_2\otimes(V(0,0)\circledast
L(j))\oplus s_1s_2\otimes(V(0,0)\circledast L(j)).
$$
It is clear that for $z\in V(0,0)\circledast L(j)$,
\begin{align}
x_1^2z=0=x_2^2z,\quad x_3^2z=q(j)z.\label{xsq-1}
\end{align}
This together with~(\ref{px2}) implies $x_1^2=0$ on
$s_2\otimes(V(0,0)\circledast L(j))$. Using~(\ref{reln:x2s1})
and~(\ref{reln:x2s2}), we obtain that
\begin{align}
(x_2^2-q(j))\big(s_2\otimes(V(0,0)\circledast L(j))\big)&\subseteq
V(0,0)\circledast
L(j)\notag\\
x_3^2\big(s_2\otimes(V(0,0)\circledast L(j))\big)&\subseteq
V(0,0)\circledast L(j)\notag.
\end{align}
This together with~(\ref{xsq-1}) shows that for any $v\in
s_2\otimes(V(0,0)\circledast L(j))$,
\begin{align}
x_1^2v=0, ~ x_2^2(x_2^2-q(j))v=0,~
x_3^2(x_3^2-q(j))v=0.\label{xsq-2}
\end{align}
Similarly using~(\ref{reln:x2s1}),~(\ref{reln:x2s2})
and~(\ref{xsq-1}) we see that
\begin{align}
(x_1^2-q(j))s_1s_2\otimes(V(0,0)\circledast L(j))&\subseteq
V(0,0)\circledast
L(j)\oplus s_2\otimes (V(0,0)\circledast L(j))\notag\\
(x_3^2)s_1s_2\otimes(V(0,0)\circledast L(j))&\subseteq
V(0,0)\circledast L(j)\notag.
\end{align}
Therefore it follows from~(\ref{xsq-1})~and~(\ref{xsq-2}) that for
any $w\in s_1s_2\otimes(V(0,0)\circledast L(j))$
\begin{align}
x_1^2(x_1^2-q(j))w=0, ~x_3^2(x_3^2-q(j))w=0.\label{xsq-3}
\end{align}
By~(\ref{reln:x2s1})~and~(\ref{xsq-1}), we obtain that for any
$z\in V(0,0)\circledast L(j)$,
\begin{align}
x_2^2s_1s_2\otimes z&=s_1s_2\otimes x_1^2
z+\big((1+c_1c_2)x_1+x_2(1+c_1c_2)\big)s_2\otimes z\notag\\
&=\big((1+c_1c_2)x_1+x_2(1+c_1c_2)\big)s_2\otimes z.
\end{align}
This together with~(\ref{reln:x2s1}) and $x_3^2=q(j)$ on $
V(0,0)\circledast L(j)$ shows that for $z\in V(0,0)\circledast
L(j)$
\begin{align}
(x_2^2-q(j))x_2^2(s_1s_2\otimes
z)&=(x_2^2-q(j))\big((1+c_1c_2)x_1+x_2(1+c_1c_2)\big)s_2\otimes z\notag\\
&=\big((1+c_1c_2)x_1+x_2(1+c_1c_2)\big)(x_2^2-q(j))s_2\otimes z\notag\\
&=((1+c_1c_2)x_1+x_2(1+c_1c_2))(-x_2(1-c_2c_3)-(1-c_2c_3)x_3)z\notag\\
&=0\quad \text{ by Lemma~\ref{lem:NP}}.\notag
\end{align}
Therefore for any $w\in s_1s_2\otimes(V(0,0)\circledast L(j))$,
\begin{align}
(x_2^2-q(j))x_2^2w=0 .\label{xsq-4}
\end{align}
Combining~(\ref{xsq-1}), (\ref{xsq-2}), (\ref{xsq-3}) and
(\ref{xsq-4}), we see that the actions of $x_1, x_2, x_3$ on the
$\mathfrak{H}_3^{\mathfrak{c}}$-module ${\rm
ind}^{\mathfrak{H}_3^{\mathfrak{c}}}_{\mathfrak{H}_{2,1}^{\mathfrak{c}}}V(0,0)\circledast
L(j)$ satisfy
$$
x_1^2(x_1^2-q(j))=0,~ x_2^2(x_2^2-q(j))=0,~x_3^2(x_3^2-q(j))=0.
$$
It follows that $x_1^2, x_2^2, x_3^2$ act semisimply on ${\rm
ind}^{\mathfrak{H}_3^{\mathfrak{c}}}_{\mathfrak{H}_{2,1}^{\mathfrak{c}}}L(0^2)\circledast
L(j)$.

Now assume
$M\in\operatorname{Rep}_{\I}\mathfrak{H}_3^{\mathfrak{c}}$ is
irreducible, on which all $x_k^2, 1\leq k\leq n$ act semisimply. Let
us assume $M$ is not completely splittable, then by
Proposition~\ref{prop:equiv.cond.} $M$ has a weight of the form
$(i,i,j)$ or $(j,i,i)$ for some $i,j\in\I$. By
Corollary~\ref{cor:n=2} we obtain that $i=0$. Hence by Frobenius
reciprocity $M$ is a quotient of ${\rm
ind}^{\mathfrak{H}_3^{\mathfrak{c}}}_{\mathcal{P}_3^{\mathfrak{c}}}L(0)\circledast
L(0)\circledast L(j)$ or ${\rm
ind}^{\mathfrak{H}_3^{\mathfrak{c}}}_{\mathcal{P}_3^{\mathfrak{c}}}L(j)\circledast
L(0)\circledast L(0)$.

If $j=0$, then $M$ is isomorphic to the Kato module ${\rm
ind}^{\mathfrak{H}_3^{\mathfrak{c}}}_{\mathcal{P}_3^{\mathfrak{c}}}L(0)\circledast
L(0)\otimes L(0)$. By~\cite[Lemma 4.15]{BK}, all Jordan blocks of
$x_1$ on $M$ are of size $3$. This means $x_1^4=0$ on $M$ but not
$x_1^2$. Hence $x_1^2$ does not act semisimply on $M$.

If $j=1$, then the weights of $M$ belong to $S_3\cdot (0,0,1)$.
By~\cite[\S 5-d]{BK}, there are at most three non-isomorphic
irreducible $\mathfrak{H}_3^{\mathfrak{c}}$-modules whose weights
belong to the set $S_3\cdot(0,0,1)=\{(0,0,1), (0,1,0), (1,0,0)\}$.
By Theorem~\ref{thm:Classficiation}, the
$\mathcal{P}_3^{\mathfrak{c}}$-module $V(0,1,0)=L(0)\circledast
L(1)\circledast L(0)$ affords an irreducible completely splittable
$\mathfrak{H}_3^{\mathfrak{c}}$-module via $s_1=\Xi_1, s_2=\Xi_2$.
Observe that the modules $V(0,0,1), V(1,0,0)$ and $V(0,1,0)$ are
non-isomorphic and have weights belonging to $S_3\cdot(0,0,1)$.
Since $M$ is not completely splittable, $M\cong V(0,0,1)$ or $M\cong
V(1,0,0)$.

If $j\neq 0,1$, by~\cite[Theorem 5.18]{BK} we have that
\begin{align}
{\rm
ind}^{\mathfrak{H}_3^{\mathfrak{c}}}_{\mathcal{P}_3^{\mathfrak{c}}}L(0)\circledast
L(0)\otimes L(j)&\cong{\rm
ind}^{\mathfrak{H}_3^{\mathfrak{c}}}_{\mathfrak{H}_{2,1}^{\mathfrak{c}}}V(0,0)\circledast L(j)\notag\\
&\cong{\rm
ind}^{\mathfrak{H}_3^{\mathfrak{c}}}_{\mathfrak{H}_{2,1}^{\mathfrak{c}}}L(j)\circledast
V(0,0)\notag\\
&\cong{\rm
ind}^{\mathfrak{H}_3^{\mathfrak{c}}}_{\mathcal{P}_3^{\mathfrak{c}}}L(j)\circledast
L(0)\otimes L(0)\notag
\end{align}
is irreducible. Hence $M\cong{\rm
ind}^{\mathfrak{H}_3^{\mathfrak{c}}}_{\mathfrak{H}_{2,1}^{\mathfrak{c}}}V(0,0)\circledast
L(j)$.
\end{proof}
Observe that the subalgebra generated by $x_k, x_{k+1}, x_{k+2},
c_k, c_{k+1}, c_{k+2}, s_k, s_{k+1}$ is isomorphic to
$\mathfrak{H}_3^{\mathfrak{c}}$ for fixed $1\leq k\leq n-2$ . By
Proposition~\ref{prop:n=3}, we have the following.
\begin{cor}\label{cor:n=3} Suppose that $M\in\operatorname{Rep}_{\I}\mhcn$, on which all
$x_k^2, 1\leq k\leq n$ act semisimply. Let $\underline{i}\in\I^n$ be
a weight of $M$. Then there does not exist $1\leq k\leq n-2$ such
that $i_k=i_{k+1}=i_{k+2}$.
\end{cor}
\subsection{Conjecture for general $n$.}
\begin{prop}\label{x2semi}
Suppose that $M\in\operatorname{Rep}_{\I}\mhcn$ is irreducible and
$M_{\underline{i}}\neq 0$ for some $\underline{i}\in\I^n$. If
$x_k^2, 1\leq k\leq n,$ act semisimply on $M$, then $\underline{i}$
satisfies the following.
\begin{enumerate}

\item If $i_k\neq i_{k+1}\pm1$, then $s_k\cdot\underline{i}$ is a
weight of $M$.

\item If $i_k=i_{k+1}$ for some $1\leq k\leq n-1$, then
$i_k=i_{k+1}=0$.

\item There does not exist $1\leq k\leq n-2$ such that
$i_k=i_{k+1}=i_{k+2}$.

\item If $i_k=i_{k+2}$ for some $1\leq k\leq n-2$, then
\begin{enumerate}
\item If $p=0$, then $i_k=i_{k+2}=0$.

\item If $p\geq 3$, then either $(i_{k}=i_{k+2}=\frac{p-3}{2}$ and
$ i_{k+1}=\frac{p-1}{2})$ or $(i_k=i_{k+2}=0)$.

\end{enumerate}

\end{enumerate}
\end{prop}
\begin{proof}
(1)  If $i_k\neq i_{k+1}\pm1$, by Lemma~\ref{lem:newoperator}
$\widehat{\Phi}_k$ is a well-defined bijection from
$M_{\underline{i}}$ to $M_{s_k\cdot\underline{i}}$. Hence
$M_{s_k\cdot\underline{i}}\neq 0$.

(2)  It follows from Corollary~\ref{cor:n=2}.

(3) It follows from Corollary~\ref{cor:n=3}.

(4) Suppose $i_k=i_{k+2}=u$ and $i_{k+1}=v$ for some $1\leq k\leq
n-2$. Observe that for each fixed $1\leq k\leq n-2$, $x^2_k,
x^2_{k+1}, x^2_{k+2}$ act semisimply on the restriction of $M$ to
the subalgebra generated by $x_k, x_{k+1}, x_{k+2}, c_k, c_{k+1},
c_{k+2}, s_k, s_{k+1}$ which is isomorphic to
$\mathfrak{H}_3^{\mathfrak{c}}$. This implies that $(u,v,u)$ appears
as a weight of a $\mathfrak{H}_3^{\mathfrak{c}}$-module on which
$x_1^2, x_2^2, x_3^2$ act semisimply. By Proposition~\ref{prop:n=3},
if $p=0$, then $u=0$; if $p\geq 3$, then either $u=0, v $ is
arbitrary or $u=\frac{p-3}{2}, v=\frac{p-1}{2}$.
\end{proof}

\begin{cor} Suppose that $M\in\operatorname{Rep}_{\I}\mhcn$ is irreducible and
$M_{\underline{i}}\neq 0$ for some $\underline{i}\in\I^n$. If all
$x_k^2, 1\leq k\leq n$ act semisimply on $M$, then $\underline{i}$
satisfies the following.

\begin{enumerate}

\item If $p=0$ and $u=i_k=i_l\geq 1$ for some $1\leq k<l\leq n$,
then
\begin{align}
\{u-1,u+1\}&\subseteq\{i_{k+1},\ldots,i_{l-1}\},\notag\\
 &or\notag\\
(u,u-1,\ldots,1,0,0,1,\ldots,u-1,u)& \text{ is a subsequence of }
(i_{k+1},\ldots,i_{l-1}).\notag
\end{align}

\item If $p\geq3$ and $u=i_k=i_l\geq 1$ for some $1\leq k<l\leq
n$, then \begin{align}
\{u-1,u+1\}&\subseteq\{i_{k+1},\ldots,i_{l-1}\},\notag\\
 &or\notag\\
(u,u-1,\ldots,1,0,0,1,\ldots,u-1,u)& \text{ is a subsequence of }
(i_{k+1},\ldots,i_{l-1}),\notag\\
&or\notag\\
(u,u+1,\ldots,\frac{p-3}{2},\frac{p-1}{2},\frac{p-3}{2},\ldots,u+1,u)&
\text{ is a subsequence of } (i_{k+1},\ldots,i_{l-1}).\notag
\end{align}
\end{enumerate}
\end{cor}
\begin{proof}
(1) Without loss of generality, we can assume
$u\not\in\{i_{k+1},\ldots,i_{l-1}\}$. By the technique used in the
proof of Proposition~\ref{prop:restr.2}, one can show that
$u-1\in\{i_{k+1},\ldots, i_{l-1}\}$. Now assume
$u+1\notin\{i_{k+1},\ldots, i_{l-1}\}$. Then $u-1$ appears at
least twice between $i_{k+1}$ and $i_{l-1}$ in $\underline{i}$;
otherwise we can apply admissible transpositions to
$\underline{i}$ to obtain a weight of $M$ of the form $(\cdots, u,
u-1, u,\cdots)$ which contradicts Proposition~\ref{x2semi}(4).
Hence there exist $k<k_1<l_1<l$ such that
$$
i_{k_1}=u-1=i_{l_1}, \{u,
u-1\}\cap\{i_{k_1+1},\ldots,i_{l_1-1}\}=\emptyset.
$$
An identical argument shows that there exist $k_1<k_2<l_2<l_1$
such that
$$
i_{k_2}=u-2=i_{l_2}, \{u, u-1,
u-2\}\cap\{i_{k_2+1},\ldots,i_{l_2-1}\}=\emptyset.
$$
Continuing in this way, we achieve the claim.

(2) By the technique used in (1), one can easily show that if
$u+1\notin\{i_{k+1},\ldots, i_{l-1}\}$ then
$(i_{k+1},\ldots,i_{l-1})$ contains
$(u,u-1,\ldots,1,0,0,1,\ldots,u-1,u)$ as a subsequence. If
$u-1\notin\{i_{k+1},\ldots,i_{l-1}\}$, an identical argument used
in the proof of Proposition~\ref{prop:restr.2}(5) shows that
$(i_{k+1},\ldots,i_{l-1})$ contains
$(u,u+1,\ldots,\frac{p-3}{2},\frac{p-1}{2},\frac{p-3}{2},\ldots,u+1,u)$
as a subsequence.
\end{proof}

\begin{conjecture} Suppose that $M\in\operatorname{Rep}_{\I}\mhcn$ is irreducible. Then
$x_k^2, 1\leq k\leq n$, act semisimply on $M$ if and only if each
weight of $M$ satisfies the list of properties stated in
Proposition~\ref{x2semi}.
\end{conjecture}
\begin{thm}
The above conjecture holds for $n=2,3$.
\end{thm}
\begin{proof}
Clearly the above conjecture holds for $n=2$ by Lemma~\ref{lem:n=2}.
Suppose $M$ is an irreducible $\mathfrak{H}_3^{\mathfrak{c}}$-module
whose weights satisfy the list of properties stated in
Proposition~\ref{x2semi}.  Let $(i_1,i_2,i_3)\in\I^3$ be a weight of
$M$. Then by Frobenius reciprocity, $M$ is isomorphic to a quotient
of ${\rm
ind}^{\mathfrak{H}_3^{\mathfrak{c}}}_{\mathcal{P}_3^{\mathfrak{c}}}L(i_1)\circledast
L(i_2)\circledast L(i_3)$.  Hence the weights of $M$ are of the form
$\sigma\cdot(i_1,i_2,i_3)$ for $\sigma\in S_3\cdot$.  If
$i_1,i_2,i_3$ are distinct, then all weights $\underline{j}$ of $M$
satisfy $j_k\neq j_{k+1}$ for $k=1,2$. By
Proposition~\ref{prop:equiv.cond.}, $M$ is completely splittable and
hence all $x_1^2, x_2^2, x_3^2$ act semisimply on it.

Now assume $i_1,i_2,i_3$ are not distinct. If $p=0$, by the
properties in Proposition~\ref{x2semi} we have that $(i_1, i_2,
i_3)$ is of the form $(0,0,j)$, $(0,j,0)$ or $(j,0,0)$ for some
$j\geq 1$. By Proposition~\ref{prop:n=3}, all $x_k^2, 1\leq k\leq 3$
act semisimply on $M$. If $p\geq 3$, by the properties in
Proposition~\ref{x2semi} we see that either $(i_1, i_2,
i_3)=(\frac{p-3}{2},\frac{p-1}{2},\frac{p-3}{2})$ or $(i_1, i_2,
i_3)$ has the form $(0,0,j)$, $(0,j,0)$ or $(j,0,0)$ for some $j\geq
1$. In the latter case,  by Proposition~\ref{prop:n=3}, all $x_k^2,
1\leq k\leq 3$ act semisimply on $M$. Assume $(i_1, i_2,
i_3)=(\frac{p-3}{2},\frac{p-1}{2},\frac{p-3}{2})$. Since $M$
satisfies the properties in Proposition~\ref{x2semi},
$(\frac{p-3}{2},\frac{p-3}{2},\frac{p-1}{2})$ and
$(\frac{p-3}{2},\frac{p-3}{2},\frac{p-1}{2})$ are not the weights of
$M$. Hence $M$ has only one weight, that is,
$(\frac{p-3}{2},\frac{p-1}{2},\frac{p-3}{2})$. By
Proposition~\ref{prop:equiv.cond.}, $M$ is completely splittable and
hence all $x_1^2, x_2^2, x_3^2$ act semisimply on it.
\end{proof}
%



\begin{thebibliography}{ABC}

\bibitem[BK]{BK} J. Brundan and A. Kleshchev, {\em Hecke-Clifford superalgebras, crystals of type
$A^{(2)}_{2l}$, and modular branching rules for $\widehat{S}_n$},
Repr. Theory {\bf 5} (2001), 317--403.

\bibitem[C1]{C1} I. Cherednik, {\em Special bases of irreducible
representations of a degenerate affine Hecke algebra}, Funct. Anal.
Appl. {\bf 20} (1986), no.1, 76--78.

\bibitem[C2]{C2} I. Cherednik, {\em A new interpretation of Gel'fand-Tzetlin
bases}, Duke. Math. J. {\bf 54} (1987), 563--577.

\bibitem[D]{Dr} V. Drinfeld, {\em Degenerate affine Hecke algebras and
Yangians}, Funct. Anal. Appl. {\bf 20} (1986), no.1, 62--64.

\bibitem[HKS]{HKS} D. Hill, J. Kujawa, J. Sussan,
{\em Degenerate affine Hecke-Clifford algebras and type $Q$ Lie
superalgebras}, preprint, arXiv:0904.0499, 2009.

\bibitem[JN]{JN} A. Jones, M. Nazarov, Affine Sergeev algebra and $q$-analogues of
the Young symmetrizers for projective representations of the
symmetric group, Proc. London Math. Soc. {\bf 78} (1999), 481--512.

\bibitem[K1]{K1} A. Kleshchev, {\em Completely splittable representations
of symmetric groups}, J. Algebra {\bf 181} (1996), 584--592.


\bibitem[K2]{K2} A. Kleshchev, {\em Linear and Projective
Representations of Symmetric Groups}, Cambridge University Press,
2005.

\bibitem[KR]{KR} A. Kleshchev, A. Ram, {\em  Homogeneous representations of Khovanov-Lauda
algebras}, preprint, arXiv:0809.0557, 2008.

\bibitem[Le]{Le} B. Leclerc, {\em Dual canonical bases, quantum
shuffles and q-characters}, Math. Z. {\bf 246} (2004), no. 4,
691--732.

\bibitem[Lu]{Lus} G. Lusztig, {\em Affine Hecke algebras and their
graded version}, J. Amer. Math. Soc. {\bf 2} (1989), 599--635.

\bibitem[M]{M} O. Mathieu, {\em On the dimension of some modular irreducible representations of
the symmetric group}, Lett. Math. Phys. {\bf 38} (1996), 23--32.


\bibitem[N1]{Na1} M. Nazarov, {\em Young's orthogonal form of
irreducible projective representations of the symmetric group},
 J. London Math. Soc. (2){\bf 42} (1990), no. 3, 437--451.


\bibitem[N2]{Na2} M. Nazarov, {\em Young's symmetrizers for projective representations
of the symmetric group}, Adv. Math. {\bf 127} (1997), no. 2,
190--257.

\bibitem[OV]{OV} A. Okounkov and A. Vershik, {\em A new approach to
representation theory of symmetric groups}, Selecta Math. (N.S)
{\bf 2} (1996), 581--605.

\bibitem[Ra]{Ra} A. Ram, {\em Skew shape representations are
irreducible}, (English summary) Combinatorial and geometric
representation theory (Seoul, 2001), 161--189, Contemp. Math., 325,
Amer. Math. Soc., Providence, RI, 2003.




\bibitem[Ru]{Ru} O. Ruff, {\em Completely splittable representations
of symmetric groups and affine Hecke aglebras}, J. Algebra {\bf
305} (2006), 1197--1211.



\bibitem[W]{W} W. Wang, {\em Double affine Hecke-Clifford algebras for
the spin symmetric group}, preprint, math.RT/0608074, 2006.

\end{thebibliography}
\end{document}